\documentclass[11pt]{amsart}
\usepackage{verbatim, latexsym, amssymb, amsmath,color}
\usepackage{hyperref} 
\usepackage{epsfig}
\usepackage{bm}
\usepackage{mathrsfs}
\usepackage[margin=1.25in]{geometry}

\newtheorem{theorem}{Theorem}[section]
\newtheorem{corollary}[theorem]{Corollary}
\newtheorem{claim}[]{Claim}
\newtheorem{lemma}[theorem]{Lemma}

\newtheorem{proposition}[theorem]{Proposition}

\theoremstyle{definition}
\newtheorem{definition}[theorem]{Definition}

\newtheorem{question}[theorem]{Question}

\newtheorem{remark}[theorem]{Remark}

\newtheorem*{acknowledgements}{Acknowledgements}

\numberwithin{equation}{section}

\newcommand{\al}{\alpha}

\newcommand{\V}{\mathcal{V}}
\newcommand{\RV}{\mathcal{RV}}
\newcommand{\IV}{\mathcal{IV}}
\newcommand{\R}{\mathbb{R}}
\newcommand{\N}{\mathbb{N}}
\newcommand{\mH}{\mathcal{H}}
\newcommand{\F}{\mathcal{F}}

\newcommand{\A}{\mathcal{A}}
\newcommand{\B}{\mathcal{B}}
\newcommand{\C}{\mathcal{C}}

\newcommand{\weakto}{\rightharpoonup}

\newcommand{\lc}{\llcorner}

\newcommand{\si}{\sigma}

\newcommand{\de}{\delta}

\newcommand{\ep}{\epsilon}
\newcommand{\pr}{\prime}

\newcommand{\Ga}{\Gamma}
\newcommand{\ga}{\gamma}

\newcommand{\mZ}{\mathbb{Z}}
\newcommand{\Z}{\mathcal{Z}}

\newcommand{\M}{\mathbf{M}}

\newcommand{\bL}{\mathbf{L}}

\newcommand{\bI}{\mathbf{I}}
\newcommand{\mf}{\mathbf{f}}

\newcommand{\mF}{\mathbf{F}}
\newcommand{\mI}{\mathbf{I}}
\newcommand{\mR}{\mathcal{R}}
\newcommand{\bmu}{\boldsymbol \mu}
\newcommand{\btau}{\boldsymbol \tau}
\newcommand{\bleta}{\boldsymbol \eta}

\newcommand{\ttimes}{\scalebox{0.7}{$\mathbb{X}$}}

\newcommand{\tM}{\widetilde{M}}

\newcommand{\tBcal}{\widetilde{\mathcal{B}}}
\newcommand{\tAcal}{\widetilde{\mathcal{A}}}
\newcommand{\tScal}{\widetilde{\mathcal{S}}}

\newcommand{\Area}{\textrm{Area}}

\newcommand{\VarTan}{\operatorname{VarTan}}

\newcommand{\spt}{\operatorname{spt}}
\newcommand{\dist}{\operatorname{dist}}
\newcommand{\Div}{\operatorname{div}}

\newcommand{\interior}{\operatorname{int}}
\newcommand{\Ric}{\operatorname{Ric}}
\newcommand{\Clos}{\operatorname{Clos}}

\newcommand{\rom}[1]{\expandafter\romannumeral #1}

\title[Min-max Theory for Free Boundary G-Invariant Minimal Hypersurfaces]{Min-max Theory for Free Boundary G-Invariant Minimal Hypersurfaces}
\author{Tongrui Wang}
\address{Institute for Theoretical Sciences, Westlake Institute for Advanced Study, Westlake University, Hangzhou, Zhejiang, 310024, China}
\email{wangtongrui@westlake.edu.cn}

\begin{document}
\maketitle
\begin{abstract}
	Given a compact Riemannian manifold $M^{n+1}$ with dimension $3\leq n+1\leq 7$ and $\partial M\neq\emptyset$, the free boundary min-max theory built by Li-Zhou \cite{li2021min} shows the existence of a smooth almost properly embedded minimal hypersurface with free boundary in $M$. 
  In this paper, we generalize their constructions into equivariant settings. 
  Specifically, let $G$ be a compact Lie group acting as isometries on $M$ with cohomogeneity at least $3$. 
 Then there exists a nontrivial smooth almost properly embedded free boundary $G$-invariant minimal hypersurface. 
 Moreover, if the Ricci curvature of $M$ is non-negative and $\partial M$ is strictly convex, then there exist infinitely many properly embedded $G$-invariant minimal hypersurfaces with free boundary. 
\end{abstract}

\section{Introduction}
Min-max theory was first proposed by Almgren \cite{almgren1962homotopy}\cite{almgren1965theory} in 1960s to construct minimal submanifolds in compact Riemannian manifolds $(M^{n+1},g_{_M})$. 
For the case that $\partial M=\emptyset$ (namely the {\em closed} case), weak solutions for minimal submanifolds in any dimension (i.e. stationary integer rectifiable varifolds) were constructed by Almgren in \cite[Theorem 15.1]{almgren1965theory}. 
After this, by proposing the {\em almost minimizing varifolds}, Pitts \cite{pitts2014existence} obtained smooth minimal hypersurfaces in closed manifolds with dimension $3\leq n+1\leq 6$. 
The regularity results in higher dimensional closed manifolds ($n+1\geq 7$) were established by Schoen-Simon in \cite{schoen1981regularity}. 
In recent years, the famous Almgren's isomorphism theorem (\cite[Theorem 7.5]{almgren1962homotopy}) extends the definition of sweepout from curves to mappings defined on high-dimensional parameter spaces. 
Hereafter, many longstanding conjectures have been solved in closed manifolds by the multi-parameter min-max theory (for instance, \cite{marques2014min}\cite{marques2017existence}).

For the case that $\partial M\neq\emptyset$, a {\em constrained free boundary problem} was proposed by Almgren and solved by Li-Zhou in \cite{li2021min} using the min-max theory with the relative cycle space. 
Namely, if $\partial M\neq\emptyset$ and $3\leq n+1\leq 7$, there exists a smooth embedded hypersurface $\Sigma\subset M$ so that $\Sigma$ is minimal in its interior and $\Sigma$ meets $\partial M$ orthogonally on $\partial \Sigma\subset\partial M$ (\cite[Theorem 1.1]{li2021min}). 
Similarly, many conclusions in the closed case can be extended to the free boundary case (\cite{guang2021min}\cite{wang2020existence}).

The main purpose of this paper is to establish an equivariant generalization of the min-max theory built by Li-Zhou in \cite{li2021min}. 
Specifically, let $M^{n+1}$ be a compact Riemannian manifold with boundary $\partial M$, and $G$ be a compact Lie group acting as isometries on $M$ with $\dim(G\cdot p)\leq n-2$, $\forall p\in M$. 
We have the following {\em equivariant constrained free boundary problem}:
\begin{question}\label{Q: equi free boundary problem}
	Is there a smooth, embedded, $G$-invariant minimal hypersurface $\Sigma$ contained in $M$ with boundary $\partial\Sigma\subset\partial M$ solving the free boundary problem?
\end{question}
Here, $\Sigma$ is $G$-invariant in the sense that $g\cdot \Sigma=\Sigma$ for all $g\in G$. 
To get such $G$-invariant hypersurfaces, we can use only $G$-invariant relative $n$-cycles in the min-max constructions. 
Hence, this kind of theory is known as the {\em equivariant min-max theory}. 
In the 1980s, Pitts-Rubinstein \cite{pitts1988equivariant}\cite{pitts1987applications} firstly proposed an equivariant min-max theory for finite $G$ acting on $3$-dimensional closed manifolds without a full proof. 
This construction was recently completed by Ketover in \cite{ketover2016equivariant} provided that the action of the finite group $G$ is orientation preserving and $M/G$ is an orientable orbifold without boundary. 
Since then, the equivariant min-max theory was extended by Liu in \cite{liu2021existence} to any compact connected Lie group $G$ acting as isometries on a closed manifold $M^{n+1}$ with $3\leq n+1\leq 7$ and ${\rm Cohom}(G)\neq 0 ,2$. 
Both Ketover's and Liu's equivariant constructions are under the smooth setting of min-max (c.f. \cite{de2013existence}) using generalized hypersurfaces instead of currents. 
Therefore, the author proposed in \cite{wang2022min} an equivariant min-max construction under the setting of Almgren-Pitts, where $M^{n+1}$ is assumed to be a closed manifold with dimension $3\leq n+1\leq 7$, $G$ is assumed to be a compact Lie group with ${\rm Cohom}(G)\geq 3$, and $M\setminus M^{reg}$ is assumed to be a submanifold of $M$ with dimension no more than $n-2$.  Here $M^{reg}$ is the union of all principal orbits, and ${\rm Cohom}(G)$ is the cohomogeneity of the action of $G$.

Note these equivariant results are built in the closed case $\partial M=\emptyset$ with some additional assumptions on the Lie group $G$ actions. 
For the case that $\partial M\neq\emptyset$, Ketover \cite{ketover2016free} first showed an equivariant min-max theory to construct new free boundary minimal surfaces in $\mathbb{B}^3\subset \R^3$ with dihedral symmetries, which can be generalized to any three-dimensional compact manifold with strictly mean convex boundary and a finite group $G$ acting by orientation-preserving isometries (see \cite{franz2021equivariant}). 
As before, these constructions are based on the smooth setting of min-max making some topological properties controllable. 
However, these existing results only solved Question \ref{Q: equi free boundary problem} in the three-dimensional case, and need additional assumptions on $G$ and $\partial M$. 
Hence, we set out this paper to establish a free boundary equivariant min-max theory under the setting of Almgren-Pitts, which not only gives a positive answer to Question \ref{Q: equi free boundary problem} in complete generality, but also provides more follow-up applications. 
Specifically, our main result is the following theorem:
 
\begin{theorem}\label{Thm:main}
	Let $M^{n+1}$ be a smooth connected compact Riemannian manifold with a compact Lie group $G$ acting as isometries of cohomogeneity ${\rm Cohom}(G)\geq 3$.
	Suppose $\partial M\neq\emptyset$ and $2\leq n\leq 6$. 
	Then there exists a hypersurface $\Sigma^n$ in $M$ with possibly empty boundary $\partial \Sigma\subset\partial M$ so that 
	\begin{itemize}
		\item[(i)] $\Sigma$ is invariant under the action of $G$, i.e. $g\cdot\Sigma=\Sigma$ for all $g\in G$,
		\item[(ii)] $\Sigma$ is smoothly embedded in $M$ up to its boundary $\partial\Sigma$,
		\item[(iii)] $\Sigma$ is minimal in its interior,
		\item[(iv)] $\Sigma$ meets $\partial M$ orthogonally along its boundary $\partial \Sigma$.
	\end{itemize}
\end{theorem}

The compactness and the cohomogeneity assumptions on $G$ are very mild and clean. 
In particular, if $G=\{e\}$ is trivial, all the assumptions are satisfied. 
Thus, our result can be seen as a generalization of \cite[Theorem 1.1]{li2021min}. 
If $n+1=3$ and ${\rm Cohom}(G)= 3$, then $G$ is finite by the compactness assumption.  
Hence, Theorem \ref{Thm:main} also generalizes the results in \cite{ketover2016free} to a certain extent. 
Moreover, if $\partial M=\emptyset$, Theorem \ref{Thm:main} provides a {\em closed} embedded minimal $G$-invariant hypersurface. 
Therefore, our theorem not only extends the results in \cite{ketover2016equivariant}\cite{liu2021existence}\cite{wang2022min} into the free boundary case but also weakens the requirements for Lie groups. 

We mention that the dimension assumption $2 \leq n \leq 6$ is for the same reason as in \cite[Remark 1.2]{li2021min}. 
Additionally, the assumption on ${\rm Cohom}(G)\geq 3$ is used to guarantee $\dim(G\cdot p)\leq n-2$ for every $p\in M$. 
Thus, a small tube around $G\cdot p$ in $M$ will be mean convex to apply the maximum principle (\cite[Theorem 2.5]{li2021min}), and the standard extension arguments (c.f. the proof of \cite[Theorem 4.1]{harvey1975extending}) can be applied to each orbit. 

Moreover, when $n\geq 7$, we remark that the same results in Theorem \ref{Thm:main} shall be valid as long as $3\leq {\rm codim}(G\cdot p)\leq 7 $ for all $p\in M$. 
To make this paper more concise, we leave the details in another upcoming paper. 

In \cite{wang2022min}, the author used an assumption on $M\setminus M^{reg}$ to overcome a technical obstacle. 
To be exact, one should notice that the injectivity radii of orbits generally do not have a positive lower bound. 
This drawback presents difficulties whenever the normal exponential maps $\exp_{G\cdot p}^\perp$ are used to carry out a uniform result. 
Nevertheless, if every orbit in the compact $G$-set $\Clos(U)$ has the same orbit type (\cite[Section 2.1.4]{berndt2016submanifolds}), then there exists a uniform injectivity radius of orbits in $\Clos(U)$ (Remark \ref{Rem:lower bound of Fermi radius}). 
Therefore, a technical assumption is added in \cite{wang2022min} that $M\setminus M^{reg}$ is a submanifold with dimension no more than $n-2$, which helps to regard $M\setminus M^{reg}$ as a whole. 
Hence, noting the closure of the small $G$-annulus around $M\setminus M^{reg}$ is contained in $M^{reg}$, the interpolation arguments in \cite[Lemma 3.5-3.9]{pitts2014existence} would carry over with $\exp_{G\cdot p}^\perp$ in place of $\exp_p$ (see \cite[Appendix B]{wang2022min}). 
Inspired by the work of Li-Zhou\cite[Appendix B]{li2021min}, we can now overcome this technical problem without any further assumptions on $M\setminus M^{reg}$. 
Roughly speaking, by the orbit type stratification, we could regard every orbit type stratum $M_i$ as a `boundary'. 
Similar to \cite[Lemma B.4-B.7]{li2021min}, our constructions only take $\exp_{G\cdot p}^\perp$ with $G\cdot p\subset M_i$ into consideration. 
Hence, the estimates of $\exp_{G\cdot p}^\perp $ can be made uniformly for $p$ in a compact $G$-subset $\Clos(U)\cap M_i$.

As an application, we can use the upper bounds of $(G,p)$-width in \cite[Theorem 10]{wang2022min} (which also holds for the equivalent classes of relative cycles) and the arguments in \cite{marques2017existence} to have the following corollaries as equivariant generalizations of \cite[Corollary 1.3, 1.4]{li2021min}. 

\begin{corollary}\label{Cor:1}
	Under the same assumptions of Theorem \ref{Thm:main}, denote $l={\rm Cohom}(G)$. 
	Then either
	\begin{itemize}
		\item[(i)] there exists a disjoint collection $\{\Sigma_1,\cdots,\Sigma_{l}\}$ of $l$ compact, smoothly embedded, $G$-connected, free boundary $G$-invariant  minimal hypersurfaces in $M$; or
		\item[(ii)] there exist infinitely many compact, smoothly embedded, $G$-connected, free boundary $G$-invariant minimal hypersurfaces in $M$.
	\end{itemize}
\end{corollary}

\begin{corollary}\label{Cor:2}
	Under the same assumptions of Theorem \ref{Thm:main}, if $\Ric_M \geq 0$ and the boundary $\partial M$ is strictly convex, then $M$ contains an infinite number of distinct compact, smooth, properly embedded, free boundary $G$-invariant minimal hypersurfaces.
\end{corollary}

\subsection{Outline}
    The structure of this paper is parallel to \cite{li2021min}. 
	In Section \ref{Sec:preliminary}, we set some basic notations and useful propositions. 
	We introduce the almost properly equivariantly embedded hypersurfaces and their variation formulas. 
	We also show the equivalence between the stability and $G$-stability for almost properly equivariantly embedded hypersurfaces with $G$-invariant normal vector fields, which indicates a compactness theorem for such hypersurfaces by \cite[Theorem 2.13]{li2021min}. 
	In Section \ref{Sec:geometry measure}, the definition of {\em $(G,\mZ_2)$-almost minimizing varifolds} will be introduced using {\em relative $(G,\mZ_2)$-cycles}. 
	Since relative $(G,\mZ_2)$-cycles form a closed subspace of the relative cycles space, all the lemmas in \cite[Section 3.1]{li2021min} remain true under equivariant settings. 
	In particular, we prove two isoperimetric lemmas for the equivalent classes of relative $(G,\mZ_2)$-cycles (Lemma \ref{Lem:F-isoperimetric}, \ref{Lem:M-isoperimetric}). 
	Additionally, we also show the equivalence of $(G,\mZ_2)$-almost minimizing varifolds defined by different metrics (Theorem \ref{Thm: equivalence-a.m.v}). 
	In Section \ref{Sec:min-max}, we introduce the min-max construction of Li-Zhou for free boundary settings and adapt their theory into our $G$-equivariant case. 
	As in \cite[Theorem 4.12, 4.14]{li2021min}, we show the discretization and interpolation Theorem \ref{Thm: discretization}, \ref{Thm:interpolation} for our $G$-equivariant free boundary min-max construction. 
	After that, the combinatorial arguments of Pitts \cite[Theorem 4.10]{pitts2014existence} are applied to show the existence of $(G,\mZ_2)$-almost minimizing varifolds with free boundary (Theorem \ref{Thm:exist amv in critical set}). 
	Finally, the regularity of such varifolds is given in Section \ref{Sec:regular}. 

\begin{acknowledgements}
	The author would like to thank Professor Gang Tian for his constant encouragement. 
	He also thanks Dr. Zhiang Wu for helpful discussions, and thanks the anonymous referees for helpful comments. 
	The author is partially supported by China Postdoctoral Science Foundation 2022M722844.
\end{acknowledgements}

\section{Preliminary}\label{Sec:preliminary}

In this paper, we always assume $(M^{n+1}, g_{_M})$ to be an compact connected Riemannian $(n+1)$-manifold with boundary $\partial M$, 
and assume $G$ to be a compact Lie group acting as isometries on $M$ of cohomogeneity ${\rm Cohom}(G)=l\geq 3$. 
Let $\mu$ be a bi-invariant Haar measure on $G$ which has been normalized to $\mu(G)=1$. 
Without loss of generality, we can also assume $M$ is orientable. 
Indeed, after taking the orientable double cover $\hat{M}:=\{(p,o_p): p\in M, o_p~\mbox{is an orientation of $T_pM$} \}$ of $M$, both the Riemannian structure and the isometric actions of $G$ can be lifted to $\hat{M}$ with $g\cdot(p, o_p) := (g\cdot p, g_*o_p)$. 
Since the involution map $\tau$ in the covering space $\hat{M}$ is a $G$-equivariant isometry, i.e. $\tau(p,o_p)=(p,-o_p)$ and $g\circ \tau = \tau\circ g$, we can firstly prove the main results for $(\hat{M}, G\times \mZ_2)$ with $\mZ_2=\{id, \tau\}$, and then quotient the $\mZ_2$-actions down to $M$. 

Now, let us collect some basic definitions and notations about the action of Lie groups and refer \cite{berndt2016submanifolds}\cite{bredon1972introduction}\cite{wall2016differential} for more details. 

Firstly, we claim $M$ can be extended to a closed Riemannian manifold $\tM$ with the same dimension such that 
\begin{itemize}
	\item[(1)] $G$ acts on $\tM$ as isometries with the same cohomogeneity;
	\item[(2)] $M\subset \tM$, and the inclusion map is a $G$-equivariant isometric embedding.
\end{itemize}
We left the proof in the Appendix \ref{Sec: extension}, which is a combination of the equivariant gluing results \cite{kankaanrinta2007equivariant} and the Riemannian extension \cite{pigola2016smooth}. 

Furthermore, $\tM$, as well as $M$, can be $G$-equivariantly isometrically embedded into some $\R^L$ by \cite{moore1980equivariant}. 
To be specific, there is an orthogonal representation $\rho:G\to O(L)$ and an isometric embedding $i$ from $\tM$ to $\R^L$, which is $G$-equivariant, i.e. $i(g\cdot x) = \rho(g)\cdot i(x) $. 
For simplicity, we always denote the acting of $g$ on $p$ by $g\cdot p:=\rho(g)(p)$ for every $p\in\R^L,g\in G$. 

Given any $p\in M$, let $G_p := \{g\in G: g\cdot p = p\}$ be the {\em isotropy group} of $p$ in $G$, and $(G_p)$ be the conjugate class of $G_p$ in $G$. 
Then we say $p\in M$ has the {\em orbit type} $(G_p)$. 
Denote $M_{(G_p)}:= \{q\in M: (G_q)=(G_p)\}$ to be the union of points with $(G_p)$ orbit type, which is a disjoint union of smooth embedded submanifolds of $M$(\cite[Chapter 6, Corollary 2.5]{bredon1972introduction}). 

For any $G$-invariant subset $U\subset M$ with connected components $\{U_i\}_{i=1}^I$, we say $U$ is {\em $G$-connected} if for any $i,j\in\{1,\dots,I\}$, there is $g_{i,j}\in G$ with $g_{i,j}\cdot U_j=U_i$. 
Given any orbit type stratum $M_{(G_p)}$, we can decompose $M_{(G_p)}$ into some $G$-connected components. 
Therefore, we can stratify $M$ by 
\begin{equation}\label{Eq: stratification of M}
	M=\cup_{i=1}^m M_i = \cup_{i=1}^m\cup_{j=1}^{m_i} M_i^j,
\end{equation}
where each $M_i$ is a stratum of an orbit type, and $\{M_i^j\}_{j=1}^{m_i}$ are the $G$-connected components of $M_i$. 
For simplicity, we denote 
\begin{equation}\label{Eq: stratification of M and dimension}
	n_{i,j} := \dim(M_i^j) \leq n+1,\quad l_{i,j} := \dim(M_i^j) - \dim(G\cdot p),
\end{equation}
where $p\in M_i^j$, $i=1,\dots, m, j=1,\dots, m_i$. 
Moreover, there exists a minimal conjugate class of isotropy groups $(P)$ such that $M_{(P)}$ forms an open dense submanifold of $M$, which is known as the {\em principal orbit type}. 
The {\em cohomogeneity} ${\rm Cohom}(G)= l$ of $G$ is defined as the codimension of a principal orbit. 
Clearly, $n_{i,j} = n+1$ if and only if $p\in M_i^j$ has the  principal orbit type. 
Since $G$ also acts by isometries on $N := \partial M$, there is also a stratification of $\partial M$ by the orbit types and their $G$-connected components:
\begin{equation}\label{Eq: stratification of boundary}
	N:= \partial M = \cup_{i=1}^{m'} N_i = \cup_{i=1}^{m'}\cup_{j=1}^{m_i'} N_i^j.
\end{equation}
Similarly, define 
\begin{equation}\label{Eq: stratification of boundary and dimension}
	n_{i,j}' :=\dim(N_i^j),\quad l_{i,j}' := \dim(N_i^j) - \dim(G\cdot p), ~{\rm for}~ p\in N_i^j. 
\end{equation}

\begin{definition}\label{Def: isolated orbit}
	Given any orbit $G\cdot p \subset M_i^j\subset M$, we say $G\cdot p$ is an {\em isolated orbit in $M$} if $l_{i,j}=0$. 
	Similarly, we say $G\cdot p\subset N_i^j\subset \partial M$ is {\em isolated in $\partial M$} if $l_{i,j}'=0$. 
	In general, we say an orbit $G\cdot p$ is {\em isolated} if it is isolated in $M$ or $\partial M$. 
\end{definition}
\begin{remark}\label{Rem: no isolated orbit}
	It follows from the Slice Theorem(c.f. \cite[Theorem 3.3.4]{wall2016differential}) that if $G\cdot p$ is isolated in $M$, then there is an open neighborhood $U\subset M$ of $G\cdot p$ so that $M_{(G_p)}\cap U = G\cdot p$. 
\end{remark}

 \subsection{Basic notations}\label{notation} 

For the action of $G$, we use the following notations:
\begin{itemize}
	\item $\pi$ : the projection $\pi:M\mapsto M/G$  given by $p \mapsto [p]:=\{g\cdot p: g\in G\}$;
	\item $T_qG\cdot p$ : the tangent space of the orbit $G\cdot p$ at $q\in G\cdot p$;
	\item ${\bf N}_qG\cdot p$ : the normal vector space of the orbit $G\cdot p$ in $M$ at  $q\in G\cdot p$;
	\item ${\bf N}(G\cdot p)$ : the normal vector bundle of the orbit $G\cdot p$ in $M$;
	\item $M^{reg}$ : the union of orbits in $M$ with the principal orbit type (as orbits in $M$);
	\item $(\partial M)^{reg}$ : the union of orbits in $\partial M$ with the principal orbit type (as orbits in $\partial M$);
\end{itemize}
Sometimes, we also add a superscript like ${\bf N}_p^{\R^L}G\cdot p$, ${\bf N}_p^{\tM}G\cdot p$, to indicate the normal space is taken in $T_p\R^L$ or $T_p\tM$. 
We borrow the following notations in $\R^L$ from \cite{li2021min} and add a superscript $G$ to signify the $G$-invariance. 
\begin{itemize}
	\item $B_r(p),~B^G_r(p)$ : Euclidean open ball and tube of radius $r$ centered at $p$ and $G\cdot p$;
	\item $A_{s,r}(p),~A^G_{s,r}(p)$ : Euclidean open annulus  $B_r(p) \setminus \Clos(B_s(p))$ and open $G$-annulus $B^G_r(p) \setminus \Clos(B^G_s(p))$;
	\item $\bmu_r$ : the homothety map  $x \mapsto r \cdot x$;
	\item $\btau_p$ : the translation map $x \mapsto x-p $;
	\item $\bleta_{p,r}$ : the composition $\bmu_{r^{-1}} \circ \btau_p $;
	\item $\Clos(A)$ : the closure of a subset $A \subset \mathbb{R}^L$;
	\item $\mathcal{H}^k$ : the $k$-dimensional Hausdorff measure in $\mathbb{R}^L$;
	\item $\omega_k$ : the volume of the $k$-dimensional unit ball $ B_1(0) \subset \R^k$;
	\item $\mathfrak{X}(\mathbb{R}^L)$ : the space of smooth vector fields on $\mathbb{R}^L$. 
\end{itemize}

The following geodesic notations are also borrowed from \cite{li2021min} and a superscript $G$ will be added to signify the $G$-invariance (see Appendix \ref{Sec:Fermi half-tubes} and \cite[Appendix A]{li2021min} for the notations involving Fermi coordinates):
\begin{itemize}
	\item $\nu_{\partial M}$ : the inward unit normal of $\partial M$ with respect to $M$;
	\item $\tBcal_r(p), ~\tBcal^G_r(p)$ : the open geodesic ball and tube in $\tM$ of radius $r$ around $p$ and $G\cdot p$;
	\item $\tScal_r(p), ~\tScal^G_r(p)$ : the geodesic sphere and cylinder in $\tM$ of radius $r$ around $p$ and $G\cdot p$;
	\item $\tAcal_{s,r}(p)$ : the open geodesic annulus $\tBcal_r(p) \setminus \Clos(\tBcal_s(p))$ in $\tM$;
	\item $\tAcal_{s,r}^G(p)$ : the open geodesic $G$-annulus $\tBcal^G_r(p) \setminus \Clos(\tBcal^G_s(p))$ in $\tM$;
	\item $\tBcal^+_r(p)$ : the (relative) open Fermi half-ball of radius $r$ centered at $p \in \partial M$;
	\item $\tBcal^{G,+}_r(p)$ : the (relative) open Fermi half-tube of radius $r$ centered at $G\cdot p \subset \partial M$;
	\item $\tScal^+_r(p)$ : the Fermi half-sphere of radius $r$ centered at $p \in \partial M$;
	\item $\tScal^{G,+}_r(p)$ : the Fermi half-cylinder of radius $r$ centered at $G\cdot p \subset \partial M$;
	\item $\interior_M(A)$ : the interior of $A\subset M$ with respect to the subspace topology of $M$;
	\item $\partial_{rel} A$ : the relative boundary of $A\subset M$ given by $\Clos(A)\setminus \interior_M(A)$.
\end{itemize}

Sometimes, we add $G$- in front of objects meaning they are $G$-invariant:
  \begin{itemize}
    \item a $G$-varifold $V$ satisfies $g_{\#} V=V$ for all $g\in G$;
    \item a $G$-current $T$ satisfies $g_{\#} T=T$ for all $g\in G$;
    \item a $G$-vector field $X$ satisfies $g_{*} X=X$ for all $g\in G$;
    \item a $G$-set ($G$-neighborhood) is a (relative open) set which is a union of orbits.
  \end{itemize}

As in \cite{li2021min}, we define the following space of vector fields
\begin{eqnarray}
	\mathfrak{X}(M) &:=& \{X\in\mathfrak{X}(\R^L) :  X(p)\in T_p M,~\forall p\in M\}\nonumber;
	\\
	\mathfrak{X}^G(M) &:=& \{X\in\mathfrak{X}(M) :  g_*X=X,~\forall g\in G\};
	\\
	\mathfrak{X}_{tan}(M) &:=& \{X\in\mathfrak{X}(M) : X(p)\in T_p \partial M,~\forall p\in \partial M\};\nonumber
	\\
	\mathfrak{X}_{tan}^G(M) &:=& \mathfrak{X}^G(M)\cap \mathfrak{X}_{tan}(M).\nonumber
\end{eqnarray}
Next, we introduce some varifolds spaces (see \cite[2.1(18)]{pitts2014existence}\cite{simon1983lectures}). 
Denote $\RV_k(M)$ (resp. $\IV_k(M)$) as the space of (integer) rectifiable $k$-varifolds in $\R^L$ supported in $M$. 
Moreover, let $\RV^G_k(M)$ (resp. $\IV^G_k(M)$) be the set of $G$-varifolds in $\RV_k(M)$ (resp. $\IV_k(M)$). 
Then the closure of $\RV_k(M)$ in the weak topology is denoted by $\mathcal{V}_k(M)$, and $\V^G_k(M)$ is the set of $G$-varifolds in $\V_k(M)$. 
We also denote by $\|V\|$ and $\VarTan(V,p)$ the {\em weight measure} and the {\em varifold tangents} of $V\in \V_k(M)$ (\cite[$\S$38, 42.3]{simon1983lectures}). 
The $\mF$-metric on $\V_k(M)$ is defined in \cite[2.1(19)]{pitts2014existence}, which induces the weak topology on any mass bounded subset of $\V_k(M)$

For any $V\in \V_k(M)$ and $X\in \mathfrak{X}_{tan}(M)$, the first variation of $V$ along $X$ is given by 
\begin{equation}
	\delta V(X) := \frac{d}{dt}\Big|_{t=0}\|(\phi_t)_\# V\|(M) = \int \Div_SX(p) ~dV(p,S),
\end{equation}
where $\phi_t$ are the diffeomorphisms generated by $X$. 
Thus, we have the following definition:

\begin{definition}[$G$-stationary varifolds with free boundary]\label{Def: G-stationary}
	Let $U \subset M$ be a relative open $G$-subset. 
	A varifold $V\in\V_k^G(M)$ is said to be {\it $G$-stationary in $U$ with free boundary} if $\delta V(X) =0$ for all $X \in \mathfrak{X}^G_{tan}(M)$ with compact support in $U$. 
\end{definition}

As in the closed case (see \cite{liu2021existence}), we also have the equivalence between $G$-stationary with free boundary and stationary with free boundary \cite[Definition 2.1]{li2021min}. 
\begin{lemma}\label{Lem: G-stationary with free boundary and stationary}
	Let $V\in\V_k^G(M)$ and $U\subset M$ be a relative open $G$-subset. 
	Then $V$ is $G$-stationary in $U$ with free boundary if and only if $V$ is stationary in $U$ with free boundary. 
\end{lemma}
\begin{proof}
	Since $\mathfrak{X}^G_{tan}(M)\subset \mathfrak{X}_{tan}(M)$, the `if' part is obvious. 
	Suppose $V$ is $G$-stationary in $U$ with free boundary. 
	Then for any $X\in \mathfrak{X}_{tan}(M)$ with $\spt(X)\subset U$, define $X_g:=(g^{-1})_* X$ and $X_G:=\int_G X_g~d\mu(g)$. 
	One verifies that $X_g \in \mathfrak{X}_{tan}(M)$, since $U$ is a $G$-invariant set and $G$ acts on $M$ as isometries. 
	Thus, $X_G\in \mathfrak{X}_{tan}(M)$ and
	\begin{eqnarray*}
		g_*X_G = g_*\int_G X_h~d\mu(h) = \int_G g_*h^{-1}_*X~d\mu(h) = \int_G X_{h\circ g^{-1}} ~d\mu(h) = X_G,
	\end{eqnarray*}
	where the third equality comes from the bi-invariance of $\mu$. 
	This suggests $X_G\in \mathfrak{X}_{tan}^G(M)$. 
	As in \cite[Lemma 3.2]{liu2021existence}, we further have $\delta V(X)=\delta V(X_G)$, which gives the lemma. 
\end{proof}

By the above lemma, we can apply the monotonicity formula \cite[Theorem 2.3]{li2021min} and the maximum principle \cite[Theorem 2.5]{li2021min} to any varifolds that are $G$-stationary with free boundary.

\subsection{Almost proper equivariant embeddings and $G$-stability}
We now modify the definitions given in \cite[Section 2.3]{li2021min} into $G$-invariant forms. 
A smooth manifold is said to be a {\em $G$-manifold} if $G$ acts on it by diffeomorphisms. 
\begin{definition}[Almost proper equivariant embeddings]
	Let $\Sigma^n$ be a smooth $n$-dimensional $G$-manifold with boundary $\partial \Sigma$ (possibly empty). 
	A smooth $G$-equivariant embedding $\phi:\Sigma \to \tM$ is said to be an {\it almost proper equivariant embedding of $\Sigma$ into $M$} if 
\[ \phi(\Sigma) \subset M \qquad \text{and} \qquad \phi(\partial \Sigma) \subset \partial M,\] 
which is written as $\phi:(\Sigma,\partial \Sigma) \stackrel{G}{\to} (M,\partial M)$. 
For simplicity, we often take $\phi$ as the inclusion map $\iota$, and write $(\Sigma,\partial \Sigma) \stackrel{G}{\subset} (M,\partial M)$. 

Given an almost properly equivariantly embedded hypersurface $(\Sigma,\partial \Sigma) \stackrel{G}{\subset} (M,\partial M)$, we say that $p \in \Sigma \cap \partial M$ is 
\begin{itemize}
\item a {\it true boundary point} if $p \in \partial \Sigma$; or
\item a {\it false boundary point} otherwise.
\end{itemize}
If every $p \in \Sigma \cap \partial M$ is a true boundary point, i.e. $\Sigma \cap \partial M=\partial \Sigma$, we say that $(\Sigma,\partial \Sigma) \stackrel{G}{\subset} (M,\partial M)$ is {\it properly equivariantly embedded}.
\end{definition}

\begin{remark}
	$(\Sigma,\partial \Sigma)$ is almost properly equivariantly embedded in $(M,\partial M)$ if and only if it is $G$-invariant and almost properly embedded in the sense of \cite[Definition 2.6]{li2021min}, i.e. 
	$$ (\Sigma,\partial \Sigma) \stackrel{G}{\subset} (M,\partial M) \iff (\Sigma,\partial \Sigma) \subset (M,\partial M) ~{\rm and}~\Sigma ~{\rm is}~ \mbox{$G$-invariant}. $$
	Thus, we often omit the term `equivariant' and say $(\Sigma,\partial \Sigma) \stackrel{G}{\subset} (M,\partial M)$ is an {\it almost properly embedded $G$-hypersurface} for simplicity. 
\end{remark}

\begin{definition}\label{Def:minimal hypersurface with free boundary}
	An almost properly (equivariantly) embedded hypersurface $(\Sigma,\partial\Sigma)\subset (M,\partial M)$ is said to be {\it a free boundary ($G$-invariant) minimal hypersurface} if the mean curvature of $\Sigma$ vanishes identically and $\Sigma$ meets $\partial M$ orthogonally along $\partial \Sigma$. 
\end{definition}

For an almost properly embedded $G$-hypersurface $(\Sigma,\partial \Sigma) \stackrel{G}{\subset} (M,\partial M)$, we define the following spaces of vector fields 
\begin{eqnarray}
	\mathfrak{X}(M,\Sigma) &:=& \left\{ X \in \mathfrak{X}(M): \begin{array}{c}
		X(q) \in T_q(\partial M) \text{ for all $q$ in an open }\\
		\text{neighborhood of $\partial \Sigma$ in $\partial M$.} \end{array} \right\},
	\\
	\mathfrak{X}^G(M,\Sigma) &:=& \mathfrak{X}(M,\Sigma) \cap \mathfrak{X}^G(M). 
\end{eqnarray}
The first variation of $\Sigma$ along any compactly supported $X\in \mathfrak{X}(M,\Sigma) $ is given by 
\begin{equation}\label{Eq: first variation}
	\delta \Sigma (X) := \frac{d}{dt}\Big|_{t=0} \Area(\phi_t(\Sigma)) =  -\int_\Sigma \langle H,X\rangle +\int_{\partial \Sigma} \langle \eta ,X\rangle,
\end{equation}
where $\phi_t$ are the diffeomorphisms generated by $X$, $H$ is the mean curvature vector of $\Sigma$, and $\eta$ is the unit co-normal of $\partial \Sigma$ pointing outward. 

\begin{definition}\label{Def: G-stationary for hypersurface}
	An almost properly embedded $G$-hypersurface $(\Sigma,\partial\Sigma)\stackrel{G}{\subset}(M,\partial M)$ is said to be {\it $G$-stationary} if $\delta \Sigma(X)=0$ for any $X\in\mathfrak{X}^G(M,\Sigma) $ with compact support. 
	Moreover, if $\delta \Sigma(X)=0$ for any $X\in\mathfrak{X}(M,\Sigma) $, then $(\Sigma,\partial\Sigma)$ is said to be {\em stationary}. 
\end{definition}
\begin{remark}
	Given $(\Sigma,\partial\Sigma)\stackrel{G}{\subset}(M,\partial M)$ and $X\in \mathfrak{X}^G(M,\Sigma)$, we do not require $X(q)\in T_q\partial M$ for any false boundary point $q\in (\Sigma\setminus\partial\Sigma)\cap\partial M$ in the above definition. 
	Thus, a $G$-stationary varifold $V$ with free boundary and smooth support may not be $G$-stationary in the sense of Definition \ref{Def: G-stationary for hypersurface}. 
	However, an almost properly embedded $G$-stationary hypersurface $(\Sigma,\partial\Sigma)\stackrel{G}{\subset}(M,\partial M)$ must be $G$-stationary with free boundary in the sense of Definition \ref{Def: G-stationary}. 
\end{remark}

As in Lemma \ref{Lem: G-stationary with free boundary and stationary}, we also have the equivalence between $G$-stationary and stationary for almost properly embedded $G$-hypersurfaces. 
\begin{lemma}\label{Lem: $G$-stationary for hypersurface and stationary}
	Let  $(\Sigma,\partial\Sigma)\stackrel{G}{\subset}(M,\partial M)$ be an almost properly embedded $G$-hypersurface. 
	Then $(\Sigma,\partial \Sigma)$ is $G$-stationary if and only if $(\Sigma,\partial \Sigma)$ is stationary. 
\end{lemma}
\begin{proof}
	Since $\mathfrak{X}^G(M,\Sigma)\subset \mathfrak{X}(M,\Sigma)$, we only need to show the `only if' part.  
	Suppose $(\Sigma,\partial\Sigma)\stackrel{G}{\subset}(M,\partial M)$ is $G$-stationary. 
	Then for any $X\in \mathfrak{X}(M,\Sigma)$ with compact support, we define $X_g:=(g^{-1})_* X$ and $X_G:=\int_G X_g~d\mu(g)$. 
	One verifies that $X_g \in \mathfrak{X}(M,\Sigma)$, which indicates $X_G\in \mathfrak{X}^G(M,\Sigma)$. 
	Additionally, sine the vector fields $H$ and $\eta$ in (\ref{Eq: first variation}) are both $G$-invariant, a direct compute shows that $ 0=\delta \Sigma (X_G) = \delta\Sigma(X)$, which gives the lemma. 
\end{proof}

Let $(\Sigma,\partial\Sigma)\stackrel{G}{\subset}(M,\partial M)$ be an almost properly embedded $G$-hypersurface.
By the first variation formula (\ref{Eq: first variation}) and Lemma \ref{Lem: $G$-stationary for hypersurface and stationary}, it is clear that $(\Sigma,\partial \Sigma)$ is $G$-stationary if and only if it is a {\em free boundary $G$-invariant minimal hypersurface} in the sense of Definition \ref{Def:minimal hypersurface with free boundary}. 

Suppose $(\Sigma,\partial\Sigma)\stackrel{G}{\subset}(M,\partial M)$ is a free boundary $G$-invariant minimal hypersurface with a {\em $G$-invariant} unit normal vector field $\nu$. 
Let $X=f\nu$ be a $G$-invariant normal vector field which extends to a smooth $G$-vector field in $\mathfrak{X}^G(M,\Sigma) $. 
The second variation of $\Sigma$ along $X$ is then given by 
\begin{equation}\label{Eq: second variation}
	\delta^2\Sigma(X) := Q_\Sigma(f, f) =\int_\Sigma |\nabla^\Sigma f|^2-(\Ric_M(\nu,\nu)+|A^\Sigma|^2)f^2~da - \int_{\partial\Sigma}h(\nu,\nu)f^2~ds,
\end{equation}
where $A^\Sigma$ and $h$ are the second fundamental forms of $\Sigma$ and $\partial M$ with respect to $\nu$ and $\nu_{\partial M}$. 

\begin{definition}\label{Def: G-stability}
	An almost properly embedded $G$-hypersurface $(\Sigma,\partial\Sigma)\stackrel{G}{\subset}(M,\partial M)$ is said to be {\it $G$-stable} if 
	\begin{itemize}
		\item $\Sigma$ admits a $G$-invariant unit normal vector field $\nu$,
		\item $\Sigma$ is stationary in the sense of Definition \ref{Def: G-stationary for hypersurface},
		\item $\delta^2\Sigma(X)\geq 0$ for all $X= f\nu \in\mathfrak{X}^G(M,\Sigma)$ with compact support.
	\end{itemize}
\end{definition}

Note the existence of a $G$-invariant unit normal vector field $\nu$ is a priori requirement of our definition for the $G$-stability. 
We mention that a two-sided $G$-invariant hypersurface may not admit such a $G$-vector field when $G$ is not connected. 
However, if $\Sigma$ is a part of the smooth boundary of some relative open $G$-subset $U$, then $\Sigma$ has a $G$-invariant unit normal vector field pointing inside (or outside) of $U$. 

\begin{lemma}\label{Lem: G-stable and stable}
	Let $(\Sigma,\partial\Sigma)\stackrel{G}{\subset}(M,\partial M)$ be an almost properly embedded $G$-hypersurface. 
	Then $(\Sigma,\partial\Sigma)$ is $G$-stable if and only if it has a $G$-invariant unit normal vector field $\nu$ and is stable in the sense of \cite[Definition 2.10]{li2021min}. 
	
	In particular, if $(\Sigma,\partial\Sigma)\stackrel{G}{\subset}(M,\partial M)$ is a part of the boundary of some relative open $G$-subset $U\subset M$, then the $G$-stability of $(\Sigma,\partial\Sigma)$ is equivalent to the stability. 
\end{lemma}
\begin{proof}
	The if part is clear due to the fact that $\mathfrak{X}^G(M,\Sigma)\subset \mathfrak{X}(M,\Sigma)$. 
	Next, suppose $(\Sigma,\partial\Sigma)\stackrel{G}{\subset}(M,\partial M)$ is $G$-stable. 
	Let $X=f\nu\in \mathfrak{X}(M,\Sigma)$ be the first (Robin) eigenvector field of the Jacobi operator $L_\Sigma$ of $\Sigma$ so that $L_\Sigma f=-\lambda_1f$ on $\Sigma$, and $\partial f/\partial\nu_{\partial M}= h(\nu,\nu) f$ on $\partial \Sigma$. 
	Define $X_G :=\int_G (g^{-1})_* X ~d\mu(g)$ and $f_G(p) := \int_G f(g\cdot p)~d\mu(g)$. 
	Since $\nu$ is $G$-invariant, one verifies that $X_G = f_G\cdot\nu$ on $\Sigma$, and $X_G\in \mathfrak{X}^G(M,\Sigma)$. 
	Additionally, since the first (Robin) eigenfunction $f$ is positive, 
	we have $f_G\neq 0$. 
	A direct computation shows $Lf_G=-\lambda_1f_G$ (\cite[Lemma 7]{wang2022min}) and $\partial f_G/\partial  \nu_{\partial M} = h(\nu,\nu) f_G$ (since $\nu_{\partial M} $ and $\nu$ are $G$-invariant), which implies 
	$0 \leq \delta^2\Sigma (X_G) = \lambda_1\int_\Sigma f_G^2$. 
	Hence, $\lambda_1\geq 0$ and $\Sigma$ is stable. 
\end{proof}

By the above lemma, we can apply the Compactness Theorem \cite[Theorem 2.13]{li2021min} for $G$-stable almost properly embedded $G$-hypersurfaces with a uniform area bound. 
Note the limit of two-sided hypersurfaces coming from the compactness theorem may be globally one-sided hypersurfaces, which are not stable (or $G$-stable) due to the prior requirement of the ($G$-invariant) unit normal in the definitions. 
Thus, the compactness theorem \cite[Theorem 2.13]{li2021min} is localized in simply connected relative open subsets, which guarantees that every embedded hypersurface is two-sided (c.f. \cite{samelson1969orientability}). 
Therefore, we make the following localized definition and then state our compactness theorem. 

\begin{definition}
	Let $U \subset M$ be a relative open $G$-subset. 
	An almost properly embedded $G$-hypersurface $(\Sigma,\partial\Sigma)\stackrel{G}{\subset}(U,U \cap \partial M)$ is said to be {\em $G$-stationary in $U$} if $\delta \Sigma(X)=0$ for any $X \in \mathfrak{X}^G(M,\Sigma)$ compactly supported in $U$. 
	Moreover, if $(\Sigma,\partial\Sigma)$ has a $G$-invariant unit normal vector field on $\Sigma\cap U$ and $\delta^2 \Sigma (X) \geq 0$ for any $X \in \mathfrak{X}^G(M,\Sigma)$ compactly supported in $U$, then $(\Sigma,\partial\Sigma)$ is {\em $G$-stable in $U$}. 
\end{definition}

\begin{lemma}\label{Lem: G-stable and stable in U}
	Let $U \subset M$ be a relative open $G$-subset, $(\Sigma,\partial\Sigma)\stackrel{G}{\subset}(U,U \cap \partial M)$ be an almost properly embedded $G$-hypersurface. 
	Then $\Sigma$ is $G$-stationary in $U$ if and only if it is stationary in $U$. 
	Moreover, $\Sigma$ is $G$-stable in $U$ if and only if it has a $G$-invariant unit normal vector field and is stable in $U$ in the sense of \cite[Definition 2.12]{li2021min}. 
\end{lemma}
\begin{proof}
	This first statement follows from the arguments in Lemma \ref{Lem: $G$-stationary for hypersurface and stationary}. 
	As for the second statement, the if part is obvious. 
	Suppose $\Sigma$ is $G$-stable in $U$ and is unstable in $U$. 
	Then we can find a compact $G$-invariant domain $\Omega\subset \Sigma\cap U$ so that $\Gamma_1 = \Omega\cap \partial \Sigma$ and $\Gamma_2 = \Clos(\partial\Omega\setminus \Gamma_1)$ are smooth $G$-invariant $(n-1)$-submanifolds, $\Gamma_3 = \Gamma_1\cap \Gamma_2$ is a smooth $G$-invariant $(n-2)$-submanifold, and $\lambda_1(\Omega)<0$, where
	$$ \lambda_1(\Omega) := \inf \{ Q_\Sigma(f, f):  f\in W^{1,2}(\Omega), ~\|f\|_{L^2(\Omega)}=1, ~\mbox{$f=0$ on $\interior (\Gamma_2)$ in the sense of traces} \}. $$
	(Note that if $\Gamma_1 = \emptyset$ or $\Gamma_3 = \emptyset$, then the proof in \cite[Lemma 7]{wang2022min} and Lemma \ref{Lem: G-stable and stable} would carry over to show $\Sigma$ is stable in $U$.) 
	By a standard argument, there exists $f\in W^{1,2}(\Omega)$ so that $f$ is smooth except at the corner $\Gamma_3$,  
	\begin{equation}\label{Eq: first eigenvalue}
		\mbox{$f > 0$ in $\interior(\Omega), \quad $ and } \quad  \left\{\begin{array}{ll} L_{\Sigma}(f) = -\lambda_1(\Omega) f > 0 & \text { in } \quad \interior(\Omega), \\ \frac{\partial f}{\partial \nu_{\partial M}} = h(\nu,\nu) f& \text { on } \quad \interior(\Gamma_1), \\ f=0 & \text { on } \quad \interior(\Gamma_2 ) ,\end{array}\right.
	\end{equation}
	where $L_\Sigma f = \triangle_\Sigma f+ (|A^\Sigma|^2 + \Ric_M(\nu, \nu)  )f$. 
	Since $G$ acts by isometries, for all $g\in G$, we have $f\circ g$ as well as $f_G(p):= \int_G f\circ g(p) d\mu(g)$ also satisfies (\ref{Eq: first eigenvalue}) and is smooth except at $\Gamma_3$. 
	Since $\Gamma_3$ is $G$-invariant and $\dim(\Gamma_3 )\leq n-2$, we can apply the standard logarithmic cut-off trick to $f_G$ on $\Gamma_3$, 
	and see $\Sigma$ is $G$-unstable in $\Omega\subset U$, which is a contradiction. 
\end{proof}

\begin{theorem}[Compactness theorem for $G$-stable $G$-hypersurfaces]\label{Thm:compactness $G$-stable hypersurface} 
Let $2 \leq n \leq 6$ and $U$ be a relative open $G$-subset of $M$. 
Suppose $(\Sigma_k,\partial \Sigma_k) \stackrel{G}{\subset} (U,U \cap \partial M)$, $k\in\N$, is a sequence of almost properly embedded free boundary minimal $G$-hypersurfaces that are $G$-stable in $U$ and $\sup_{k} {\rm Area}(\Sigma_k) < \infty$. 

Then after passing to a subsequence, $(\Sigma_k,\partial \Sigma_k)$ converges (possibly with multiplicity) to an almost properly embedded free boundary minimal $G$-hypersurface $(\Sigma_\infty,\partial \Sigma_\infty) \stackrel{G}{\subset} (U,U \cap \partial M)$, which is also stable in any simply connected relative open subset of $U$. 
Moreover, the convergence is uniform and smooth on simply connected compact subsets of $U$.
\end{theorem}
\begin{proof}
	By Lemma \ref{Lem: G-stable and stable in U}, the $G$-stability of $\Sigma_k$ implies the stability. 
	Fix any simply connected relative open subset $U'$ of $U$.
	By \cite[Theorem 2.13]{li2021min}, after passing to a subsequence, $(\Sigma_k,\partial \Sigma_k)$ converges in $U'$ to some almost properly embedded free boundary minimal hypersurface $(\Sigma_\infty,\partial \Sigma_\infty) \subset (U',U' \cap \partial M)$ which is also stable in $U'$. 
	Moreover, the convergence is uniform and smooth on compact subsets of $U'$. 
	Since $U'$ is arbitrary, we have an almost properly embedded free boundary minimal hypersurface $(\Sigma_\infty,\partial \Sigma_\infty) \subset (U,U \cap \partial M)$. 
	Noting the varifold limit of $G$-varifolds is also a $G$-varifold, we have $(\Sigma_\infty,\partial \Sigma_\infty)$ is a $G$-invariant hypersurface by the varifold convergence $\Sigma_k\to \Sigma_\infty$. 
	Thus, $(\Sigma_\infty,\partial \Sigma_\infty) \stackrel{G}{\subset} (U,U \cap \partial M)$. 
\end{proof}


\section{$(G,\mZ_2)$-almost minimizing varifolds with free boundary}\label{Sec:geometry measure}

\subsection{Equivalence classes of relative $G$-cycles}\label{Sec:relative G-cycle} 
In this subsection, we suppose $K$ is a compact subset in $M$ and $U\subset M$ is a relative open $G$-subset containing $K$. 
We also denote $(A,B) :=(U,U\cap\partial M)$ for simplicity. 
Note both $A$ and $B$ are $G$-sets. 

Denote $\mR_k(M;\mZ_2)$ ($\mR_k^G(M;\mZ_2)$) as the space of integer rectifiable $\mZ_2$-coefficients ($G$-invariant) $k$-currents in $\R^L$ supported in $M$ (see \cite{federer2014geometric}). 
Let $\M$ and $\F^K$ be the mass norm and the flat semi-norm (relative to $K\subset M$) in $\mR_k(M;\mZ_2)$. 
Denote by $|T|$ and $\|T\|$ the integer rectifiable varifold and the Radon measure induced by $T\in \mR_k(M;\mZ_2)$. 
Next, the spaces of {\em relative $\mZ_2$-cycles} are defined as:
\begin{eqnarray}
	&Z_k(A,B;\mZ_2)& := \{ T\in \mR_k(M;\mZ_2) : \spt(T)\subset A,~\spt(\partial T)\subset B   \},
	\\
	&Z_k^G(A,B;\mZ_2)& := \{ T\in \mR_k^G(M;\mZ_2) :  \spt(T)\subset A,~\spt(\partial T)\subset B   \},
	\\
	&Z_k(B,B;\mZ_2)& := \{ T\in \mR_k(M;\mZ_2) :  \spt(T)\subset B \}.
\end{eqnarray}
As in \cite[Definition 3.1]{li2021min}, the space of {\it equivalence classes of relative $\mZ_2$-cycles} is defined as the quotient group
$$ \Z_k(A, B;\mZ_2):=Z_k(A, B;\mZ_2)/Z_k(B, B;\mZ_2).$$
Given $\tau \in \Z_k(A, B;\mZ_2)$, the {\it support} of $\tau$ is defined as
$$ \spt(\tau):=\bigcap_{T \in \tau} \spt(T).$$
Moreover, the (relative) {\it mass norm} and the {\it flat (semi)-norm} (with respect to the compact subset $K \subset A$) is defined in $\Z_k(A, B;\mZ_2)$ by
$$ \M(\tau):=\inf_{T \in \tau} \M(T), \qquad \F^K(\tau):=\inf_{T \in \tau} \F^K(T), \qquad \forall\tau\in\Z_k(A, B;\mZ_2).$$
We say a sequence $\tau_i\in \Z_k(A, B;\mZ_2)$ {\em weakly converges} to $\tau_\infty\in \Z_k(A, B;\mZ_2)$ and write $\tau_i\rightharpoonup \tau_\infty$, if $\F^K(\tau_i-\tau_\infty)\to 0$ for every compact $K\subset A$. 

If $T\in Z_k(A,B;\mZ_2)$ and $T' := T\llcorner(A\setminus B)=T-T\llcorner B$, then $[T]=[T']\in \Z_k(A, B;\mZ_2)$. 
This element $T'\in Z_k(A, B;\mZ_2)$ is known as the {\it canonical representative of $\tau=[T]\in \Z_k(A, B;\mZ_2)$}, which is the unique element in $\tau$ with a vanishing restriction on $B$ (see \cite[Definition 3.2]{li2021min}). 

For any $\tau,\sigma\in \Z_k(A, B;\mZ_2)$, let $T,S\in Z_k(A, B;\mZ_2)$ be the canonical representative of $\tau$ and $\sigma$ respectively. 
Then the {\it $\mF$-distance between $\tau$ and $\sigma$} is defined as 
$$ \mF(\tau,\sigma):= \F^M(\tau-\sigma)+\mF(|T|,|S|), $$ 
where $|T|,|S|$ are the rectifiable varifolds induced by $T$ and $S$ (see \cite[Definition 3.12]{li2021min}). 
For any $g\in G$ and $\tau\in \Z_k(A, B;\mZ_2)$, because $G$ acts by isometries on $M$, we can define $g_{\#}\tau := [g_{\#}T] = [g_{\#} T']\in \Z_k(A, B;\mZ_2)$, where $T$ is the canonical representative of $\tau$, $T'\in \tau$. 

\begin{definition}\label{Def:relative G-cycle}
	The space of {\it equivalence classes of relative $(G,\mZ_2)$-cycles} is defined as a subspace of $\Z_k(A, B;\mZ_2)$ by:
	$$ \Z_k^G(A, B;\mZ_2) := \{\tau \in  \Z_k(A, B;\mZ_2) : g_{\#}\tau = \tau ,~\forall g\in G\} . $$
	Moreover, the definitions of $\M,~\F^K,~\spt,~\mF,$ and the canonical representative are naturally induced from $\Z_k(A, B;\mZ_2)$ to $\Z_k^G(A, B;\mZ_2)$. 
\end{definition}

\begin{remark}
	By definitions, $\tau\in  \Z_k^G(A, B;\mZ_2)$ if and only if the canonical representative $T$ of $\tau$ is $G$-invariant, i.e. $T\in Z_k^G(A,B;\mZ_2)$. 
	Thus, $\Z_k^G(A, B;\mZ_2) = Z_k^G(A, B;\mZ_2)/Z_k(B, B;\mZ_2)$. 
\end{remark}

We now generalize the lemmas in \cite[Section 3.1]{li2021min} to the equivalence classes of relative $(G,\mZ_2)$-cycles. 
Since $G$ acts by isometries, $\Z_k^G(A, B;\mZ_2)$ is a closed subspace of $\Z_k(A, B;\mZ_2)$. 
Thus, we have the following results as simple generalizations of \cite[Lemma 3.3, 3.5, 3.7]{li2021min}. 
We leave the proof to readers. 

\begin{lemma}\label{Lem:canonical representative}
	Given $\tau\in\Z_k^G(A, B;\mZ_2)$, let $T\in\tau$ be the canonical representative of $\tau$. 
	Then $T\in Z_k^G(A, B;\mZ_2) $ and $\M(\tau)=\M(T)$, $\spt(\tau)=\spt(T)$. 
\end{lemma}

\begin{lemma}\label{Lem:weak convergence}
Given a sequence $\{\tau_i\}\subset \Z_k^G(A,B;\mZ_2)$ and $\tau_\infty \in\Z_k^G(A,B;\mZ_2)$, denote by $T_i \in \tau_i$ and $T_\infty \in \tau_\infty$ the canonical representatives. 
If $ \tau_i \weakto \tau_\infty$, i.e. for every compact $K \subset A$, $\F^K(\tau_i -\tau_\infty) \to 0$ as $i \to \infty$. 
Then $T_i$ converges weakly to $T_\infty$ in the open subset $A \setminus B$. 
\end{lemma}

\begin{lemma}[Lower semi-continuity of $\M$]\label{Lem:lower-semicontinue of M}
Given $\tau_i \weakto \tau_\infty$ in $\Z_k^G(A,B;\mZ_2)$, then 
$$ \M(\tau_{\infty})\leq\liminf_{i\rightarrow\infty}\M(\tau_i).$$
If $\spt(\tau_i)\subset K$, $i\in\N$, for some compact $K \subset A$, then $\spt(\tau_{\infty})\subset K$.
\end{lemma}

Denote by ${\bf I}_k(A;\mZ_2)\subset Z_k(A,A;\mZ_2)$ the space of integral $\mZ_2$-coefficients $k$-currents supported in $A$, i.e. $\partial T\in \mR_{k-1}(A;\mZ_2) $ if $T\in {\bf I}_k(A;\mZ_2)$. 
Let $${\bf I}_k^G(A;\mZ_2) := \{T\in {\bf I}_k(A;\mZ_2) : g_\#T=T, \forall g\in G \},$$ which is a {\em closed} subspace of $ {\bf I}_k(A;\mZ_2) $ since $G$ acts by isometries. 
Noting the sets $A$ and $B$ are both $G$-invariant, the distance function $\dist_A(\cdot, B)$ is clearly a $G$-invariant function. 
Hence, the slice of a $G$-current $T\in Z_k^G(A, B;\mZ_2) $ by $\dist_A(\cdot, B)$ is also a $G$-current. 
Additionally, let $\pi'$ be the nearest point projection map onto $B$ defined on a small $G$-neighborhood of $B$ in $A$. 
Then $\pi'$ is $G$-equivariant. 
Therefore, the slicing trick and the cutoff trick in the proof of \cite[Lemma 3.8, 3.10, 3.13]{li2021min} can be applied in our equivariant settings, which gives the following three lemmas: 

\begin{lemma}\label{Lem:mass compute among integral currents}
Given $\tau\in\Z_k^G(A,B;\mZ_2)$, there exists a sequence $T_i\in \tau$ such that 
\begin{itemize}
	\item $T_i\in \bI^G_k(A;\mZ_2)$ for every $i$,
	\item $\lim_{i\to\infty}\M(T_i)=\M(\tau)$.
\end{itemize}
\end{lemma}

\begin{lemma}[Compactness theorem for relative $(G,\mZ_2)$-cycles]\label{Lem:compactness for cycles}
Given $C>0$ and compact Lipschitz neighborhood retracts $K,K' \subset A$ (see \cite{almgren1962homotopy}) such that $K \subset \interior_M(K')$, the set
$$ \Z_{k, K, C}^G(A, B;\mZ_2):=\{\tau\in\Z_{k}^G(A, B;\mZ_2) : \spt(\tau)\subset K,\ \M(\tau)\leq C\}$$
is (sequentially) compact in the $\F^{K^{\pr}}$ topology. 
\end{lemma}

\begin{lemma}\label{Lem:F-metric}
Let $\tau, \tau_i \in \Z_k^G(M, \partial M;\mZ_2)$, $i\in\N$. 
Then $\mF(\tau_i,\tau) \to 0$ if and only if $\F(\tau_i-\tau) \to 0$ and $\M(\tau_i) \to \M(\tau)$.
\end{lemma}

\medskip
We now show the isoperimetric lemmas for equivalent classes of relative $(G,\mZ_2)$-cycles. 
Since $\Z_k^G(A, B;\mZ_2)$ is a subspace of $\Z_k(A, B;\mZ_2)$, we can apply the isoperimetric lemmas \cite[Lemma 3.15, 3.17]{li2021min} directly. 
The main difficulty here is to show the isoperimetric choice is still $G$-invariant.
In what follows, we denote $\F=\F^M$ for simplicity.

\begin{lemma}[$G$-invariant $\F$-isoperimetric lemma]\label{Lem:F-isoperimetric}
	There exist $\ep_M>0$ and $C_M \geq 1$ depending only on the isometric embedding $M \hookrightarrow \R^L$, such that for any $\tau_1, \tau_2\in \Z_n^G(M, \partial M;\mZ_2)$ with $ \F(\tau_2-\tau_1) < \ep_M$,
	there exists $Q\in\mI_{n+1}^G(M;\mZ_2)$, called \emph{the $\F$-isoperimetric choice} of $\tau_1,\tau_2$, satisfying
	\begin{itemize}
		\item $\spt(T_2-T_1-\partial Q)\subset\partial M$, 
		\item $\M(Q)\leq C_M \F(\tau_2-\tau_1)$, 
	\end{itemize}
	where $T_1,T_2$ are the canonical representatives of $\tau_1,\tau_2$ respectively. 
\end{lemma}

\begin{proof}
	Let $\epsilon_1>0$ and $C_M\geq 1$ be the constants in \cite[Lemma 3.15]{li2021min}. 
	Take $\epsilon_M\in (0, \epsilon_1)$ sufficiently small so that $2C_M \epsilon_M<\M(M)$. 
	Then, for any $\tau_1, \tau_2\in \Z_n^G(M, \partial M;\mZ_2)$ with $ \F(\tau_2-\tau_1) < \ep_M$, there exists $Q\in\mI_{n+1}(M;\mZ_2)$ so that $\spt(T_2-T_1-\partial Q)\subset\partial M$ and $\M(Q)\leq C_M \F(\tau_2-\tau_1)$ by \cite[Lemma 3.15]{li2021min}. 
	We claim this $Q$ is unique. 
	Indeed, if there is another $P\in \mI_{n+1}(M;\mZ_2)$ satisfies $\spt(T_2-T_1-\partial P)\subset\partial M$ and $\M(P)\leq C_M \F(\tau_2-\tau_1)$. 
	Then we have $\spt(\partial(Q-P))\subset \partial M$, which implies $Q-P = kM$ for some $k\in\{0,1\}$ by the Constancy Theorem \cite[Theorem 26.27]{simon1983lectures} and $\mZ_2$-coefficients. 
	Additionally, 
	$$\M(P-Q) \leq \M(P) + \M(Q) \leq 2C_M\epsilon_M<\M(M).$$
	Thus, we have $P-Q=0$, i.e. the $\F$-isoperimetric choice in \cite[Lemma 3.15]{li2021min} is unique with such $\epsilon_M$. 
	Furthermore, since $G$ acts by isometries and $T_1, T_2$ are $G$-invariant, it is clear that $g_\# Q$ is also an $\F$-isoperimetric choice of $\tau_1$ and $\tau_2$ in the sense of \cite[Lemma 3.15]{li2021min} for every $g\in G$. 
	The uniqueness suggests $Q = g_\# Q$, i.e. $Q\in  \mI_{n+1}^G(M;\mZ_2)$. 
\end{proof}

\begin{lemma}[$G$-invariant $\M$-isoperimetric lemma]\label{Lem:M-isoperimetric}
	There are $\ep_M>0$ and $C_M \geq 1$ depending only on the isometric embedding $M \hookrightarrow \R^L$, such that for any $\tau_1, \tau_2\in \Z_n^G(M, \partial M;\mZ_2)$ with $ \M(\tau_2-\tau_1) < \ep_M,$
	there exist $Q\in\mI_{n+1}^G(M;\mZ_2)$ and $R\in \mR_n^G(\partial M;\mZ_2)$ satisfying
	\begin{itemize}
		\item $T_2-T_1=\partial Q+R$,
		\item $\M(Q)+\M(R)\leq C_M \M (\tau_2-\tau_1)$,
	\end{itemize}
	where $T_1,T_2$ are the canonical representatives of $\tau_1,\tau_2$ respectively.
\end{lemma}

\begin{proof}
	Let $\epsilon_2>0$ and $C_M\geq 1$ be the constants in \cite[Lemma 3.17]{li2021min}. 
	Take $\epsilon_M\in (0, \epsilon_2)$ small enough so that $2C_M\epsilon_M < \M(M)$. 
	Then, for any $\tau_1, \tau_2\in \Z_n^G(M, \partial M;\mZ_2)$ with $ \M(\tau_2-\tau_1) < \ep_M$, there exist $Q\in\mI_{n+1}(M;\mZ_2)$ and $R\in \mR_n(\partial M;\mZ_2)$ so that $T_2-T_1 = \partial Q + R$ and $\M(Q) + \M(R)\leq C_M \M(\tau_2-\tau_1)$ by \cite[Lemma 3.17]{li2021min}. 
	If there is another $P\in \mI_{n+1}(M;\mZ_2)$ satisfying $ R' := \partial P - (T_2 - T_1)\in \mR_n(\partial M;\mZ_2)$ and $\M(P) + \M(R')\leq C_M \M(\tau_2-\tau_1)$. 
	Then the arguments in Lemma \ref{Lem:F-isoperimetric} imply $Q=P$. 
	Thus, the uniqueness of such $Q$ indicates $Q\in \mI_{n+1}^G(M;\mZ_2)$ is $G$-invariant. 
	Obviously, $R = T_2 - T_1 - \partial Q$ is also $G$-invariant. 
\end{proof}

\subsection{Definitions of almost minimizing varifolds}

In this subsection, we introduce the $G$-almost minimizing varifolds with free boundary. 
In the rest of this section, take a relative open $G$-subset $U\subset M$, and let $\F=\F^M$.

\begin{definition}\label{a.m.deform}
	Let $\underline{\underline{\nu}} = \F$, $\M$ or $\mF$. 
	For any $\epsilon>0,~\delta>0$, and $\tau\in \Z_{n}^G(M,\partial M;\mZ_2)$, we call a finite sequence $\{\tau_i\}_{i=1}^q\subset \Z_{n}^G(M,\partial M;\mZ_2)$ a {\em $G$-invariant $(\epsilon,\delta)$-deformation} of $\tau$ in $U$ under the metric $\underline{\underline{\nu}} $, if
	\begin{itemize}
		\item $\tau_0=\tau$ and ${\rm spt}(\tau-\tau_i)\subset U$ for all $i=1,\dots, q;$
		\item $\underline{\underline{\nu}} (\tau_i-\tau_{i-1})\leq \delta$ for all $i=1,\dots, q;$
		\item ${\bf M}(\tau_i)\leq {\bf M}(\tau)+\delta$ for all $i=1,\dots, q;$
		\item ${\bf M}(\tau_q)< {\bf M}(\tau)-\epsilon$.
	\end{itemize}
	Moreover, we define
	$$ \mathfrak{a}^G_n(U;\epsilon,\delta; \underline{\underline{\nu}}  ) $$
	to be the set of all $\tau\in \Z_{n}^G(M,\partial M;\mZ_2)$ that does not admit any $G$-invariant $(\epsilon,\delta)$-deformation under the metric $\underline{\underline{\nu}} $.
\end{definition}

\begin{definition}[$(G,\mZ_2)$-almost minimizing varifolds with free boundary]\label{Def:G-am-varifolds}
	A varifold $V\in\V_n^G(M)$ is {\em $(G,\mZ_2)$-almost minimizing in $U$ with free boundary} if there exist sequences $\ep_i \to 0$, $\de_i \to 0$, and $\tau_i\in \mathfrak{a}_n^G(U; \ep_i, \de_i; \F)$, such that $\mF(|T_i|, V)\leq \ep_i$, where $\{T_i \in \tau_i\}_{i\in\N}$, are the canonical representatives. 
	
	Furthermore, we say $V$ is $(G,\mZ_2)$-almost minimizing {\em of boundary type} in $U$ with free boundary, if for each $\tau_i$, $i\in\N$, there exists $Q_i\in\mI_{n+1}^G(M;\mZ_2)$ such that $\tau_i=[\partial Q_i]$. 
\end{definition}

We mention that the term `boundary type' means each $\tau_i$ is induced by a boundary $\partial Q_i$ for some $Q_i\in\mI_{n+1}^G(M;\mZ_2)$, and $V$ is approximated by the `relative boundary' $|\partial Q_i\llcorner (M\setminus\partial M) |$.

\begin{proposition}\label{Prop:am-varifold is stable}
If $V \in \V_n^G(M)$ is $(G,\mathbb{Z}_2)$-almost minimizing with free boundary in a relative open $G$-subset $U \subset M$, then $V$ is stationary in $U$ with free boundary.
\end{proposition}
\begin{proof}
	By Lemma \ref{Lem: G-stationary with free boundary and stationary}, the proof of \cite[Theorem 3.3]{pitts2014existence} would carry over.
\end{proof}

Recall the definition of isolated orbits (Definition \ref{Def: isolated orbit}). 
We have the following theorem which is parallel to \cite[Theorem 3.20]{li2021min}.

\begin{theorem}\label{Thm: equivalence-a.m.v}
	Let $V\in\mathcal{V}^G_n(M)$. 
	The following statements satisfy $(a) \Rightarrow (b) \Rightarrow (c)\Rightarrow(d)$:
	\begin{itemize}
		\item[(a)] $V$ is $(G,\mathbb{Z}_2)$-almost minimizing in $U$ with free boundary;
		\item[(b)] for any $\epsilon>0$, there exist $\delta>0$ and $\tau \, \in \, \mathfrak{a}_n^G(U; \ep, \de; \mF)$ such that ${\bf F}(V,|T|)<\epsilon$, where $T \in \tau$ is the canonical representative;
		\item[(c)] for any $\epsilon>0$, there exist $\delta>0$ and $\tau \, \in \, \mathfrak{a}_n^G(U; \ep, \de; \M)$ such that ${\bf F}(V,|T|)<\epsilon$, where $T \in \tau$ is the canonical representative;
		\item[(d)] $V$ is $(G,\mathbb{Z}_2)$-almost minimizing in $W$ with free boundary for any relative open $G$-subset $W\subset\subset U$ containing no isolated orbit of $M$ or $\partial M$. 
	\end{itemize}
\end{theorem}
\begin{proof}
	See Appendix \ref{Sec:proof-of-equivalence-a.m.v}.
\end{proof}

\section{The equivariant min-max construction}\label{Sec:min-max}

This section is parallel to \cite[Section 4]{li2021min} with $\Z_n^G(M,\partial M;\mathbb{Z}_2)$ in place of $\Z_n(M,\partial M)$.

\subsection{Cubical complex}

Denote $I^m=[0, 1]^m$, $I^m_0=\partial I^m$ to be the unit cube and its boundary in $\R^m$.  
For any $j\in \mathbb{N}$, we denote $I(1,j)$ to be the cube complex on $I=[0,1]$ with $1$-cells 
   $$[0,3^{-j}], [3^{-j},2 \cdot 3^{-j}],\dots,[1-3^{-j}, 1],$$
   and $0$-cells (vertices)
$$[0], [3^{-j}],\dots,[1-3^{-j}], [1].$$
Then the cell complex on $I^m$ is defined as:
$$I(m,j):=I(1,j)\otimes\dots \otimes I(1,j)\quad (\mbox{$m$ times}).$$
Define then the boundary homeomorphism $\partial: I(m, j)\rightarrow I(m, j)$ as
$$\partial(\al_1\otimes\cdots\otimes\al_m):=\sum_{i=1}^m(-1)^{\si(i)}\al_1\otimes\cdots\otimes \partial \al_i\otimes\cdots\otimes\al_m,$$
where $\si(i)=\sum_{l<i}\dim(\al_l)$, $\partial[a, b]=[b]-[a]$ if $[a, b]\in I(1, j)_1$, and $\partial[a]=0$ if $[a]\in I(1, j)_0$.

We say $\alpha=\alpha_1 \otimes \cdots\otimes \alpha_m$ is a {\em $q$-cell} of $I(m,j)$, if $\alpha_i$ is a cell
of $I(1,j)$ for each $i$, and $\sum_{i=1}^m {\rm dim}(\alpha_i) =q$. 
Let $I(m,j)_p$ be the set of all $p$-cells in $I(m,j)$, and $I_0(m,j)_p$ be the set of all $p$-cells in $I(m,j)$ supported in $I_0^m$. 


The {\em distance} between two vertices $x, y\in I(m,j)_0$ is defined by ${\bf d}(x,y) := 3^j\cdot\sum_{i=1}^m|x_i-y_i|$. And we say $x,y$ are {\em adjacent} if ${\bf d}(x,y)=1$.
Given $i,j\in \mathbb{N}$ we define ${\bf n}(i,j):I(m,i)_0\rightarrow I(m,j)_0$ as a map so that ${\bf n}(i,j)(x)$ is the unique element of $I(m,j)_0$ with 
$${\bf d}(x,{\bf n}(i,j)(x)) := \inf\{{\bf d}(x,y) : y\in I(m,j)_0\}.$$
For any map $\phi:I(m,j)_0\rightarrow  \mathcal{Z}_n^G(M,\partial M;\mathbb{Z}_2)$, the {\em $\M$-fineness} of $\phi$ is defined as
$${\bf f}_{\M}(\phi) := \sup\left\{{\bf M}(\phi(x)-\phi(y)) :  {\bf d}(x,y)=1,~ x,y\in  I(m,j)_0\right\}.$$

\subsection{Homotopy notions in discrete settings} 

In the rest of this paper, we denote $\phi: I(m, j)_{0}\rightarrow\big(\Z_{n}^G(M, \partial M;\mZ_2), \{0\}\big)$ as a mapping such that $\phi\big(I(m, j)_{0}\big)\subset\Z_{n}^G(M, \partial M;\mZ_2)$ and $\phi|_{I_{0}(m, j)_{0}}=0$.

\begin{definition}[Homotopy for mappings]\label{Def:homotpy for maps}
	Let $\phi_i:I(m,k_i)_0\rightarrow  \big(\Z_{n}^G(M, \partial M;\mZ_2), \{0\}\big)$, $i=1,2$. We say $\phi_1$ and $\phi_2$ are {\it $m$-homotopic in $\big(\Z_{n}^G(M, \partial M;\mZ_2), \{0\}\big)$ with $\M$-fineness $\delta$}  if there exists
	$$\psi: I(1,k)_0\times I(m,k)_0\rightarrow  \Z_{n}^G(M, \partial M;\mZ_2)$$
	for some $k\geq\max\{k_1,k_2\}$ such that
	\begin{itemize}
		\item[(i)] ${\bf f}_\M (\psi)<\delta$;
		\item[(ii)] if $i\in\{1,2\}$ and $x\in I(m,k)_0$, then $\psi([i-1],x)=\phi_i( {\bf n}(k,k_i)(x))$;
		\item[(iii)] $\psi(I(1,k)_0\times I_0(m,k)_0)=0 $.
	\end{itemize}
\end{definition}

\begin{definition}\label{Def:homotopy sequence}
	A sequence of mappings $S=\{\phi_i\}_{i=1}^\infty$,
	$\phi_i:I(m,k_i)_0\rightarrow \big(\Z_{n}^G(M, \partial M;\mZ_2), \{0\}\big),$ is called an
 	$$\mbox{{\it $(m,{\bf M})$-homotopy sequence of mappings into $\big(\Z_{n}^G(M, \partial M;\mZ_2), \{0\}\big)$}},$$ 
	if $\phi_i$ and $\phi_{i+1}$ are $m$-homotopic in $\big(\Z_{n}^G(M, \partial M;\mZ_2), \{0\}\big)$ with $\M$-fineness $\delta_i$ such that
	\begin{itemize}
		\item[(i)] $\lim_{i\rightarrow\infty} \delta_i=0$;
		\item[(ii)]$\sup\{{\bf M}(\phi_i(x)) :  x\in I(m,k_i)_0, ~i\in \mathbb{N}\}<+\infty.$
	\end{itemize}
\end{definition}

\begin{definition}[Homotopy for sequences of mappings]\label{Def:homotopy for sequence}
	Given $S^j=\{\phi^j_i\}_{i=1}^\infty$, $j=1,2$, two $(m,{\bf M})$-homotopy sequences of mappings into $\big(\Z_{n}^G(M, \partial M;\mZ_2), \{0\}\big)$, we say {\em $S^1$ is $G$-homotopic to $S^2$} if there exists a sequence $\{\delta_i\}_{i\in \mathbb{N}}$ such that
 	\begin{itemize}
		\item[(i)] $\phi^1_i$  is $m$-homotopic to $\phi^2_i$ in $\big(\Z_{n}^G(M, \partial M;\mZ_2), \{0\}\big)$ with $\M$-fineness $\delta_i$;
		\item[(ii)] $\lim_{i\rightarrow\infty} \delta_i=0.$
 	\end{itemize}
	Moreover, we call the equivalence class of any such sequence an
	$$\mbox{{\it $(m,{\bf M})$-homotopy class of mappings into $ \big(\Z_{n}^G(M, \partial M;\mZ_2), \{0\}\big)$}} ,$$
	and denote by $\pi_m^\sharp \big(\Z_{n}^G(M, \partial M; \M; \mZ_2), \{0\}\big)$ the set of all such equivalence classes.
\end{definition}

For any $\Pi \in  \pi_m^\sharp \big(\Z_{n}^G(M, \partial M; \M; \mZ_2), \{0\}\big)$, define the function ${\bf L} : \Pi\rightarrow [0,+\infty]$ by
\begin{eqnarray*}
	{\bf L}(S) := \limsup_{i\rightarrow\infty} \max_{x\in \mathrm{dmn}(\phi_i)} {\bf M}(\phi_i(x)), \qquad \forall S=\{\phi_i\}_{i\in \mathbb{N}}\in\Pi
\end{eqnarray*}

\begin{definition}[Width]\label{Def:width}
	Given $\Pi \in \pi_m^\sharp \big(\Z_{n}^G(M, \partial M; \M; \mZ_2), \{0\}\big)$, the {\em width} of $\Pi$ is 
$${\bf L}(\Pi):=\inf_{S\in \Pi}{\bf L}(S).$$
We say  $S\in \Pi$ is a {\it critical sequence} for $\Pi$ if ${\bf L}(S)={\bf L}(\Pi)$. 
\end{definition}

\begin{definition}[Critical set]\label{Def:critical set}
	Let $\Pi \in \pi_m^\sharp \big(\Z_{n}^G(M, \partial M; \M; \mZ_2), \{0\}\big)$. 
	For any $S=\{\phi_i\}_{i\in \mathbb{N}}\in \Pi$, we define the {\em image set} of $S$ as the compact subset ${\bf K}(S)\subset\mathcal{V}^G_n(M)$ given by
	\begin{multline*}
		{\bf K}(S) := \{V\in\V^G_n(M)  : V=\lim_{j\rightarrow\infty}|T_j|\mbox{ for some sequence}~ i_1<i_2<\dots,\\
		\mbox{some $x_j\in \mathrm{dmn}(\phi_{i_j})$, and $T_j\in \phi_{i_j}(x_j)$ is canonical}\}. 
	\end{multline*}
	The {\em critical set}  ${\bf C}(S)$ of a critical sequence $S\in \Pi$ is given by
	$${\bf C}(S) := \{V\in {\bf K}(S) : \|V\|(M)={\bf L}(S)\}.$$
\end{definition}

\subsection{Homotopy notions in continuous settings}\label{Sec:homotopy continuous}

The notions in the above subsection are the discrete analogs of the usual notions for continuous maps. 
It seems to be essential to use those discrete families in the proof of the regularity result. 
Nevertheless, one can also restrict the proof to apply a more reasonable continuous theory. 
In this subsection, we generalize the continuous settings in \cite[Section 3]{marques2016morse} to our $G$-invariant free boundary case. 
To distinguish from the discrete case, we use bold symbols and capital letters here. 

We denote $\Phi:I^m\to \big(\Z_{n}^G(M, \partial M;\underline{\underline{\nu}} ;\mZ_2), \{0\}\big)$ as a map from $I^m$ to $\Z_{n}^G(M, \partial M;\mZ_2)$, which is continuous in the $\underline{\underline{\nu}} $-topology ($\underline{\underline{\nu}}  = \F$, $\M$ or $\mF$) and $\Phi\vert_{\partial I}=0$. 

\begin{definition}\label{Def:homotopy continuous}
	Let $\Phi_i:I^m\to \big(\Z_{n}^G(M, \partial M;\mF;\mZ_2), \{0\}\big)$, $i=1,2$, be two $\mF$-continuous maps. 
	We say $\Phi_1$ is {\em $G$-homotopic} to $\Phi_2$ if there exists a map $\Psi:I\times I^m\to \Z_{n}^G(M, \partial M;\F;\mZ_2)$ continuous in the  $\F$-topology so that 
	\begin{itemize}
		\item $\Psi(0,x)=\Phi_1(x),~\Psi(1,x)=\Phi_2(x)$, for any $x\in I^m$; 
		\item $\Psi(t,x) =0 $ for any $t\in I,~x\in \partial I^m$. 
	\end{itemize}
	We denote by $\bm{\Pi}$ a {\em (continuous) $G$-homotopy class}, and denote by $\pi_m' \big(\Z_{n}^G(M, \partial M;\mF;\mZ_2), \{0\}\big)$ the set of all continuous $G$-homotopy classes.
\end{definition}

Note the maps $\Phi_i$ are continuous in the $\mF$-topology, but the homotopy map $\Psi$ is only $\F$ continuous. 

Given $\bm{\Pi}\in \pi_m' \big(\Z_{n}^G(M, \partial M;\mF;\mZ_2), \{0\}\big)$, let ${\bf L}: \bm{\Pi}\rightarrow [0,+\infty)$ be the function defined as
\begin{eqnarray*}
	{\bf L}(\Phi) := \sup_{x\in I^m}\M(\Phi(x)), \quad \forall \Phi\in\bm{\Pi}. 
\end{eqnarray*}
Furthermore, if $\{\Phi_i\}_{i=1}^\infty$ is a sequence in $\bm{\Pi}$, we denote 
$${\bf L}(\{\Phi_i\}_{i=1}^\infty):=\limsup_{i\to\infty}\sup_{x\in I^m}\M(\Phi_i(x)).$$

\begin{definition}[Width]\label{Def:width}
	Let $\bm{\Pi}\in \pi_m' \big(\Z_{n}^G(M, \partial M;\mF;\mZ_2), \{0\}\big)$ be a continuous $G$-homotopy class. 
	The {\em width} of $\bm{\Pi}$ is defined by
	$${\bf L}(\bm{\Pi}):=\inf_{\Phi\in \bm{\Pi}}{\bf L}(\Phi).$$
	We call $\{\Phi_i\}_{i=1}^\infty\subset \bm{\Pi}$ a {\it min-max sequence for $\bm{\Pi}$} if 
	${\bf L}(\{\Phi_i\}_{i=1}^\infty)={\bf L}(\bm{\Pi}).$
\end{definition}

Clearly, there exists a min-max sequence for any $\bm{\Pi}\in \pi_m' \big(\Z_{n}^G(M, \partial M;\mF;\mZ_2), \{0\}\big)$.

\begin{definition}[Critical set]\label{Def:critical set}
	Let $\bm{\Pi}\in \pi_m' \big(\Z_{n}^G(M, \partial M;\mF;\mZ_2), \{0\}\big)$ be a continuous $G$-homotopy class. 
	For any $\{\Phi_i\}_{i\in \mathbb{N}}\subset \bm{\Pi}$, we define the {\em image set} of $\{\Phi_i\}_{i\in\N}$ as the compact subset ${\bf K}(\{\Phi_i\}_{i\in\N})\subset\mathcal{V}^G_n(M)$ given by
	\begin{multline*}
		{\bf K}(\{\Phi_i\}_{i\in\N}) := \{V\in\V^G_n(M)  : V=\lim_{j\rightarrow\infty}|T_j|\mbox{ for some sequence}~ i_1<i_2<\dots,\\
		\mbox{some }x_j\in \mathrm{dmn}(\Phi_{i_j}), \mbox{ and $T_j\in \Phi_{i_j}(x_j)$ is canonical}\}.
	\end{multline*}
	Moreover, if $\{\Phi_i\}_{i\in\N}$ is a min-max sequence, then the {\em critical set} ${\bf C}(\{\Phi_i\}_{i\in\N})$ is given by
	$${\bf C}(\{\Phi_i\}_{i\in\N}) := \{V\in {\bf K}(\{\Phi_i\}_{i\in\N}) :  \|V\|(M)={\bf L}(\bm{\Pi})\}.$$
\end{definition}

\subsection{Discretization and interpolation}

The discretization and interpolation results have been proved by Li-Zhou in \cite[Section 4.2]{li2021min} for manifolds with boundary (see also \cite[Section 13,14]{marques2014min} for the closed case), which build a bridge between continuous and discrete mappings into $\Z_n(M,\partial M)$. 
In this section, we go further to show these results under $G$-invariant restrictions. 
To begin with, we introduce the following equivariantly modified technical assumption:

\begin{definition}\label{Def:no concent mass on orbit}
	Given $\Phi:I^m\rightarrow \mathcal{Z}_n^G(M,\partial M;\mathbb{Z}_2)$ continuous in the flat topology, define 
	$${\bf m}^G(\Phi,r) :=\sup\{\|T_x\|(\tBcal_r^G(p)) : x\in I^m,~p\in M  \},$$
	where $T_x$ is the canonical representative of $\Phi(x)$, and $\tBcal_r^G(p)$ is the open geodesic tube in $\tM$ of radius $r$ centered at $G\cdot p$. 
	Then, $\Phi$ is said to {\it have no concentration of mass on orbits} if 
	$$\lim_{r\rightarrow 0}{\bf m}^G(\Phi,r)=0.$$
\end{definition}

Without the actions of $G$, this definition was first made by Marques-Neves in \cite[Definition 4.2]{marques2014min} for closed manifolds. 
For compact manifolds with boundary, a similar definition was given by Li-Zhou in \cite[Section 4.2]{li2021min}. 
It is a mild technical condition and we also have the following lemma for relative $G$-cycles. 

\begin{lemma}\label{Lem: mass.continu.no.concent}
	Suppose $\Phi:I^m\to \Z_{n}^G(M, \partial M;\mZ_2)$ is continuous in the mass norm, then $\sup_{x\in I^m}\M(\Phi(x)) < \infty$ and $\Phi$ has no concentration of mass on orbits. 
	
	Furthermore, if $\Phi:I^m\to \Z_{n}^G(M, \partial M;\mZ_2)$ is continuous in the ${\bf F}$-metric, then $\Phi$ has mass bounded images and has no concentration of mass on orbits.
\end{lemma}
\begin{proof}
	The proof is essentially the same as \cite[Lemma 8]{wang2022min}. 
	By Lemma \ref{Lem:canonical representative} and \ref{Lem:F-metric}, we have the continuity of maps $x\mapsto \M(\Phi(x))$ and $x \mapsto \|T_x\|(\tBcal_r^G(p))$ for $x\in I^m$, where $T_x$ is the canonical representative of $\Phi(x)$ and $\tBcal_r^G(p)$ is the open geodesic tube in $\tM$ of radius $r$ centered at $G\cdot p$. 
	Then the compactness arguments in \cite[Lemma 8]{wang2022min} can be applied to deliver the result. 
\end{proof}

Next, we show the discretization theorem, which generates an $(m,{\bf M})$-homotopy sequence of mappings into $\mathcal{Z}_n^G(M, \partial M; \mathbb{Z}_2)$ from a flat continuous map with no concentration of mass on orbits.

\begin{theorem}[Discretization Theorem]\label{Thm: discretization}
	Let $\Phi:I^m\rightarrow \mathcal{Z}_n^G(M,\partial M;\mathbb{Z}_2)$ be a continuous map in the flat topology satisfying:
	\begin{itemize}
		\item[$(a)$] $\sup_{x\in I^m}\M (\Phi(x))<\infty$,
		\item[$(b)$] $\Phi$ has no concentration of mass on orbits,
		\item[$(c)$] $\Phi|_{I^m_0}$ is continuous in the $\mF$-metric.
	\end{itemize}
	Then there exists a sequence of maps
	$$\phi_i:I(m,j_i)_0 \rightarrow \mathcal{Z}_n^G(M,\partial M;\mathbb{Z}_2),$$
	with $j_i<j_{i+1}$, and a sequence of positive numbers $\{\delta_i\}_{i\in\mathbb{N}}$ converging to zero such that
	\begin{itemize}
		\item[(i)] $S=\{\phi_i\}_{i\in\mathbb{N}}$ is an $(m,{\bf M})$-homotopy sequence into $\mathcal{Z}_n^G(M,\partial M;\mathbb{Z}_2)$ with $\M$-fineness ${\bf f}_\M(\phi_i)<\delta_i$;
		\item[(ii)] there exists some sequence $k_i \to +\infty$ such that for all $x\in I(m, j_i)_0$,
		$$ \M (\phi_i(x)) \leq \sup \{ \M (\Phi(y)) : \al\in I(m, k_i)_m, x, y \in\al\}+\de_i,$$
		which implies $\bL(S)\leq\sup_{x\in I^m}\M(\Phi(x))$; 
		\item[(iii)] $\sup\{\mathcal F(\phi_i(x)-\Phi(x)) : x\in I(m,j_i)_0\}\leq \delta_i;$
		\item[(iv)] $\M(\phi_i(x))\leq \M(\Phi(x))+\de_i$, for all $x\in I_0(m, j_i)_0$.
	\end{itemize}
	
	Furthermore, if $\Phi|_{I^m_0}=0$, then $\phi_i$ can be chosen with $\phi_i\vert_{I_0(m,j_i)_0}=0 $, i.e. $S$ can be chosen as an $(m, \M)$-homotopy sequence of mappings into $\big(\Z_{n}^G(M, \partial M;\mZ_2), \{0\}\big)$. 
\end{theorem}

\begin{proof}
	The proof is essentially the same as \cite[Theorem 13.1]{marques2014min} and \cite[Theorem 4.12]{li2021min}, where the only explicit operations appeared in \cite[Theorem 13.4]{marques2014min}. 
	Hence, it is sufficient to adapt \cite[Lemma 13.4]{marques2014min} into a $G$-invariant free boundary version, and then this theorem follows from the lemmas in Section \ref{Sec:relative G-cycle} with a combinatorial argument as in \cite[Theorem 13.1]{marques2014min}. 
	Indeed, we can use the $G$-invariant isoperimetric Lemma \ref{Lem:F-isoperimetric} in place of \cite[Corollary 1.14]{almgren1962homotopy}, and define $\psi_i$ as (\ref{Eq: pre-interpolation}) to deliver the $G$-invariant free boundary version of \cite[Lemma 13.4]{marques2014min} for the space $\Z_n^G(M,\partial M; \mZ_2 )$. 
\end{proof}

\begin{remark}
	Motivated by the constructions of cones in \cite[Section 3.5, 3.7]{pitts2014existence}, Zhou \cite{zhou2017min} proved a stronger version of the discretization theorem for continuous maps of {\em boundary type} without the no mass concentration assumption. 
	In our equivariant setting, suppose there is no isolated orbit, then we can combine the constructions in Appendix \ref{Sec:proof-of-equivalence-a.m.v} with the arguments in \cite[Section 5]{zhou2017min} to show that Theorem \ref{Thm: discretization} is also valid for any $\F$-continuous map $\Phi$ of boundary type (i.e. there exists $Q(x)\in\bI_{n+1}^G(M;\mZ_2) $ so that $\Phi(x)=[\partial Q(x)]$) without the no mass concentration on orbits condition (b). 
\end{remark}

The following theorem shows we can generate an ${\bf M}$-continuous map from a discrete map with small $\M$-fineness, which is parallel to \cite[Theorem 4.14]{li2021min} and \cite[Theorem 3]{wang2022min}. 

\begin{theorem}[Interpolation Theorem]\label{Thm:interpolation}
	There exist positive constants $C_0=C_0(M,G,m)$ and $\delta_0=\delta_0(M,G)$  so that if
	$$\phi:I(m,k)_0\rightarrow \mathcal{Z}_n^G(M,\partial M;\mathbb{Z}_2)$$
	has ${\bf f}_{\M}(\phi)<\delta_0$, then there exists a map
	$$ \Phi: I^m \rightarrow \mathcal{Z}_n^G(M,\partial M;\mathbb{Z}_2)$$
	continuous in the ${\bf M}$-topology satisfying:
	\begin{itemize}
		\item[(i)] $\Phi(x)=\phi(x)$ for all $x\in I(m,k)_0$;
		\item[(ii)] if $\alpha$ is some $j$-cell in $I(m,k)$, then $\Phi$ restricted to $\alpha$ depends only on the values of $\phi$ assumed on the vertices of  $\alpha$;
		\item[(iii)] $\sup\{{\bf M}(\Phi(x)-\Phi(y)) :  x,y\mbox{ lie in a common cell of } I(m,k)\}\leq C_0{\bf f}_\M(\phi).$
	\end{itemize}
	Moreover, $\Phi|_{I^m_0}=0$ provided $\phi|_{I_0(m,k)_0} =0$. 
\end{theorem}

We call the map $\Phi$ in Theorem \ref{Thm:interpolation} \textit{the Almgren $G$-extension} of $\phi$. 

\begin{proof}
	The proof of this interpolation theorem is a combination of the arguments in \cite[Theorem 3]{wang2022min} and \cite[Theorem 4.14]{li2021min}. 
	We only consider the case of $\phi:I(m,0)_0\rightarrow \mathcal{Z}_n^G(M,\partial M;\mathbb{Z}_2)$.
	As for general ${\rm dmn}(\phi)=I(m,k)_0$, the conclusion can be made as in \cite[Theorem 14.2]{marques2014min}. 
	
	By Lemma \ref{Lem:M-isoperimetric}, we can take $\delta_0$ small enough such that for any $\alpha\in I(m,0)_1$ with $\partial \alpha=[b]-[a]$, $\phi([a])=\tau_1$, $\phi([b])=\tau_2$, there exist $Q(\al)\in\mI_{n+1}^G(M;\mZ_2)$ and $R(\al)\in\mR_n^G(\partial M;\mZ_2)$ satisfying 
	\begin{itemize}
		\item $T_2-T_1=\partial Q(\al)-R(\al)$,
		\item $\M(Q(\al))+\M(R(\al))\leq C_M\M(\tau_2-\tau_1)$,
	\end{itemize}
	where $T_i\in\tau_i,i\in\{1,2\}$, are the canonical representatives. 
	
	By Illman's work \cite{illman1983equivariant}\cite{illman2000Existence}, we can get a triangulation $\triangle$ of the orbit space $M/G$, which generates an equivariant triangulation $\widetilde{\triangle} := \{\pi^{-1}(s) : s\in\triangle\}$ of $M$ with some nice properties (see \cite[Appendix A]{wang2022min}). 
	Therefore, we can construct the cutting maps and the deforming maps for integral $G$-currents in the same way as in \cite[Theorem 3]{wang2022min}. 
	Specifically, since the distance function to any $G$-cell $\tilde{s}\in\widetilde{\triangle}$ is a $G$-invariant Lipschitz function, we can use $\dist_M(\tilde{s}, \cdot)$ to construct $G$-invariant slice current (\cite[Lemma 4]{wang2022min}). 
	Thus, for any finite set $\Lambda\subset \mI_{n+1}^G(M;\mathbb{Z}_2)$, we can use the $G$-invariant slices to associate every $\tilde{s}\in\widetilde{\triangle}$ to a $G$-neighborhood $L(\tilde{s})\subset U(\tilde{s}):= \cup_{\tilde{s}\subset\tilde{s}'}\tilde{s}'$ of $\tilde{s}$ and construct a cutting map $C_\Lambda: \widetilde{\triangle}\times\Lambda \rightarrow \mI_{n+1}^G(M;\mathbb{Z}_2)$ similar to \cite[Section 5]{almgren1962homotopy} such that
	\begin{eqnarray}
		C_\Lambda(\tilde{s},Q) &=& \Big( Q-\sum_{\tilde{s}'\prec\tilde{s}} C_\Lambda(\tilde{s}',Q) \Big) \cap L(\tilde{s}),
		\\
		{\bf M}\Big( \partial C_\Lambda(\tilde{s},Q) &-& \partial\big(Q-\sum_{\tilde{s}'\prec\tilde{s}} C_\Lambda(\tilde{s}',Q) \big )\cap L(\tilde{s})  \Big) 
		\\
		&\leq & C_0\cdot (\#\Lambda)\cdot{\bf M}(Q-\sum_{\tilde{s}'\prec\tilde{s}} C_\Lambda(\tilde{s}',Q)), \nonumber
		\\
		{\rm spt}(C_\Lambda(\tilde{s},Q))&\subset & U(\tilde{s}),\quad \forall \tilde{s}\in\widetilde{\triangle},~Q\in\Lambda,
	\end{eqnarray}
	where $C_0>0$ depends only on $\widetilde{\triangle}$ and $m$. 
	Note the $G$-neighborhood $L(\tilde{s})$ can be chosen with a uniform bound $N(\tilde{s})\subset \pi^{-1}({\rm Int}(\cup_{s \subset s'} s'))$ for all finite set $\Lambda\subset \mI_{n+1}^G(M;\mathbb{Z}_2)$, i.e. ${\rm spt}(C_\Lambda(\tilde{s},Q))\subset N(\tilde{s})$ for any finite set $\Lambda\subset \mI_{n+1}^G(M;\mathbb{Z}_2)$ and $Q\in\Lambda$. 
	Next, in the uniformly defined $G$-neighborhood $N(\tilde{s})$, the deforming maps can be constructed as in \cite[Theorem 3]{wang2022min}. 
	Thus, we get a map
	\begin{eqnarray*}
		D(\tilde{s}) : I\times \mI_{n}^G(N(\tilde{s});\mathbb{Z}_2) \rightarrow \mI_{n}^G(N(\tilde{s});\mathbb{Z}_2),
	\end{eqnarray*}
	so that
	\begin{itemize}
		\item $D(\tilde{s})$ is continuous in the ${\bf M}$-norm,
		\item $D(\tilde{s})(0,T)=T$ for any $\tilde{s}\in\widetilde{\triangle}$ and $T\in \mI_{n}^G(N(\tilde{s});\mathbb{Z}_2) $,
		\item $D(\tilde{s})(1,T)=0$ for any $\tilde{s}\in\widetilde{\triangle}$ and $T\in \mI_{n}^G(N(\tilde{s});\mathbb{Z}_2) $,
		\item ${\bf M}(D(\tilde{s})(t,T))\leq C {\bf M}(T)$, where $C=C(M,G)>0$, $\tilde{s}\in\widetilde{\triangle},~t\in I,~T\in \mI_{n}^G(N(\tilde{s});\mathbb{Z}_2)$.
	\end{itemize}

	Now, for any $k$-cell $\alpha\in I(m,0)_k$, we write the interpolation formula $\tilde{h}_\alpha: I^k\rightarrow \mR_{n}^G(U(\tilde{s});\mathbb{Z}_2)$ similar to \cite[Theorem 4.14]{li2021min}:
	\begin{equation*}
		\tilde{h}_\alpha(0) := T_\alpha\in\phi(\alpha)~{\rm ~the~canonical~representative},\quad {\rm for} ~k=0,
	\end{equation*}
	and for $k\geq 1$:
	\begin{equation*}
		\begin{split}
			&\tilde{h}_\alpha(x_1,\dots,x_k) := \sum_{\gamma\in\Gamma_\alpha} {\rm sign}(\gamma)\cdot 
			\\
			&\sum_{\tilde{s}_1,\dots,\tilde{s}_k\in\widetilde{\triangle}} D(\tilde{s}_1,x_1)\circ\cdots\circ D(\tilde{s}_k,x_k) \circ \partial \circ C_{\Lambda(\gamma_k)}(\tilde{s}_k)\circ\cdots\circ C_{\Lambda(\gamma_1)}(\tilde{s}_1) (Q(\gamma_1)),
		\end{split}
	\end{equation*}
	where $\Ga_{\al}$ is the set of all sequences $\{\ga_i\}_{i=1}^k$ such that $\ga_k=\al$, and $\ga_i$ is a $(\dim(\ga_{i+1})-1)$-face of $\ga_{i+1}$ for $1\leq i\leq k-1$. 
	Moreover, ${\rm sign}(\gamma)\in\{-1,1\}$ is defined by \cite[6.2]{almgren1962homotopy}, 
	and $\Lambda(\gamma_i )$ are all defined inductively as in \cite[Page 763]{marques2014min}. 
	Hence, for each $k$-cell $\al\in I(m,0)_k$, there is an $\M$-continuous function $h_\alpha:I^k\to\Z_n^G(M,\partial M;\mZ_2)$ defined by 
	$$h_\alpha(x_1,\dots,x_k) := [\tilde{h}_\alpha(x_1,\dots,x_k)].$$
	We mention that in the definition of $\tilde{h}_\alpha$, we do not need to cut or deform $R(\gamma_1)$ as in the proof of \cite[Theorem 4.14]{li2021min}. 
	This is because the term $R(\gamma_1)$ in $\partial M$ will be ignored after taking the equivalent class to get $h_\alpha$. 
	(Indeed, the chain maps $\phi_B^{n,\epsilon}$ and $\psi_B^n$ are zero in the proof of \cite[Theorem 7.5]{almgren1962homotopy}.)
	Finally, the rest of the proof is the same as in \cite[Page 763]{marques2014min}, and the last statement follows directly from the construction. 
\end{proof}

\begin{remark}\label{Rem:almgren iso and homo}
	As an application of Theorem \ref{Thm:interpolation} and the lemmas in Section \ref{Sec:relative G-cycle}, the Almgren's Isomorphism \cite[Theorem 7.5]{almgren1962homotopy} and Almgren's Homotopies \cite[Theorem 8.2]{almgren1962homotopy} remain true for the space $\mathcal Z_{n}^G(M,\partial M;\mathbb{Z}_2)$. 
	Moreover, we also have \cite[Corollary 3.12]{marques2017existence} holds for mappings into $\mathcal Z_{n}^G(M,\partial M;\mathbb{Z}_2)$. 
	Specifically, if $\{\phi_i\}_{i\in\N} $ is the sequence of maps associated to $\Phi:I^m\rightarrow \mathcal{Z}_n^G(M,\partial M;\mathbb{Z}_2)$ by Theorem \ref{Thm: discretization}, and $\Phi_i$ is the Almgren $G$-extension of $\phi_i$ given by Theorem \ref{Thm:interpolation} for $i$ large enough, then $\Phi$ is $G$-homotopic to $\Phi_i$ for $i$ large enough. 
	We omit the proofs since they are essentially the same as in \cite{almgren1962homotopy} and \cite{marques2017existence}. 
\end{remark}

\subsection{Tightening}
The \emph{tightening process} (see \cite[Section 4.3]{li2021min}) for a critical sequence $S^*=\{\phi_i^*\} \in \Pi$ is to make every $V\in {\bf C}(S^*) $ a stationary varifold with free boundary. 
For any $n$-varifold $V$ that is not stationary with free boundary, there is a vector field $X(V)\in \mathfrak{X}_{tan}(M)$ so that the mass of $V$ is decreasing under the flow $\{f_t^{X(V)}\}_{t \in [0,1]}$ generated by $X(V)$. 
The idea is to continuously associate every varifold $V$ with a vector field $X(V)\in \mathfrak{X}_{tan}(M)$ so that: if $V$ is stationary with free boundary, then $X(V)=0$; if $V$ is not stationary with free boundary, then $X(V)\neq 0$ is a mass decreasing variation vector field. 
Noting $f_t^{X(V)}(\partial M) = \partial M $, we can roughly take $\phi_i(x)= [(f^{X(|T_x|)}_1)_\# T_x ]$, $i\in\N$, as the tighten sequence, where $T_x$ is the canonical representative of $\phi_i^*(x)$. 
Here we say `roughly' because the push-forwards of $\phi_i^*$ by the equivariant variations $\{f^{X(|T_x|)}_t\}$ are only continuous in the $\mF$-topology, while the fineness of $\phi_i$ is defined with the $\M$-norm, which may not be well controlled under such deformations. 
To fill this gap, an interpolation and discretization process is needed (see \cite[Page 766-768]{marques2014min}). 
Nevertheless, if $\{\Phi_i^*\}\subset \bm{\Pi}$ is a continuous min-max sequence, then $\Phi_i(x)=[(f^{X(|T_x|)}_1)_\# T_x ]$ is also a min-max sequence, where $T_x\in \Phi_i^*(x)$ is canonical. 
Furthermore, we can use the $G$-invariant Discretization Theorem \ref{Thm: discretization} to generate a sequence of discrete mappings with finesses tending to zero.

\begin{proposition}[Tightening]\label{Prop:tightening}
	Let $\bm{\Pi}\in \pi_m' \big(\Z_{n}^G(M, \partial M;\mF;\mZ_2), \{0\}\big)$ be a continuous $G$-homotopy class. 
	For any min-max sequence $\{\Phi_i^*\}_{i\in\N}\subset \bm{\Pi}$, there exists another min-max sequence $\{\Phi_i\}_{i\in\N}$ for $\bm{\Pi}$ such that
	\begin{itemize}
		\item ${\bf C}(\{\Phi_i\}_{i\in\N}) \subset {\bf C}(\{\Phi_i^*\}_{i\in\N})$,
		\item every $G$-varifold $V\in {\bf C}(\{\Phi_i\}_{i\in\N})$ is stationary in $M$ with free boundary.
	\end{itemize}
	Moreover, there exist a sequence $k_i\to\infty$ and a sequence $S=\{\varphi_i\}_{i\in\N}$ of discrete mappings 
	$$\varphi_i: I(m, k_i)_{0}\rightarrow\big(\Z_{n}^G(M, \partial M;\mZ_2), \{0\}\big)$$
	so that 
	\begin{itemize}
		\item[(1)] $\mf_\M(\varphi_i)=\delta_i\to 0$, as $i\to\infty$; 
		\item[(2)] the Almgren $G$-extension of $\varphi_i$ is $G$-homotopic to $\Phi_i$ in the flat topology; 
		\item[(3)] ${\bf L}(S)={\bf L}(\{\Phi_i\}_{i\in\N})={\bf L}(\bm{\Pi})$;
		\item[(4)] ${\bf C}(S)={\bf C}(\{\Phi_i\}_{i\in\N})$.
	\end{itemize}
\end{proposition}

Here we define 
${\bf L}(S) := \limsup_{i\rightarrow\infty} \max_{x\in \mathrm{dmn}(\varphi_i)} {\bf M}(\varphi_i(x))$ and ${\bf C}(S)$ as in Definition \ref{Def:critical set}, even though $S=\{\varphi_i\}_{i\in\N}$ may not be an $(m,{\bf M})$-homotopy sequence of mappings into $\big(\Z_{n}^G(M, \partial M;\mZ_2), \{0\}\big)$.

\begin{proof}
	Denote by $C:=\sup_{i\in \N}\sup_{x\in I^m}\M(\Phi_i^*(x)) < \infty$ and 
	\begin{eqnarray*}
		A&:=&\{V\in\mathcal{V}^G_n(M) : \|V\|(M)\leq C \},
		\\
		A_0&:=&\{V\in A : V {\rm ~is ~stationary~in~} M {\rm~with~free~boundary}\}.
	\end{eqnarray*} 
	Since $G$ acts by isometries, $A$ and $A_0$ are compact in the weak topology of varifolds (c.f. \cite[Lemma 3.4]{li2015general}). 
	By Lemma \ref{Lem: G-stationary with free boundary and stationary}, we only need to consider $Y\in \mathfrak{X}^G_{tan}(M)$. 
	Indeed, if $Y\in\mathfrak{X}_{tan}(M)$, we can take $Y_G\in \mathfrak{X}^G_{tan}(M)$ as in the proof of Lemma \ref{Lem: G-stationary with free boundary and stationary}. 
	Then, $ \delta V(Y)=\delta V(Y_G)$ for any $G$-varifold $V\in \V_n^G(M)$. 
	Therefore, we can follow the arguments in \cite[Page 764-765]{marques2014min} to get a continuous map $X: A \to \mathfrak{X}_{tan}^G(M)$ and a continuous time function $\eta:A \to[0,1]$ 
	such that 
	\begin{itemize}
		\item $X(V)=0$ and $\eta(V)=0$, if $V\in A_0$;
		\item $\delta V(X(V))<0$ and $\eta(V)>0$, if $V\in A\setminus A_0$;
		\item $\|(f^{X(V)}_{t})_\#V\|(M) < \|(f^{X(V)}_{s})_\#V\|(M) $ for all $V\in A$ and $0\leq s<t\leq \eta(V)$, 
	\end{itemize}
	where $\{f^{X(V)}_t\}$ are the equivariant diffeomorphisms generated by $X(V)$. 
	Consider $F^V(t,\cdot):=(f^{X(V)}_{\eta(V)t})_\#(\cdot)$ and the following $\mF$-continuous map:
	\begin{eqnarray}\label{Eq-pulltightmap}
		H:~I\times \big(\mathcal{Z}_n^G(M,\partial M;{\bf F};\mathbb{Z}_2)&\cap & \{\tau :  {\bf M}(\tau)\leq C \} \big)\nonumber
		\\
		&\rightarrow & \mathcal{Z}_n^G(M,\partial M;{\bf F};\mathbb{Z}_2)\cap  \{\tau :  {\bf M}(\tau)\leq C \},
		\\
		H(t,\tau) &:=& [F^{|T|}(t,T)], ~T\in\tau {\rm ~is~canonical,}\nonumber
	\end{eqnarray}
	which satisfies:
	\begin{itemize}
		\item $H(0,\tau)=\tau$;
		\item if $|T|$ is stationary with free boundary, then $H(t,\tau)=\tau$ for all $t\in[0,1]$;
		\item if $|T|$ is not stationary with free boundary, then ${\bf M}(H(1,\tau)) < \M(\tau)$.
	\end{itemize}
	Let $\Phi_i :=H(1,\Phi_i^*)$. 
	Then $\{\Phi_i\}_{i\in\N}\subset \bm{\Pi}$ is also a continuous min-max sequence, and every element in ${\bf C}(\{\Phi_i\}_{i\in\N})$ ($\subset {\bf C}(\{\Phi_i^*\}_{i\in\N})$) is stationary in $M$ with free boundary. 
	
	Next, by Lemma \ref{Lem: mass.continu.no.concent}, we can apply Theorem \ref{Thm: discretization} to each $\Phi_i$ and obtain a sequence of maps $\phi_j^i: I(m, k_j^i)_0 \to \Z_n^G(M,\partial M; \mZ_2)$ with $k_j^i<k_{j+1}^i$. 
	It then follows from Theorem \ref{Thm: discretization}(ii)(iii), Lemma \ref{Lem:lower-semicontinue of M}, \ref{Lem:F-metric}, and the continuity of the map $x\mapsto \M(\Phi_i(x))$ that 
	\begin{equation}\label{Eq: F estimates}
		\lim_{j\to\infty} \sup \{ \mF(\phi^i_j(x), \Phi_i(x)) : x\in I(m,k^i_j)_0  \} = 0. 
	\end{equation}
	For $j$ large enough, we can apply Theorem \ref{Thm:interpolation} to obtain the Almgren $G$-extension $\Phi^i_j$ of $\phi^i_j$. 
	As we mentioned in Remark \ref{Rem:almgren iso and homo}, the arguments in \cite[Corollary 3.12]{marques2017existence} suggest that $\Phi^i_j$ is $G$-homotopic to $\Phi_i$ (for $j$ large enough), and thus $\Phi^i_j\in {\bm \Pi}$. 
	Additionally, by Theorem \ref{Thm:interpolation}(i)(iii) and Theorem \ref{Thm: discretization}(ii), we have 
	\begin{equation}\label{Eq: width estimate}
		{\bf L}({\bm \Pi}) \leq {\bf L} (\{\Phi^{i}_{j} \}_{j\in \N} )= {\bf L} (\{\phi^i_j\}_{j\in\N}) \leq \sup _{x \in I^m} \{\M (\Phi_{i}(x) ) \} \rightarrow \mathbf{L}({\bm \Pi}) \text { ~as ~} i \rightarrow \infty. 
	\end{equation}
	After taking a subsequence $j(i)\to\infty$, we have $S=\{\varphi_i \}_{i\in\N}$, $\varphi_i := \phi^i_{j(i)}$, satisfies $\mf_\M(\varphi_i) \to 0$, the Almgren $G$-extension of $\varphi_i$ is $G$-homotopic to $\Phi_i$, and 
	\begin{itemize}
		\item ${\bf L}({\bm \Pi}) = {\bf L}(\{\varphi_i\}_{i\in\N})$ (by (\ref{Eq: width estimate}));
		\item $\lim_{i\to\infty} \sup \{ \mF(\varphi_i(x), \Phi_i(x)) : x\in I(m, k^i_{j(i)})_0  \} = 0$ (by (\ref{Eq: F estimates}));
		\item $\lim_{i\to\infty} \sup \{ \mF(\Phi_i(x), \Phi_i(y)) : x,y\in \alpha, \alpha \in  I(m, k^i_{j(i)})  \} = 0$ (by the $\mF$-continuity). 
	\end{itemize}
	It follows from the last two bullets that ${\bf C}(S)={\bf C}(\{\Phi_i\}_{i\in\N})$. 
\end{proof}

\subsection{Existence of $(G,\mZ_2)$-almost minimizing varifolds}
\label{SS:existence-am}

\begin{definition}[$G$-annulus neighborhood]\label{Def:annulus}
Let $p \in M$ and $r>0$. 
If $p \notin \partial M$, we assume $r<\min \{ {\rm Inj}(G\cdot p), \dist_M(p,\partial M) \}$. 
If $p \in \partial M$, we assume $r < r_{\textrm{Fermi}}^G(p)$ (see Lemma \ref{Lem:Fermi convex}). 
For any $s\in(0,r)$, the (relatively) open $G$-annular neighborhood is defined as 
\[ \A_{s,r}^G(p)= \left\{ \begin{array}{cl}
\tBcal_r^G(p) \setminus \tBcal_s^G(p) & \text{ if } p \in M \setminus \partial M, \\
\tBcal^{G,+}_r(p) \setminus \tBcal^{G,+}_s(p) & \text{ if } p \in \partial M. 
\end{array} \right. \]
\end{definition}

By the choice of $r$, $\exp_{G\cdot p}^\perp$ (or $\widetilde{\exp}_{G\cdot p}^\perp$ in Definition \ref{Def:Fermi exponential map}) is a $G$-equivariant diffeomorphism on $\A_{s,r}^G(p)$. 
For any $v\in {\bf N}(G\cdot p)$ and $\lambda >0$, $\lambda v$ has the same isotropy group as $v$, and thus $\A_{s,r}^G(p)$ contains no isolated orbit (Remark \ref{Rem: no isolated orbit}). 

\begin{definition}\label{Def:am-annuli}
	A $G$-varifold $V \in \V_n^G(M)$ is said to be \emph{$(G,\mZ_2)$-almost minimizing (of boundary type) in annuli with free boundary} if for each $p\in M$, there is $r_{am}(G\cdot p) >0$ so that $V$ is $(G,\mZ_2)$-almost minimizing (of boundary type) in $\A_{s,t}^G(p)$ with free boundary for all $0<s<t\leq r_{am}(G\cdot p)$. 
\end{definition}
\begin{remark}
	By Lemma \ref{Lem:Fermi convex}, we can replace $\mathcal{A}_{s,r}^G(p)$, $\dist_M(p,\partial M)$ by $A_{s, r}^G(p)\cap M$, $\dist_{\R^L}(p, \partial M)$ in the above definition, which is equivalent by shrinking $r_{am}$. 
\end{remark}

\begin{theorem}\label{Thm:exist amv}
	Let $S=\{\varphi_i\}_{i\in\N}$, 
	$$\varphi_i: I(m, k_i)_{0}\rightarrow\big(\Z_{n}^G(M, \partial M;\mZ_2), \{0\}\big),$$
	be a sequence of mappings with $k_i\to\infty$, $\mf_\M(\varphi_i) \to 0$, as $i\to\infty$, ${\bf L}(S)>0$, and every $G$-varifold $V\in{\bf C}(S)$ is stationary in $M$ with free boundary. 
	If no element $V\in {\bf C}(S)$ is $(G,\mZ_2)$-almost minimizing in annuli with free boundary, then there exists a sequence $S^*=\{\varphi_i^*\}_{i\in\N}$ of mappings
	$$\varphi_i^*: I(m, l_i)_{0}\rightarrow\big(\Z_{n}^G(M, \partial M;\mZ_2), \{0\}\big),$$
	for some $l_i\to\infty$ as $i\to\infty$, such that 
	\begin{itemize}
		\item[(i)] $\varphi_i$ and $\varphi_i^*$ are $m$-homotopic to each other in $\big(\Z_{n}^G(M, \partial M;\mZ_2), \{0\}\big)$ with $\M$-finenesses tending to zero,
		\item[(ii)] ${\bf L}(S^*)=\limsup_{i\rightarrow\infty} \max_{x\in \mathrm{dmn}(\varphi_i^*)} {\bf M}(\varphi_i^*(x))<{\bf L}(S)$.
	\end{itemize}
\end{theorem}

\begin{proof}
	The proof is essentially the same as \cite[Theorem 4.10]{pitts2014existence}, which has been modified to $G$-invariant settings in \cite[Theorem 5]{wang2022min} and free boundary settings in \cite[Theorem 4.21]{li2021min}. 
	Note we always use the canonical representatives to induce varifolds throughout the proof (instead of $|\varphi_i(x)|$). 
	By the assumption, for every $V\in{\bf C}(S)$, there exists an orbit $G\cdot p_V\subset M$ so that $V$ is not $(G,\mZ_2)$-almost minimizing in some arbitrary small $G$-annuli $\{A^G_{s,t}(p)\cap M\}$.  
	For the case $p_V\notin \partial M$, the sequence $S^*$ can be obtained by \cite[Theorem 5]{wang2022min} with Theorem \ref{Thm: equivalence-a.m.v} in place of \cite[Theorem 4]{wang2022min}. 
	If $p_V\in\partial M$, we can adapt the proof of \cite[Theorem 4.10]{pitts2014existence} into $G$-invariant free boundary settings with a few modifications. 
	Specifically, in \cite[Page 164, Part 2]{pitts2014existence}, we use Theorem \ref{Thm: equivalence-a.m.v} in place of \cite[Theorem 3.9]{pitts2014existence}. 
	Secondly, in \cite[Page 165, Part 5]{pitts2014existence}, the $\F$-isoperimetric Lemma \ref{Lem:F-isoperimetric} for $G$-invariant relative cycles shall be used in place of \cite[Corollary 1.14]{almgren1962homotopy}. 
	Finally, the cut-and-paste constructions in \cite[Page 166, Part 9]{pitts2014existence} can be done in a similar way to the proof of Case 1 in Lemma \ref{L:pre-interpolation} (see (\ref{Eq: pre-interpolation})). 
	The rest parts in \cite[Theorem 4.10]{pitts2014existence} are purely combinatorial, which can be easily adapted. 
\end{proof}

Given a continuous $G$-homotopy class $\bm{\Pi}\in \pi_m' \big(\Z_{n}^G(M, \partial M;\mF;\mZ_2), \{0\}\big)$, we can apply Theorem \ref{Thm:exist amv} to the sequence of discrete mappings $S=\{\varphi_i\}_{i\in\N}$ in Proposition \ref{Prop:tightening}: 

\begin{theorem}\label{Thm:exist amv in critical set}
	Let $\bm{\Pi}\in \pi_m' \big(\Z_{n}^G(M, \partial M;\mF;\mZ_2), \{0\}\big)$ be a continuous $G$-homotopy class. 
	There exists $V\in\V^G_n(M)$ such that 
	\begin{itemize}
		\item[(i)] $\|V\|(M)=\bL(\bm{\Pi})$;
		\item[(ii)] $V$ is stationary in $M$ with free boundary;
		\item[(iii)] $V$ is $(G,\mZ_2)$-almost minimizing in annuli with free boundary.
	\end{itemize}
	Moreover, if $\{\Phi_i\}_{i\in\N}$ is a min-max sequence for $\bm{\Pi}$ and $S=\{\varphi_i\}_{i\in\N}$ is a sequence of discrete mappings given by Proposition \ref{Prop:tightening}, then we can choose $V\in {\bf C}(S)\subset {\bf C}(\{\Phi_i\}_{i\in\N})$.
\end{theorem}

\begin{proof}
	Let $\{\Phi_i\}_{i\in\N}$ and $S=\{\varphi_i \}_{i\in\N}$ be given by Proposition \ref{Prop:tightening}. 
	If no element $V\in {\bf C}(S)$ is $(G,\mZ_2)$-almost minimizing in annuli with free boundary, then Theorem \ref{Thm:exist amv} gives a sequence $S^*=\{\varphi_i^*\}_{i\in\N}$ of mappings
	$\varphi_i^*: I(m, l_i)_{0}\rightarrow\big(\Z_{n}^G(M, \partial M;\mZ_2), \{0\}\big)$ so that $\varphi_i,\varphi_i^*$ are $m$-homotopic in $\big(\Z_{n}^G(M, \partial M;\mZ_2), \{0\}\big)$ with $\M$-finenesses tending to zero and ${\bf L}(S^*)<{\bf L}(S)$. 
	As we mentioned in Remark \ref{Rem:almgren iso and homo}, the arguments in \cite[Corollary 3.12]{marques2017existence} suggest that the Almgren $G$-extension $\Phi_i, \Phi_i^*$ of $\varphi_i,\varphi_i^*$ are $G$-homotopic for $i$ large enough, and thus $\Phi_i,\Phi_i^*\in {\bm \Pi}$ for $i$ large. 
	Now we have a contradiction 
	$$ {\bf L}({\bm \Pi}) \leq {\bf L}(\{\Phi_i^*\}_{i\in\N}) \leq {\bf L}(\{\varphi_i^*\}_{i\in\N}) < {\bf L}(\{\varphi_i \}_{i\in\N}) = {\bf L}({\bm \Pi})$$
	by Theorem \ref{Thm:interpolation}(i)(iii), Theorem \ref{Thm:exist amv}(ii), and Proposition \ref{Prop:tightening}(3). 
\end{proof}

\begin{remark}\label{Rem:boundary-type-amv}
	Note $\Phi\llcorner \partial I^m = 0$ for any $\Phi\in \bm{\Pi}$ and $\bm{\Pi}\in \pi_m' \big(\Z_{n}^G(M, \partial M;\mF;\mZ_2), \{0\}\big)$. 
	By Lemma \ref{Lem:F-isoperimetric}, every map $\Phi\in \bm{\Pi}$ is {\em boundary type}, i.e. there exists $Q\in\bI_{n+1}^G(M;\mZ_2) $ so that $\Phi(x)=[\partial Q]$. 
	Then we can further require $V$ in the above theorem to be $(G,\mZ_2)$-almost minimizing of {\em boundary type} in annuli with free boundary (in the sense of Definition \ref{Def:G-am-varifolds}). 
	Indeed, the explicit operations for currents in the proof of Theorem \ref{Thm: discretization} and \ref{Thm:exist amv} are based on the cut-and-paste constructions similar to the proof of Case 1 in Lemma \ref{L:pre-interpolation}, which deforms a boundary to a boundary (see (\ref{Eq: pre-interpolation})). 
	Hence, the discrete mappings in Proposition \ref{Prop:tightening} and Theorem \ref{Thm:exist amv} are all boundary type maps. 
	Since we always use the canonical representatives to induce varifolds (instead of $|\varphi_i(x)|$), the constructions in Theorem \ref{Thm:exist amv} can also be applied if `none of $V\in {\bf C}(S)$ is $(G,\mZ_2)$-almost minimizing of {\em boundary type} in annuli with free boundary'. 
	Therefore, the varifold $V$ in Theorem \ref{Thm:exist amv in critical set} shall be $(G,\mZ_2)$-almost minimizing of boundary type in annuli with free boundary (in the sense of Definition \ref{Def:G-am-varifolds}). 
\end{remark}

As we mentioned in Remark \ref{Rem:almgren iso and homo}, we have the Almgren's isomorphism for the relative $G$-cycle space: 
\begin{eqnarray}
	\mZ_2 & \cong & H_{n+1}(M, \partial M;\mZ_2) \cong  \pi^{\sharp}_1\big(\Z_{n}^G(M, \partial M; \M;\mZ_2), \{0\}\big) 
	\\
	&\cong & \pi_1\big(\Z_{n}^G(M, \partial M; \M;\mZ_2), \{0\}\big) \cong \pi_1\big(\Z_{n}^G(M, \partial M; \F; \mZ_2), \{0\}\big). \nonumber
\end{eqnarray}
Moreover, we have 
$\bL([\Phi]) > 0 $
provided $\Phi: [0,1] \to \big(\Z_{n}^G(M, \partial M;\M;\mZ_2), \{0\}\big)$ corresponds to the fundamental class $[M]\in H_{n+1}(M, \partial M;\mZ_2)$. 
By Lemma \ref{Lem:M-isoperimetric}, we also see that every map in $\bm{\Pi}_M = [\Phi]$ is a boundary type map. 
Hence we obtain the following corollary by applying Theorem \ref{Thm:exist amv in critical set} and Remark \ref{Rem:boundary-type-amv} to such $\bm{\Pi}_M$:

\begin{corollary}\label{Cor:exist amv}
	There exists a nontrivial $G$-invariant varifold $0\neq V\in\V_n^G(M)$ so that $\|V\|(M) = \bL(\bm{\Pi}_M)$, $V$ is stationary in $M$ with free boundary, and $V$ is $(G,\mZ_2)$-almost minimizing of boundary type in annuli with free boundary. 
\end{corollary}

\section{Regularity of $(G,\mZ_2)$-almost minimizing varifolds}\label{Sec:regular}

The regularity of varifolds that are almost minimizing in annuli with free boundary is built in \cite[Section 5]{li2021min}. 
In this section, we adapt the proof in \cite{li2021min} to $G$-invariant settings and show our equivariant regularity result, combining which with Corollary \ref{Cor:exist amv} gives the main Theorem \ref{Thm:main}. 
As mentioned in Remark \ref{Rem:boundary-type-amv}, we only take {\em boundary type} almost minimizing varifolds into consideration. 
Additionally, since the interior regularity was already achieved in \cite[Theorem 7]{wang2022min}, we mainly focus on the regularity of the free boundary.

\begin{theorem}[Interior regularity theorem]\label{Thm:interior regul-amv}
	Let $2\leq n\leq 6$ and ${\rm Cohom(G)}\geq 3$. 
	If $V \in \mathcal{V}^G_n(M)$ is a $G$-varifold which is 
	\begin{itemize}
		\item stationary in $M$ with free boundary and 
		\item $(G,\mathbb{Z}_2)$-almost minimizing of boundary type in annuli with free boundary.
	\end{itemize}
	Then $\spt(\|V\|)\cap (M\setminus \partial M)$ is a smooth embedded $G$-invariant minimal hypersurface $\Sigma\subset M\setminus \partial M$. 
	Moreover, there exists $\{n_i\}_{i=1}^N \subset \N$ so that 
	$$V\llcorner (M \setminus \partial M)= \sum_{i=1}^N n_i |\Sigma_i|,$$
	where $\{\Sigma_i\}_{i=1}^N$ are the $G$-connected components of $\Sigma$. 
\end{theorem}
\begin{proof}
	Note the assumption on $M\setminus M^{reg}$ in \cite[Section 6]{wang2022min} is redundant if $V\in \V^G(M)$ is $(G,\mathbb{Z}_2)$-almost minimizing of boundary type in annuli, instead of `in regular annuli'. 
	Hence, the theorem follows from the proof of \cite[Theorem 7]{wang2022min} and the constancy theorem \cite[2.4(5)]{pitts2014existence}. 
\end{proof}

We are going to show that $N$ is finite and each $\Sigma_i$ can be smoothly extended up to $\partial M$ as a hypersurface $\widetilde{\Sigma}_i$, which remains to be $G$-invariant and may have a free boundary in $\partial M$. 
Specifically, we have the following main regularity result indicating the Main Theorem \ref{Thm:main} by Corollary \ref{Cor:exist amv}.

\begin{theorem}[Main regularity]\label{Thm:main-regularity}
	Let $2\leq n\leq 6$ and ${\rm Cohom(G)}\geq 3$.  
	If $V \in \mathcal{V}^G_n(M)$ is a $G$-varifold which is 
	\begin{itemize}
		\item stationary in $M$ with free boundary, 
		\item $(G,\mathbb{Z}_2)$-almost minimizing of boundary type in annuli with free boundary,
	\end{itemize}
	then there exists $N \in \N$ and $n_i \in \N$, $i=1,\cdots,N$, such that
	$$V= \sum_{i=1}^N n_i |\Sigma_i|,$$
	where each $(\Sigma,\partial\Sigma)\stackrel{G}{\subset}(M,\partial M)$ is a smooth, compact, $G$-connected, almost properly embedded free boundary $G$-invariant minimal hypersurface.
\end{theorem}

\subsection{Good $G$-replacement property}\label{SecS:replacement}

To begin with, let us show the existence of free boundary $G$-replacements. 
Since the interior regularity has been covered in \cite{wang2022min}, we only consider the $G$-neighborhoods $U\subset M$ with $U\cap\partial M\neq\emptyset$.

\begin{proposition}[Existence of $G$-replacements]\label{Prop:good-replacement-property}
	Let $V\in\V_n^G(M)$ be $(G,\mZ_2)$-almost minimizing of boundary type in a relatively open $G$-set $U \subset M$ with free boundary, and $K \subset U$ be a compact $G$-subset. 
	Then there exists $V^{*}\in \V_n^G(M)$, called \emph{a $G$-replacement of $V$ in $K$}, such that
	\begin{enumerate}
		\item[(i)] $V\lc (M\setminus K) =V^{*}\lc (M\setminus K)$;
		\item[(ii)] $\|V\|(M)=\|V^{*}\|(M)$;
		\item[(iii)] $V^{*}$ is $(G,\mZ_2)$-almost minimizing of boundary type in $U$ with free boundary;
		\item[(iv)] $V^{*}  =\lim_{i \to \infty} |T_i^*|$ as varifolds, where $T_i^*\in Z_n^G(M,\partial M;\mZ_2)$ such that $T_i^*$ is $G$-locally mass minimizing in $\interior_M(K)$ (relative to $\partial M$), and $T_i^*= \partial Q_i^*\llcorner (M\setminus\partial M) $ for some $Q_i^*\in\bI_{n+1}^G(M;\mZ_2)$. 
	\end{enumerate}
\end{proposition}

Here $T\in Z_n^G(M, \partial M;\mZ_2)$ is said to be {\em $G$-locally mass minimizing} in $\interior_M(K)$ (relative to $\partial M$), if for any $p \in \interior_M(K)$, there exists $\rho>0$ such that 
$$\M(S) \geq \M(T),$$
 for all $S \in Z_n^G(M, \partial M;\mZ_2)$ with $S-T \in Z_n^G(A, B;\mZ_2)$, where 
 \begin{equation}\label{Eq: tube in M or partial M}
 	A = \left\{ \begin{array}{cl}
	\tBcal_\rho^{G}(p)  & \text{ if } p\in M\setminus\partial M, \\
	\tBcal_\rho^{G,+}(p)  & \text{ if } p \in \partial M, 
	\end{array} \right. \quad A\subset\subset \interior_M(K), \qquad B=A\cap\partial M.
 \end{equation}
Moreover, if $\M(S) \geq \M(T)$ for all $S \in Z_n(M, \partial M;\mZ_2)$ with $S-T \in Z_n(A, B;\mZ_2)$ ($S$ may not be $G$-invariant), then we say $T$ is {\em locally mass minimizing} in $\interior_M(K)$ (relative to $\partial M$). 

\begin{proof}
	Let $\epsilon,\delta>0$ be sufficiently small. 
	Given any $Q\in \bI_{n+1}^G(M;\mZ_2)$ and $\tau=[\partial Q]\in \mathfrak{a}^G_n(U;\epsilon,\delta;\F )$, define $\mathcal{C}_\tau$ to be the set of all $\sigma\in \Z_n^G(M,\partial M;\mathbb{Z}_2)$ such that there exists a finite sequence $\{\tau_i\}_{i=1}^q\subset \Z_{n}^G(M,\partial M;\mathbb{Z}_2)$ satisfying
	\begin{itemize}
		\item $\tau_0=\tau$, $\tau_q=\sigma$, and $\mbox{spt}(\tau-\tau_i)\subset K$ for all $i=1,\dots, q$;
		\item $\mathcal{F}(\tau_i-\tau_{i-1})\leq \delta$ for all $i=1,\dots, q$;
		\item ${\bf M}(\tau_i)\leq {\bf M}(\tau)+\delta$ for all $i=1,\dots, q$.
	\end{itemize}
	By the $G$-invariant $\F$-isoperimetric Lemma \ref{Lem:F-isoperimetric}, we can also assume every $\sigma\in \mathcal{C}_\tau$ is boundary type, i.e. $\sigma$ has a representative $\partial P$ for some $P\in\bI_{n+1}^G(M;\mZ_2)$. 
	
	Since $\Z_n^G(M,\partial M;\mathbb{Z}_2)$ is a closed subspace of $\Z_n(M,\partial M;\mathbb{Z}_2)$, the following three claims can be obtained by the same procedure as in the proof of \cite[Proposition 5.3, Step 1]{li2021min}. 
	Indeed, we can use lemmas in Section \ref{Sec:relative G-cycle} in place of those in \cite[Section 3.1]{li2021min}, and only take $G$-cycles and $G$-sets (like $\tBcal^{G,+}_r(p)$) into consideration. 
	Thus, we omit the proof: 
	
	\begin{claim}\label{Claim: replace 1}
		There exists $\tau^* \in \C_{\tau}$ such that 
		$\M(\tau^*)=\inf\{\M(\si):\ \si\in\C_{\tau}\}. $
	\end{claim}
	\begin{claim}\label{Claim: replace 2}
		 The canonical representative $T^* \in \tau^*$ is $G$-locally mass minimizing in $\interior_M(K)$.
	\end{claim}
	\begin{claim}\label{Claim: replce 3}
		$\tau^*\in \mathfrak{a}^G_n(U;\epsilon,\delta;\F )$.
	\end{claim}
	
	Now, since $V\in \V_n^G(M)$ is $(G,\mZ_2)$-almost minimizing of boundary type in $U$ with free boundary, there exist a sequence $Q_i\in\bI_{n+1}^G(M;\mZ_2)$ and $\tau_i=[\partial Q_i]\in \mathfrak{a}_n^G(U; \ep_i, \de_i; \F)$ with $\ep_i, \de_i \to 0$, such that $V=\lim_{i\to\infty}|T_i|$, where $T_i \in \tau_i$ is the canonical representative. 
	By Claim \ref{Claim: replace 1}, there is a mass minimizer $\tau_i^* \in \C_{\tau_i}$ for each $i$ sufficiently large, so that $\tau_i^*=[\partial Q_i^*]$ for some $Q_i^*\in\bI_{n+1}^G(M;\mZ_2)$. 
	Denote $T_i^* \in \tau_i^*$ as the canonical representative. 
	Noting $\M(T_i^*)=\M(\tau_i^*)\leq \M(\tau_i)\leq 2\|V \|(M)$, the compactness theorem for $G$-varifolds implies that $|T_i^*|$ converges as $G$-varifolds to some $V^* \in \V_n^G(M)$ up to a subsequence. 
	Hence, combining the above Claim \ref{Claim: replace 1},\ref{Claim: replace 2},\ref{Claim: replce 3}, with the arguments in \cite[Proposition 5.3, Step2]{li2021min}, one verifies that $V^*$, $T_i^*$, and $Q_i^*$ satisfy all the requirements (i)-(iv). 
\end{proof}

The above proposition is parallel to \cite[Proposition 5.3]{li2021min}. 
However, $T_i^*$ is only a mass minimizer under local {\em $G$-equivariant} variations. 
In the following proposition, we are going to show that $T_i^*$ is locally mass minimizing and thus has good regularity.
The proof is adapted from an elegant argument of Fleming (c.f. Lawson's work \cite[Theorem 1]{lawson1972equivariant}).

\begin{proposition}[Existence of $G$-replacements II]\label{Prop:good-replace-2}
	The $G$-currents $T_i^* \in Z_n^G(M,\partial M;\mathbb{Z}_2)$ obtained in Proposition \ref{Prop:good-replacement-property} (iv) are locally mass minimizing in $\interior_M(K)$ (relative to $\partial M$). 
\end{proposition}
\begin{proof}
	For any $\epsilon,\delta>0$ small enough, suppose $\tau=[\partial Q]\in \mathfrak{a}^G_n(U;\epsilon,\delta;\F ) $ for some $Q\in {\bf I}^G_{n+1}(M;\mathbb{Z}_2)$. 
	Let $\C_{\tau}$ be defined as in the proof of Proposition \ref{Prop:good-replacement-property} and $\tau^*=[\partial Q^*]\in\C_{\tau}$ be given in Claim \ref{Claim: replace 1} for some $Q^*\in {\bf I}^G_{n+1}(M;\mathbb{Z}_2)$. 
	Thus, the canonical representative $T^*$ of $\tau^*$ is $G$-locally mass minimizing in $\interior_{M}K$ by Claim \ref{Claim: replace 2}. 
	Hence, for any $p\in \interior_M(K)$, there exists $\rho>0$ such that $\|T^*\|(A)<\delta/2$ and $\M(S) \geq \M(T^*)$ for any $S \in Z_n^G(M, \partial M;\mZ_2)$ with $S-T^* \in Z_n^G(A, B;\mZ_2)$, where $A,B$ are defined in (\ref{Eq: tube in M or partial M}) with such $\rho>0$.

	Suppose there exists a current $S\in Z_n(M, \partial M;\mZ_2)$ (not $G$-invariant) with $S-T^* \in Z_n(A, B;\mZ_2)$ and $\M(S)<\M(T^*)$. 
	Denote $\sigma^*=[S]\in \Z_n(M, \partial M;\mZ_2)$. 
	Thus, $\M(\tau^*-\sigma^*)\leq \|T^*-S\|(M)\leq 2\|T^*\|(A)<\delta$. 
	Since $\delta>0$ is sufficiently small, we have $\sigma^*=[\partial P]$ for some $P\in\bI_{n+1}(M;\mZ_2)$ with $\spt(P-Q^*)\subset A$ 
	by the $\M$-isoperimetric lemma \cite[Lemma 3.17]{li2021min}. 
	
	Let $S^*=\partial P\llcorner (M\setminus \partial M)$ be the canonical representative of $\sigma^*$. 
	Define $f_{P}:M\to [0,1]$ as 
	\begin{eqnarray*}
		 f_P(x) := \int_G 1_{\spt(P)}(g\cdot x)~d\mu(g). 
	\end{eqnarray*}
	Since $\spt(P)$ is closed, the function $ 1_{\spt(P)}$ is upper-semicontinuous. 
	By Fatou's Lemma, $f_P$ is upper-semicontinuous. 
	Thus, 
	\begin{equation*}
			P_\lambda := f_P^{-1}[\lambda,1 ]\subset M 
	\end{equation*}
	are $G$-invariant closed sets in $M$. 

	Define $E_{f_P} := f_P\cdot [[M]]$, where $[[M]]=\underline{\tau}(M,1,\xi)$ is the integral current induced by $M$ and $\xi$ is an orientation (see the notations in \cite[27.1]{simon1983lectures}). 
	A direct computation shows that 
	\begin{eqnarray}\label{Eq-boundary-Ef}
		\partial E_{f_P}(\omega) &=& \int_{M}\int_G \langle d\omega, \xi\rangle ~{\rm 1}_{\mbox{spt}(P)}(g\cdot x)~d\mu(g)d\mathcal{H}^{n+1}(x) \nonumber
		\\
		&=& \int_G\int_{M} \langle d\omega, \xi\rangle ~{\rm 1}_{g^{-1}(\mbox{spt}(P))}(x)~d\mathcal{H}^{n+1}(x)d\mu(g)
		\\
		&=& \int_G \partial ((g^{-1})_\# P)(\omega) ~d\mu(g), \nonumber
	\end{eqnarray}
	for any $n$-form $\omega$ on $M$. 
	The lower semi-continuity of mass implies 
	\begin{eqnarray}
		\M(\partial E_{f_P}) &\leq & \int_G \M(\partial ((g^{-1})_\# P)) = \M(\partial P), \nonumber
		\\
		\| \partial E_{f_P}\|(M \setminus\partial M ) &\leq & \int_G \| \partial ((g^{-1})_\# P)\|(M \setminus\partial M )~d\mu(g) = \| \partial P\|(M \setminus\partial M ), \nonumber
	\end{eqnarray}
	where the equalities come from the fact that $G$ acts by isometries. 
	Additionally, it is clear that $\M(E_{f_P})\leq \M([[M]])$, which implies that $E_{f_P}$ is a normal current. 
	By \cite[4.5.9(12),(13)]{federer2014geometric}, we have $P_\lambda\in\mI_{n+1}^G(M;\mZ_2) $ for almost all $\lambda\in[0,1]$, and 
	\begin{equation}\label{Eq:mass-1}
		\int_0^1 \| \partial P_\lambda\|(M\setminus\partial M) = \| \partial E_{f_P}\|(M\setminus\partial M) \leq \|\partial P\|(M\setminus\partial M).
	\end{equation}
	
	Define $S_\lambda := \partial P_\lambda \llcorner (M\setminus\partial M)$. 
	Since $Q^*=P=P_\lambda$ outside $A$, we have $S_\lambda\in Z_n^G(M, \partial M ; \mZ_2)$ and $S_\lambda -T^* \in Z_n^G(A, B;\mZ_2)$ for almost all $\lambda\in [0,1]$. 
	By (\ref{Eq:mass-1}), the following inequality holds:
	\begin{eqnarray*}
		&& \int_0^1 \M(S_\lambda)~d\lambda \leq  \|\partial P\|(M\setminus\partial M) 
		= \M(S^*) ,
	\end{eqnarray*}
	which implies $\M(S_\lambda)\leq \M(S^*)$ for $\lambda$ in a positive measure set of $[0,1]$. 
	Hence, there exists $\lambda_0\in[0,1]$ so that $S_{\lambda_0}\in Z_n^G(M, \partial M ; \mZ_2)$, and 
	\begin{equation*}
		\M(S) \geq \M(\sigma^*) = \M(S^*)\geq \M(S_{\lambda_0})\geq \M(T^*),
	\end{equation*}
	which is a contradiction to the choice of $S$. 
\end{proof}

Combining the regularity of locally mass minimizing $\mZ_2$-currents with the compactness theorem, we can get the following regularity result for $G$-replacements.

\begin{lemma}[Regularity of $G$-replacement]\label{Lem:regular replacement}
	Suppose $2 \leq n \leq 6$. 
	Under the same hypotheses of Proposition \ref{Prop:good-replacement-property}, the restriction of the $G$-replacement $V^* \lc \interior_M(K)$ is integer rectifiable. 
	Denote $\Sigma:=\spt \|V^*\| \cap \interior_M(K)$. 
	Then $(\Sigma,\partial \Sigma) \stackrel{G}{\subset} \big( \interior_M(K),\interior_M(K) \cap \partial M \big)$ is a smooth, almost properly embedded, free boundary minimal $G$-hypersurface which is stable (\cite[Definition 2.10]{li2021min}) in any simply connected relative open subset $U'\subset \interior_MK$.
\end{lemma}

\begin{proof}
	By \cite[Theorem 4.7]{guang2021min}, we have the regularity of locally area minimizing currents in codimension one with $\mZ_2$-coefficients. 
	Thus, Proposition \ref{Prop:good-replace-2} implies the $G$-currents $\{T_i^*\}_{i\in\N}$ in Proposition \ref{Prop:good-replacement-property} (iv) are smooth, properly embedded, free boundary minimal $G$-hypersurfaces in $\interior_M(K)$. 
	
	\begin{claim}
		$T_i^*$ is $G$-stable in $\interior_M(K)$ in the sense of Definition \ref{Def: G-stability}.
	\end{claim}
	\begin{proof}
		Since $T_i^*$ is a part of the boundary of $Q_i^*\in {\bf I}^G_{n+1}(M;\mathbb{Z}_2)$, we see ${\rm spt}(T_i^*)\cap\interior_M(K)\subset\partial ({\rm spt}(Q^*_i)\cap \interior_M(K))$ has a $G$-invariant outward pointing unit normal vector field. 
		Suppose $T_i^*$ is not $G$-stable in $\interior_M(K)$.
		Then there exists $X_i\in\mathfrak{X}^G(M,\spt(T_i^*)\cap \interior_M(K))$ with compact support in $\interior_M(K)$ such that 
		\begin{itemize}
			\item $\{T_i^t = (f^{X_i}_t)_\#T_i^*\}_{t\in(-t_0,t_0)}$ are smooth properly embedded $G$-hypersurfaces in $\interior_M(K)$, 
			\item $T_i^0=T_i^*$, and $T_i^t=T_i^*$ outside $\interior_M(K)$ for all $t\in(-t_0,t_0)$,
			\item $\M(T_i^t) < \M(T_i^*)$ for all $t \in (-t_0, 0)\cup (0,t_0)$,
		\end{itemize}
		where $\{f^{X_i}_t\}_{t\in(-t_0,t_0)}$ are the equivariant diffeomorphisms generated by $X_i$. 
		Note the first bullet comes from the facts that $X_i\in \mathfrak{X}^G(M,\spt(T_i^*)\cap \interior_M(K))$ and $T_i^*$ is properly embedded in $\interior_M(K)$. 
		Denote $\tau_i^t := [T_i^t]\in\Z_n^G(M,\partial M;\mZ_2)$ with canonical representative $T_i^t$ (since $T_i^t$ is properly embedded in $\interior_M(K)$ and $T_i^*$ is canonical). 
		Then for $t>0$ sufficiently small, we have 
		\begin{itemize}
			\item $\spt(\tau_i^t-\tau_i^*)\subset \spt(T_i^t-T_i^*)\subset K$,
			\item $\F(\tau_i^t-\tau_i^*)\leq\F(T_i^t-T_i^*)\leq \delta_i$,
			\item $\M(\tau_i^t)=\M(T_i^t)<\M(T_i^*)=\M(\tau_i^*)$.
		\end{itemize}
		It is clear that $\tau_i^t\in\C_{\tau_i^*}\subset \C_{\tau_i}$, where $\tau_i$, $\C_{\tau_i}$, are defined in the proof of Proposition \ref{Prop:good-replacement-property}. 
		This contradicts the minimality of $\M(\tau_i^*)$ in Claim \ref{Claim: replace 1}. 
	\end{proof}
	Finally, by the compactness theorem (Theorem \ref{Thm:compactness $G$-stable hypersurface}), the limit of $|T^*_i|$ (up to a subsequence) is a smooth, almost properly embedded, minimal $G$-hypersurface in $\interior_M(K)$, which is also stable (in the sense of \cite[Definition 2.10]{li2021min}) in any simply connected relative open subset of $\interior_M(K)$. 
\end{proof}

\subsection{Tangent cones and rectifiability} 
In this subsection, we use the $G$-replacements to show the rectifiability of the min-max $G$-varifold $V$ in Corollary \ref{Cor:exist amv} and classify its tangent cones. 
The following results are parallel to those in \cite[Section 5.2]{li2021min} (for compact manifolds with boundary) and \cite[Section 6.2]{wang2022min} (for closed $G$-manifolds).

\begin{lemma}[Uniform volume ratio bound]\label{Lem: volum ratio bound}
	Let $2\leq n\leq 6$. 
	Suppose $V \in \mathcal{V}^G_n(M)$ is $(G,\mathbb{Z}_2)$-almost minimizing of boundary type in annuli with free boundary, and is stationary in $M$ with free boundary. 
	There exists a constant $c=c(M)>1$ such that
	\begin{equation}\label{Eq: volume ratio bound}
		c^{-1}\leq \frac{\|V\|(B_\rho(p))}{\omega_n\rho^n}\leq c\cdot \frac{\|V\|(B_{\rho_0}(p))}{\omega_n\rho_0^n}
	\end{equation}
	 for all $p\in {\rm spt}(\|V\|)\cap \partial M,~\rho\in(0,\rho_0)$, where $\rho_0=\rho_0(M) $. 
	 
	 Moreover, $\Theta^n(\|V\|,p)\geq\theta_0>0$, for some constant $\theta_0>0$ and all $p\in {\rm spt}(\|V\|)\cap \partial M$.
\end{lemma}
\begin{proof}
	We only need to show the uniform lower bound of $\Theta^n(\|V\|, p)$ for $p\in {\rm spt}(\|V\|)\cap \partial M$, and (\ref{Eq: volume ratio bound}) then follows from the monotonicity formula \cite[Theorem 2.3]{li2021min}. 
	
	Fix any $p\in {\rm spt}(\|V\|)\cap \partial M$. 
	Let $r_{am}(G\cdot p)>0$ be defined in Definition \ref{Def:am-annuli}, $r_{\textrm{Fermi}}^G(p)>0$ be defined in Lemma \ref{Lem:Fermi convex}, and $r_{mono}>0$ be defined in \cite[Theorem 2.3]{li2021min}. 
	Take $r_0=\frac{1}{20}\min\{r_{am}(G\cdot p),r_{\textrm{Fermi}}^G(p),r_{mono}\} $.  
	Therefore, for any $r\in (0,r_0)$, we can apply Proposition \ref{Prop:good-replacement-property} to get a $G$-replacement $V^*$ of $V$ in $\Clos(A_{r,2r}^G(p))\cap M$. 
	Combining the maximum principle \cite[Theorem 2.5]{li2021min} with the mean-convexity of Fermi half-tubes $\tBcal^{G,+}_s(p)$ (Lemma \ref{Lem:Fermi convex}) for $s\in(0,r)$, we have 
	\begin{equation}\label{Eq:non-empty}
		\|V^*\|\llcorner A_{r,2r}^G(p)\neq 0. 
	\end{equation}
	Thus, $\Sigma:=\spt(\|V^*\|)\cap A_{r,2r}^G(p)\neq \emptyset$ is a smooth, almost properly embedded, free boundary, $G$-invariant minimal hypersurface in $\big( A_{r,2r}^G(p)\cap M, A_{r,2r}^G(p)\cap \partial M \big)$. 
	 
	 By the five-times covering lemma, there is a finite set $\{p_i\}_{i=1}^N\subset G\cdot p$ with
	$$B_{4r}^G(p)\subset \bigcup_{i=1}^NB_{20r}(p_i),\quad B_{4r}(p_i)\cap B_{4r}(p_j)=\emptyset,~i\neq j\in\{1,\dots,N\}.$$
	Moreover, we have $\|V\|(B_{20r}(p))=\|V\|(B_{20r}(p_i))$ since $V$ is $G$-invariant and $G$ acts by isometries. 
	Together with Proposition \ref{Prop:good-replacement-property}(i)(ii), we have
	\begin{eqnarray}\label{Eq:covering}
		N\|V\|(B_{20r}(p)) &\geq & \|V\|(\bigcup_{i=1}^NB_{20r}(p_i))\nonumber
		\\
		&\geq & \|V\|(B_{4r}^G(p))= \|V^*\|(B_{4r}^G(p))\geq N\|V^*\|(B_{4r}(p)),
	\end{eqnarray}
	which implies 
	$$\|V\|(B_{20r}(p))\geq \|V^*\|(B_{4r}(p))\geq \|V^*\|(B_{2r}(q)), $$ for all $q\in A_{r,2r}(p)$. 
	
	If $\partial\Sigma\neq\emptyset$, there exists $q\in \partial \Sigma\subset A_{r,2r}^G(p)\cap\partial M$ so that $\Theta^n(\|V^*\|,q)\geq 1/2$ by Lemma \ref{Lem:regular replacement}. 
	We can also take $q\in A_{r,2r}(p)$, since $\Theta^n(\|V^*\|,g\cdot q) = \Theta^n(\|V^*\|, q)$ for all $g\in G$.
	Then the monotonicity formula \cite[Theorem 2.3]{li2021min} for free boundary stationary varifolds implies 
	\begin{eqnarray}\label{Eq: volume ratiou}
		\frac{\|V\|(B_{20r}(p))}{\omega_n(20r)^n} &\geq &\frac{\|V^*\|(B_{2r}(q))}{\omega_n(20r)^n} 
		\geq  \frac{1}{10^nC_{mono}}\lim_{s\to 0}\frac{\|V^*\|(B_{s}(q))}{\omega_n s^n}
		\\
		&\geq & \frac{\Theta^n(\|V^*\|,q)}{10^nC_{mono}} =\frac{1}{2\times10^nC_{mono}} \nonumber .
	\end{eqnarray}
	If $\partial\Sigma=\emptyset$, then $\Sigma$ is stationary in $A_{r,2r}^G(p)\cap\widetilde{M} $, and the arguments in \cite[Lemma 10]{wang2022min} would carry over to give an inequality similar to (\ref{Eq: volume ratiou}).
	Thus, by taking $r\to 0$ in (\ref{Eq: volume ratiou}), there is a uniform lower bound on the density $\Theta^n(\|V\|,p)$ for all $p\in \spt(\|V\|)\cap \partial M$. 
\end{proof}

As mentioned in \cite[Remark 5.7]{li2021min}, Lemma \ref{Lem: volum ratio bound} does not imply the rectifiability of $V$, since we do not know whether $V$ has locally bounded first variations as a varifold in $\R^L$.

\begin{lemma}\label{Lem: tangent cone}
	Let $2\leq n\leq 6$. 
	Suppose $V \in \mathcal{V}^G_n(M)$ is $(G,\mathbb{Z}_2)$-almost minimizing of boundary type in annuli with free boundary, and is stationary in $M$ with free boundary. 
	For any $p\in{\rm spt}(\|V\|)\cap \partial M$ and $C\in{\rm VarTan}(V,p)$, we have: 
	\begin{itemize}
		\item[(i)] $C\in \mathcal{V}_n(T_pM)$ is rectifiable;
		\item[(ii)] $C$ is stationary in $T_pM$ with free boundary;
		\item[(iii)] $C$ is a cone and has the form $C=T_pG\cdot p\times W$, where $W\in \mathcal{V}_{n-{\rm dim}(G\cdot p)}({\bf N}_pG\cdot p)$. 
	\end{itemize}
\end{lemma}

\begin{proof}
	Let $r_i\rightarrow 0$ be a sequence such that $ C=\lim_{i\to \infty} V_i$, where $V_i={\bm \eta}_{p,r_i\#} V$. 
	Clearly, $C$ is stationary in $T_pM$ with free boundary. 
	By the arguments in \cite[Lemma 5.8 Claim 1,2]{li2021min} with Lemma \ref{Lem: volum ratio bound}, Theorem \ref{Thm:interior regul-amv}, \cite[(17)]{wang2022min} in place of \cite[Lemma 5.6, Theorem 5.1]{li2021min}, \cite[(5.8)]{schoen1981regularity}, we can also get: 
	\begin{equation}\label{Claim:Hausdorff converge}
		\mbox{${\rm spt}(\|V_i\|)$ converges to ${\rm spt}(\|C\|)$ in the Hausdorff topology,}
	\end{equation}
	and $\Theta^n(\|C\|,q)\geq \theta_0>0$, for all $q\in{\rm spt}(\|C\|)$.
	Moreover, the doubled arguments in \cite[Lemma 5.8]{li2021min} also imply $C$ is a rectifiable cone in $T_pM$.
	
	Fix $w\in T_p(G\cdot p)$. 
	Let $g(t)\subset G$ be a curve with $g(0)=e$ and $w=\frac{d}{dt}\big\vert_{t=0} g(t)\cdot p$.
	Hence, there exists a continuous curve $A(t)$ in the $L$-matrices space so that $g(t)=I_L+tA(t)$, where $I_L$ is the identity.
	This implies $A(0)\cdot p =w$. 
	One verifies that
	\begin{eqnarray}\label{Eq:splitting}
		(\bleta_{p, r_i}\circ g(r_i))(x) 
		&=& \frac{g(r_i)\cdot x - g(r_i)\cdot p}{r_i}+\frac{g(r_i)\cdot p -  p}{r_i} \nonumber
		\\
		&=& (\btau_{-A(r_i)\cdot p}\circ g(r_i)\circ \bleta_{p, r_i})(x).
	\end{eqnarray}
	Noting $V$ is $G$-invariant, $A(r_i)\cdot p\to w$, and $g(r_i)\to e$, we have $C=(\btau_{-w})_\# C$, which implies $C$ is invariant under the translation along $T_p(G\cdot p)$. 
	Since $T_pM=T_p(G\cdot p)\times {\bf N}_p(G\cdot p)$ and $C$ is rectifiable, we have $C=T_p(G\cdot p)\times W$ for some $W\in \mathcal{V}_{n-{\rm dim}(G\cdot p)}({\bf N}_pG\cdot p)$. 
\end{proof}

\begin{remark}\label{Rem: tangent-cone}
	Suppose $V_i \in \V_n^G(M)$ is a sequence of $G$-varifolds that are $(G,\mZ_2)$-almost minimizing of boundary type in annuli with free boundary, stationary in $M$ with free boundary, and satisfy the uniform density bound (\ref{Eq: volume ratio bound}). 
	Then, the above arguments also show that (i)(ii) and the splitting property in (iii) are valid for the varifold limit $\lim_{i\to\infty}(\bleta_{p,r_i})_\sharp V_i$. 
\end{remark}

Before we show the classification of tangent cones, let us introduce the following notations.

\begin{lemma}\label{Lem: part of tube}
	For any $p\in M$ and $0<s<r$, denote 
	\begin{eqnarray}\label{Eq: part tube for blow up}
		&&T(p ,r) :=  \left\{ x\in B^G_r(p) \left| \begin{array}{c} {\rm there~exists~a~unique~point~} y\in G\cdot p ~{\rm with}~\\
		d_{\R^L}(x,y)=d_{\R^L}(x,G\cdot p)<r {\rm ~and~}d_{\R^L}(y,p)<r \end{array} \right.\right\}, 
		\\
		&&T(p,s,r) :=T(p,r)\cap A^G_{s,r}(p),\nonumber
	\end{eqnarray}
	as a part of the $G$-tube $B^G_r(p)$ and $A^G_{s,r}(p)$ around $G\cdot p$ in $\R^L$ respectively, which are well-defined open sets for sufficiently small $r>s>0$. 
	Then, for $r$ small enough and $\alpha\in(0,1)$,
	\begin{itemize} 
		\item $T(p,s,r)\cap M$ is a simply connected relative open subset of $M$;
		\item ${\bm \eta}_{p,r_i}( T(p,\alpha r_i,r_i)\cap M )\to  B^{T_pG\cdot p}_1(0)\times A^{{\bf N}_pG\cdot p}_{\alpha,1}(0)$ as $r_i\to 0$, where $B^{T_pG\cdot p}_1(0)\subset T_pG\cdot p$ is an open ball, $A^{{\bf N}_pG\cdot p}_{\alpha,1}(0)\subset {\bf N}_pG\cdot p $ is an open annulus. 
	\end{itemize}
\end{lemma}
\begin{proof}
	Let $S=\{q\in \R^L: d_{\R^L}(q,p)> d_{\R^L}(q,g\cdot p),\forall g\cdot p\neq p, g\in G\}$ be a slice of $G\cdot p$ at $p$ in $\R^L$. 
	Then for $r>0$ sufficiently small, since $S$ meets orbits transversally, $S\cap M$ is a smooth manifold of dimension $n+1-\dim(G\cdot p)$, and 
	$$S_r:= S\cap M\cap B^G_r(p) = S\cap M\cap B_r(p) \cong \mathbb{B}_r := \left\{ \begin{array}{cl}
	\mathbb{B}^{n+1-\dim(G\cdot p)}_r(0)  & \text{ if } p \notin  \partial M, \\
	\mathbb{B}^{n+1-\dim(G\cdot p)}_{r,+}(0)  & \text{ if } p \in \partial M, 
	\end{array} \right. $$
	where $\mathbb{B}^{n+1-\dim(G\cdot p)}_r(0)$ and $\mathbb{B}^{n+1-\dim(G\cdot p)}_{r,+}(0)$ are the $(n+1-\dim(G\cdot p))$-dimensional Euclidean open ball and half-ball. 
	Define the $G_p$ action on $G\times S_r$ as $h\cdot (g,q) = (gh^{-1}, hq)$, where $h\in G_p, g\in G, q\in S_r$. 
	Then $G\times_{G_p} S_r := G\times S_r / G_p$ is a fiber bundle over $G/G_p\cong G\cdot p$ with fiber $S_r$ (see \cite[Section 2.2]{berndt2016submanifolds}). 
	Define the action of $G$ on $G\times_{G_p} S_r$ by $h\cdot [g,q] = [hg, q]$, where $h,g\in G, q\in S_r$. 
	As in \cite[Section 2.2]{berndt2016submanifolds}, we have $G\times_{G_p} S_r$ is $G$-equivariant diffeomorphic to $M\cap B^G_r(p)$ by $[g,q]\mapsto g\cdot q$. 
	Using the local trivialization of the fiber bundle $G\times_{G_p} S_r$, we have the following diffeomorphisms: 
	$$T(p,r)\cap M \cong (B_r(p)\cap G\cdot p) \times S_r \cong \mathbb{B}^{\dim(G\cdot p)}_r(0)\times \mathbb{B}_r,$$
	for $r>0$ sufficiently small. 
	Additionally, for $0<s<r$, $S_r\setminus \Clos(S_s) = S\cap M\cap A^G_{s,r}(p)$ is diffeomorphic to $\mathbb{B}_r\setminus\Clos(\mathbb{B}_s)$, which is either an $(n+1-\dim(G\cdot p)) $-dimensional open annulus or a half-annulus. 
	Then the open $G$-subset $G\times_{G_p} (S_r \setminus \Clos(S_s)) \subset G\times_{G_p} S_r $ is also a fiber bundle and is equivariantly diffeomorphic to $ M\cap A^G_{s,r}(p) $. 
	The local trivialization gives
	$$T(p,s,r)\cap M \cong (B_r(p)\cap G\cdot p) \times (S_r\setminus \Clos(S_s)) \cong \mathbb{B}^{\dim(G\cdot p)}_r(0)\times (\mathbb{B}_r\setminus\Clos(\mathbb{B}_s)),$$
	which is simply connected since $\dim(\mathbb{B}_r\setminus\Clos(\mathbb{B}_s)) = n+1-\dim(G\cdot p)\geq{\rm Cohom}(G)\geq 3$. 
	
	Let $\{r_i\}_{i\in\N}$ be a sequence with $r_i>0$ and $r_i\to 0$ as $i\to\infty$. 
	After passing to a subsequence, we have the following locally uniformly convergences as $i\to\infty$: 
	\begin{equation*}
		{\bm \eta}_{p,r_i}(G\cdot p)\to T_pG\cdot p,\qquad {\bm \eta}_{p,r_i}(S\cap M)\to T_{p}(S\cap M) = {\bf N}_pG\cdot p.
	\end{equation*}
	Hence, $\lim_{i\to\infty} {\bm \eta}_{p,r_i}(G\cdot p\cap B_{r_i}(p)) = B^{T_pG\cdot p}_1(0)$ and 
	$$\lim_{i\to\infty}{\bm \eta}_{p,r_i}(S\cap T(p,\alpha r_i,r_i)\cap M) = \lim_{i\to\infty}{\bm \eta}_{p,r_i}(S \cap A_{\alpha r_i, r_i}(p) \cap M) = A^{{\bf N}_pG\cdot p}_{\alpha,1}(0).$$
	Therefore, for any $w\in B^{T_pG\cdot p}_1(0)$, there exists a sequence $\{g_i\}_{i\in\N}\subset G$ so that $g_i\cdot p\in B_{r_i}(p)$ and $\lim_{i\to\infty}{\bm \eta}_{p,r_i}(g_i\cdot p)=w $. 
	By the compactness of $G$ and the fact that $g_i\cdot p\to p$, we have $\lim_{i\to\infty}g_i =  g_0\in G_p $ up to a subsequence. 
	Noting $g_i\cdot p\in B_{r_i}(p)$, we have $(g_i\cdot S)\cap T(p,\alpha r_i,r_i)=g_i\cdot(S\cap T(p,\alpha r_i,r_i))$ by the definitions of $S$ and $T(p,s,r)$. 
	Hence, we can use (\ref{Eq:splitting}) to show that 
	$$\lim_{i\to\infty}{\bm \eta}_{p,r_i}(g_i\cdot S\cap T(p,\alpha r_i,r_i)\cap M) = \btau_{-w}\circ g_0(A^{{\bf N}_pG\cdot p}_{\alpha,1}(0))=\btau_{-w}(A^{{\bf N}_pG\cdot p}_{\alpha,1}(0)).$$
	Since $w\in B^{T_pG\cdot p}_1(0)$ is arbitrary, $\lim_{i\to\infty}{\bm \eta}_{p,r_i}(T(p,\alpha r_i,r_i)\cap M) \supset B^{T_pG\cdot p}_1(0)\times A^{{\bf N}_pG\cdot p}_{\alpha,1}(0) .$
	Moreover, because $T(p,\alpha r_i,r_i) = \cup_{\{g\in G:~ g\cdot p\in B_{r_i}(p)\}} g\cdot (S\cap T(p,\alpha r_i,r_i)) $, the above arguments also imply the reverse containment $\lim_{i\to\infty}{\bm \eta}_{p,r_i}(T(p,\alpha r_i,r_i)\cap M) \subset B^{T_pG\cdot p}_1(0)\times A^{{\bf N}_pG\cdot p}_{\alpha,1}(0) $, which gives the lemma. 
\end{proof}

Next, we show the classification of tangent cones and the rectifiability of the min-max $G$-varifolds in Corollary \ref{Cor:exist amv}. 

\begin{proposition}\label{Prop:classification tangent cone}
	Let $2\leq n\leq 6$. 
	Suppose $V \in \mathcal{V}^G_n(M)$ is $(G,\mathbb{Z}_2)$-almost minimizing of boundary type in annuli with free boundary, and is stationary in $M$ with free boundary. 
	For any $p\in{\rm spt}(\|V\|)\cap \partial M$ and $C\in{\rm VarTan}(V,p)$, we have: 
	\begin{itemize}
		\item[(i)] either $C=\Theta^n(\|V\|,p) \; |T_p(\partial M)|$ where $\Theta^n(\|V\|,p) \in \mathbb{N}$, 
		\item[(ii)] or $C=2 \Theta^n (\|V\|,p) \; |S \cap T_pM|$ for some $n$-plane $S \subset T_p \tM$ such that $S \perp T_p (\partial M)$, and $2 \Theta^n(\|V\|,p) \in \mathbb{N}$.
	\end{itemize}
	Moreover, for $\|V\|$-a.e. $p \in \partial M$, the tangent varifold of $V$ at $p$ is unique, and the set of $p \in \partial M$ in which case (ii) occurs as its unique tangent cone has $\|V\|$-measure zero. Hence, $V$ is rectifiable. 
\end{proposition}

\begin{proof}
	Let $r_i\rightarrow 0$ be a sequence with $ C=\lim_{i\to \infty} V_i$, where $V_i={\bm \eta}_{p,r_i\#} V$. 
	We also assume $r_i < \frac{1}{20}\min\{r_{am}(G\cdot p),r_{\textrm{Fermi}}^G(p),r_{mono}\}$. 
	Fix any $\alpha\in(0,\frac{1}{2})$. 
	
	By Proposition \ref{Prop:good-replacement-property}, there exists a $G$-replacement $V_i^*$ of $V$ in $K_i=\Clos(A_{\alpha r_i,r_i}^G(p))\cap M$. 
	As in (\ref{Eq:non-empty}), $\|V_i^*\|\llcorner\interior_M (K_i)\neq 0 $ by the maximum principle \cite[Theorem 2.5]{li2021min}. 
	By Proposition \ref{Prop:good-replacement-property}(ii) and the compactness theorem, we have a varifold limit (up to a subsequence) 
	$$ D:= \lim_{i\to\infty} {\bm \eta}_{p,r_i\#}V_i^*\in\mathcal{V}_n(T_pM).$$
	Moreover, by Remark \ref{Rem: tangent-cone}, $D$ is rectifiable stationary in $T_pM$ with free boundary and also has the form $D=T_pG\cdot p\times W'$ for some $W'\in \mathcal{V}_{n-{\rm dim}(G\cdot p)}({\bf N}_pG\cdot p)$.

	Define $\Sigma_i^*:={\rm spt}(\|V_i^*\|)\cap A_{\alpha r_i,r_i}^G(p)$, 
	and consider $\Sigma_i={\bm \eta}_{p,r_i\#}(\Sigma_i^*\cap T(p,\alpha r_i,r_i))$ (see the notations in (\ref{Eq: part tube for blow up})). 
	By Lemma \ref{Lem:regular replacement} and Lemma \ref{Lem: part of tube}, we have
	$$(\Sigma_i,\partial\Sigma_i )\subset ({\bm \eta}_{p,r_i}(M), {\bm \eta}_{p,r_i}(\partial M)) $$
	is a smooth, almost properly embedded, minimal hypersurface, which is stable in ${\bm \eta}_{p,r_i}( M\cap T(p,\alpha r_i,r_i) ) $. 
	Moreover, since $T(p,r)\subset B_{2r}(p)$, we can use the monotonicity formula (\cite[Theorem 2.5]{li2021min}) and the five-times covering lemma as in (\ref{Eq:covering}) to get a uniform bound on the mass of $\Sigma_i$:
	\begin{eqnarray*}
		\mathcal{H}^n(\Sigma_i) &=& \frac{\mathcal{H}^n(\Sigma_i^*\cap T(p,\alpha r_i,r_i))}{r_i^n} \leq  \frac{\|V_i^*\|(B_{2r_i}(p))}{r_i^n} 
		\\
		&\leq &\frac{\|V\|(B_{10r_i}(p))}{r_i^n} \leq C_{mono} \frac{\|V\|(B_{10r_1}(p))}{r_1^n}
	\end{eqnarray*}
	After applying the compactness theorem \cite[Theorem 2.13]{li2021min}, we have $(\Sigma_i,\partial \Sigma_i )$ converges (up to a subsequence) to some smooth almost properly embedded minimal hypersurface $(\Sigma_\infty,\partial \Sigma_\infty )\subset (T_pM,T_p(\partial M))$, which is also stable in $B^{T_pG\cdot p}_1(0)\times A^{{\bf N}_pG\cdot p}_{\alpha,1}(0)$. 
	This gives the regularity of $D$ since
	\begin{equation}
		\spt(\|D\|)\cap \big(B^{T_pG\cdot p}_1(0)\times A^{{\bf N}_pG\cdot p}_{\alpha,1}(0)\big)=\Sigma_\infty. 
	\end{equation}

	By \cite[Lemma 2.2]{li2021min}, we can consider the doubled varifolds $\overline{C}, \overline{D} \in \V_n(T_p\tM)$ of $C,D \in \V_n(T_pM)$, and the doubled varifolds $\overline{W}, \overline{W}' \in \V_{n-\dim(G\cdot p)}({\bf N}_p^{\tM}G\cdot p)$ of $W,W' \in \V_{n-\dim(G\cdot p)}({\bf N}_pG\cdot p)$, where ${\bf N}_p^{\tM}G\cdot p$ is the normal vector space of $G\cdot p$ at $p$ in $\tM$. 
	By \cite[Lemma 2.2]{li2021min}, $\overline{C}$ and $\overline{D}$ are stationary in $T_p\tM$. 
	Since $T_p\tM=T_pG\cdot p\times {\bf N}_p^{\tM}G\cdot p$, $\overline{C}=T_pG\cdot p\times\overline{W}$ and $\overline{D}=T_pG\cdot p\times \overline{W}' $, we also have $\overline{W}, \overline{W'}$ are stationary in ${\bf N}_p^{\tM}G\cdot p$. 
	Hence, by Lemma \ref{Lem: tangent cone}(iii), 
	$$\|\overline{W}\|(B_r(0)) \equiv cr^{n-\dim(G\cdot p)}, \quad \forall r>0,$$ for some constant $c>0$.
	Using Proposition \ref{Prop:good-replacement-property}(i)(ii) and (\ref{Eq: part tube formula}), we have 
	$$\|\overline{W}'\|(B_r(0))=\|\overline{W}\|(B_r(0)),\quad \forall r\in(0,\alpha)\cup(1,\infty).$$ 
	The monotonicity formula now implies 
	$\|\overline{W}'\|(B_r(0)) \equiv cr^{n-\dim(G\cdot p)}, $
	and thus $\overline{W}' $ and $\overline{D}$ are cones. 
	Additionally, we have $\overline{C}=\overline{D} $, since $\overline{W}$ agree with $\overline{W}'$ outside $A^{{\bf N}_p^{\tM}G\cdot p}_{\alpha,1}(0)$.

	Denote by $\overline{\Sigma}_\infty\subset\spt(\|\overline{D}\|)$ the doubled hypersurface of $\Sigma_\infty$, which is a stable minimal hypersurface without boundary. 
	Hence, $\overline{\Sigma}_\infty = \spt(\overline{D}) \cap A^{T_p\tM}_{\alpha,1}(0)$ is a smooth stable hypersurface without boundary. 
	Since $\alpha\in(0,1/2)$ is arbitrary, ${\rm spt}(\|\overline{D}\|)$ is a stable minimal cone with vertex $0\in T_p\tM$. 
	Meanwhile, the dimension of $T_p\tM$ is bounded by $3\leq n+1\leq 7 $, which implies ${\rm spt}(\|\overline{D}\|)$ is a hyperplane in $T_p\tM$ by the non-existence of stable cones (\cite[Theorem 7.6]{pitts2014existence}). 
	Hence, $\overline{C}=\overline{D}$ is a hyperplane with multiplicity, which is also invariant under the reflection $\theta_{\nu_{\partial M}(p)}  (u):= u - 2(u\cdot \nu_{\partial M}(p))\nu_{\partial M}(p)$. 
	Thus, either (i) or (ii) is satisfied. 
	
	Finally, the arguments in \cite[Proposition 5.10, Claim 3, 4]{li2021min} could carry over to show that the tangent varifold of $V$ at $p$ is unique for $\|V\|$-a.e. $p \in \partial M$, and $V$ is rectifiable.
\end{proof}


\subsection{Regularity of $(G,\mZ_2)$-almost minimizing varifolds.}

\begin{proof}[Proof of Theorem \ref{Thm:main-regularity}]
	Using Proposition \ref{Prop:good-replacement-property}, \ref{Prop:good-replace-2}, \ref{Prop:classification tangent cone}, and Lemma \ref{Lem: tangent cone}, the proof of Li-Zhou \cite[Page 37-45]{li2021min} would carry over under $G$-invariant restrictions. 
	Here, we only point out a few modifications. 
	First, for any $G\cdot p\subset \spt(\|V\|)\cap \partial M$, we can choose $r<\frac{1}{20}\min\{r_{am}(G\cdot p),r_{\textrm{Fermi}}^G(p),r_{mono}\}$ so that \cite[(5.7)]{li2021min} holds with $\tBcal^{G,+}_r(p), \tScal^{G,+}_r(p)$ in place of $\tBcal^{+}_r(p), \tScal^{+}_r(p)$. 
	Specifically, combining the maximum principle \cite[Theorem 2.5]{li2021min} with the mean-convexity of $\partial_{\rm rel}\tBcal^{G,+}_r(p) = \tScal^{G,+}_r(p)$ (Lemma \ref{Lem:Fermi convex}), we have
	\begin{equation}\label{Eq-convex-ball}
		\emptyset \neq {\rm spt}(W)\cap  \tScal^{G,+}_r(p)  = {\rm Clos}\big[{\rm spt}(W)\setminus {\rm Clos}(\tBcal^{G,+}_r(p)) \big] \cap  \tScal^{G,+}_r(p),
	\end{equation} 
	for all $W\in \mathcal{V}_n(M)$ that is stationary in $\tBcal^{G,+}_r(p)$ with free boundary and $W\llcorner \tBcal^{G,+}_r(p)\neq 0$. 
	
	{\bf Step 1.} To construct successive $G$-replacements as in \cite[Section 5.3, Step 1]{li2021min}, we just need to consider the $G$-annuli $\A^G_{s,t}(p) $ and use Proposition \ref{Prop:good-replacement-property}, \ref{Prop:good-replace-2} in place of \cite[Proposition 5.3]{li2021min}. 
	To be exact, we take $0<s<t<r $ and apply Proposition \ref{Prop:good-replacement-property}, \ref{Prop:good-replace-2} to find a $G$-replacement $V^*$ of $V$ in $\Clos(\A^G_{s,t}(p))$. 
	By Lemma \ref{Lem:regular replacement}, $\Sigma_1 := \spt(\|V^*\|)\cap \A^G_{s,t}(p)$ a smooth almost properly embedded free boundary minimal $G$-hypersurface, which is stable in any simply connected open subset of $\A^G_{s,t}(p)$. 
	Next, by Sard's theorem, we can take $s_2\in (s,t)$ so that $\tScal^{G,+}_{s_2}(p)$ intersects $\Sigma_1$ transversally (even at $\partial \Sigma_1$). 
	Then for any $s_1\in(0,s)$, Proposition \ref{Prop:good-replacement-property} and \ref{Prop:good-replace-2} give a $G$-replacement $V^{**}$ of $V^*$ in $\Clos(\A^G_{s_1,s_2}(p))$. 
	It also follows from Lemma \ref{Lem:regular replacement} that $\Sigma_2 := \spt(\|V^{**}\|)\cap \A^G_{s_1,s_2}(p)$ a smooth almost properly embedded free boundary minimal $G$-hypersurface, which is stable in any simply connected open subset of $\A^G_{s_1,s_2}(p)$.

	{\bf Step 2.} We claim that $\Sigma_1$ and $\Sigma_2$ can be glued smoothly across $\tScal^{G,+}_{s_2}(p)$. 
	Indeed, since the smoothly gluing result in \cite[Section 5.3, Step 2]{li2021min} is proved locally, the arguments of Li-Zhou would carry over with Proposition \ref{Prop:good-replacement-property}, \ref{Prop:classification tangent cone}, and Theorem \ref{Thm:interior regul-amv} in place of \cite[Proposition 5.3, 5.10, Theorem 5.1]{li2021min}. 
	
	{\bf Step 3.} Denote $\Sigma_{s_1}:= \Sigma_2$ to indicate the dependence of $s_1$. 
	Since $s_1\in (0, s)$ is arbitrary, we have a family of smooth, almost properly embedded, free boundary minimal $G$-hypersurfaces $\{\Sigma_{s_1} \subset \A^G_{s_1,s_2}(p)\}$, and each $\Sigma_{s_1}$ is stable in any simply connected open subset of $\A^G_{s_1,s_2}(p)$. 
	By {\bf Step 2} and the 	unique continuation of minimal hypersurfaces, we have $\Sigma_{s_1}=\Sigma_1$ in $\A^G_{s,s_2}(p)$ and $\Sigma_{s_1'} = \Sigma_{s_1}$ in $\A^G_{s_1, s_2}(p)$ for all $0<s_1'<s_1<s$. 
	Therefore, 
	$$ \Sigma := \bigcup_{0<s_1<s}\Sigma_{s_1} $$ 
	is a smooth, almost properly embedded, free boundary minimal $G$-hypersurface in $\tBcal^{G,+}_{s_2}(p)\setminus G\cdot p$, which is also stable in any simply connected open subset of $\tBcal^{G,+}_{s_2}(p)\setminus G\cdot p$. 
	Moreover, after replacing $\tBcal^+_s(p)$ and $\tScal^+_s(p)$ by $\tBcal^{G,+}_s(p)$ and $\tScal^{G,+}_s(p)$, the arguments in \cite[Page 42, Claim 3]{li2021min} would carry over, and thus $\spt(\|V\|) = \Sigma$ in the punctured half-tube $\tBcal^+_s(p)\setminus G\cdot p$. 
	
    {\bf Step 4.} In the final step, we remove the singularity at the center $G\cdot p$. 
    Firstly, by the unique continuation result in {\bf Step 3} and the arguments in \cite[Page 42, Claim 4]{li2021min}, we have for any false boundary point $q\in\Sigma$, $V$ is stationary in a neighborhood of $q$ as a varifold in $\tM$. 
    It then follows from the $G$-invariance of $V$ and the constancy theorem (outside the true boundary points of $\Sigma$) that 
    \begin{equation}\label{Eq: hypersurface with multiplicity}
    	V\llcorner ( \tBcal^{G,+}_{t}(p) \setminus G\cdot p ) = \sum_{i=1}^N n_i |\Sigma_i |,
    \end{equation}
    where each $\Sigma_i$ is a smooth, $G$-connected, almost properly embedded, free boundary minimal $G$-hypersurface that is stable in any simply connected open subset of $\tBcal^{G,+}_{t}(p) \setminus G\cdot p$. 
	In particular, $\Sigma$ and its $G$-connected components $\{\Sigma_i\}$ are stable in $(T(p, \frac{s_2}{2} )\setminus G\cdot p) \cap M$ by Lemma \ref{Lem: part of tube}. 
    Then, the arguments in \cite[Page 43, Claim 5]{li2021min} would carry over with $T(p, r, 2r)$ and Proposition \ref{Prop:classification tangent cone} in place of $A_{r,2,r}(p)$ and \cite[Proposition 5.10]{li2021min}, where $T(p, r, 2r)$ is a part of tube defined as in Lemma \ref{Lem: part of tube}. 
    Consequently, we have the following strengthened version of Proposition \ref{Prop:classification tangent cone}: 
    \begin{claim}\label{Claim: unique tangent cone at center}
    	The tangent varifolds space $\VarTan(V, p)$ of $V$ at $p$ is 
    	$${\rm either~} \left\{\Theta^{n}(\|V\|, p)\left|T_{p}(\partial M)\right|\right\}~{\rm or ~} \left\{2 \Theta^{n}(\|V\|, p)\left| S \cap T_{p} M\right| \left| \begin{array}{c} S \subset T_{p} \widetilde{M} \mbox{ is an $n$-plane},\\ S \perp T_{p}(\partial M)\end{array}\right. \right\}.$$
    \end{claim}
    As in \cite[Page 44]{li2021min}, if the first case of Claim \ref{Claim: unique tangent cone at center} occurs, then $\partial\Sigma = \emptyset$ in a punctured neighborhood of $G\cdot p$, and $V$ is stationary in that punctured neighborhood as a varifold in $\tM$. 
    Since $\dim(G\cdot p)\leq n+1 - {\rm Cohom}(G)\leq n-2$, the stationarity can be extended across $G\cdot p$ by a standard extension argument (c.f. the proof in \cite[Theorem 4.1]{harvey1975extending}), and thus $V$ is stationary in a neighborhood of $G\cdot p$ in $\tM$. 
    Moreover, since $\dim(G\cdot p)\leq n-2$, the singularity at $G\cdot p $ can be removed by \cite[Corollary 2]{wickramasekera2014general}. 
    
	Now we consider the second case of Claim 5. 
	In this case, we have $p\in\Clos(\partial\Sigma)$ and $2 \Theta^{n}(\|V\|, p) = m \in \N$ (Proposition \ref{Prop:classification tangent cone}). 
	Since $\Sigma$ is stable in $(T(p, \frac{s_2}{2} )\setminus G\cdot p) \cap M$, it follows from the Compactness Theorem (\cite[Theorem 2.13]{li2021min}) that for any $r_i\to 0$, 
	\begin{equation}\label{Eq: blow up convergence}
		(\bleta_{p,r_i} )_\# V \to m |S\cap T_pM|
	\end{equation}
	locally smoothly in $\R^L \setminus \R^{\dim(G\cdot p)} \cong T_p\R^L\setminus T_pG\cdot p$ for some n-plane $S\subset T_p\tM$, $S\perp T_p(\partial M)$. 
	Thus, $\Sigma$ has no false boundary point in a neighborhood of $G\cdot p$. 
	Moreover, by (\ref{Eq: hypersurface with multiplicity}) and the locally smoothly convergence (\ref{Eq: blow up convergence}), we can take any $\sigma >0$ sufficiently small so that 
	\begin{equation}\label{Eq: graph decomposition}
		V\llcorner A^G_{\sigma/2, \sigma}(p) = \sum_{i=1}^{N(\sigma)} n_i(\sigma)|\Sigma_i(\sigma)|, \qquad V\llcorner T(p, \sigma/2, \sigma) = \sum_{i=1}^{N(\sigma)}\sum_{j=1}^{k_i(\sigma)} n_i(\sigma)|\Sigma_{i,j}(\sigma)|,
	\end{equation}
	where $\{\Sigma_i(\sigma)\}_{i=1}^{N(\sigma)}$ are the $G$-connected components of $\Sigma\cap A^G_{\sigma/2, \sigma}(p)$; $\{\Sigma_{i,j}(\sigma)\}_{j=1}^{k_i(\sigma)}$ are the graphical decomposition of $\Sigma_i(\sigma)$ in $T(p, \sigma/2, \sigma)$, i.e. $\Sigma_{i,j}(\sigma)$ is a graph over $T(p, \sigma/2, \sigma)\cap S$ for some $n$-plane $S\subset T_p\tM$ with $S\perp T_p(\partial M )$; $n_i(\sigma)$, $N(\sigma)$, and $\{k_i(\sigma)\}_{i=1}^{N(\sigma)}$ are positive integers so that 
	\begin{equation}\label{Eq: multiplicity}
		\sum_{i=1}^{N(\sigma)} n_i(\sigma)\cdot k_i(\sigma) = m. 
	\end{equation}
	It then follows from the continuity of $\Sigma$ that for $\sigma>0$ sufficiently small, $n_i(\sigma)$, $N(\sigma)$, and $\{k_i(\sigma)\}_{i=1}^{N(\sigma)}$ are constants independent on $\sigma$. 
	Hence, we can continue each $\Sigma_i(\sigma)$ and $\Sigma_{i,j}(\sigma)$ to $(B^G_\sigma(p)\setminus G\cdot p )\cap M$ and $(T(p,\sigma)\setminus G\cdot p)\cap M$ respectively, and get $\Sigma_i $, $\Sigma_{i,j}$. 
	Similar to the first case, by a standard extension argument (c.f. the proof in \cite[Theorem 4.1]{harvey1975extending}), each $\Sigma_i$ and $\Sigma_{i,j}$ can be extended as a stationary varifold in $B^G_\sigma(p)\cap M $ and $T(p,\sigma)\cap M$ with free boundary respectively. 
	Now, for any $C_{i,j}\in \VarTan(\Sigma_{i,j}, p)$, we have $C_{i,j} $ is supported on a half $n$-plane in $T_p \tM$ perpendicular to $T_p(\partial M )$ since we are in the second case of Claim \ref{Claim: unique tangent cone at center}. 
	Additionally, because every $\Sigma_{i,j}$ is stable, the Compactness Theorem \cite[Theorem 2.13]{li2021min} indicates that the limit $C_{i,j}$ of the rescaled convergence is a smooth stable hypersurface possibly with multiplicity, and thus $2\Theta^{n}(\|C_{i,j}\|, 0) \geq 1 $. 
	By (\ref{Eq: blow up convergence})(\ref{Eq: graph decomposition}) and (\ref{Eq: multiplicity}), we must have $2\Theta^{n}(\|C_{i,j}\|, 0) = 1 $. 
	Hence, every $\Sigma_{i,j}$ is stationary in $T(p,\sigma)\cap M$ with free boundary and $ \Theta^{n}(\|\Sigma_{i,j}\|, p) = \frac{1}{2}$. 
	We can now apply the Allard type regularity theorem for stationary rectifiable varifolds with free boundary \cite[Theorem 4.13]{gruter1986allard} to see the singularity of $\Sigma_{i,j}$ at $p$ can be removed. 
	By the $G$-invariance of $\Sigma_i$ and $\Sigma_i\cap T(p,\sigma) = \cup_{j=1}^{k_i}\Sigma_{i,j}$, we see $\Sigma_i$ extends as a smooth hypersurface across $G\cdot p$. 
	Thus, every $\Sigma_{i,j}$ is a smooth, almost properly embedded, free boundary minimal hypersurface in $(T(p,\sigma)\cap M, T(p,\sigma)\cap \partial M)$. 
	Finally, if $N>1 $ or $k_i>1$ for some $1\leq i \leq N$, then we can find $\Sigma_{i_1,j_1}$ and $\Sigma_{i_2, j_2}$ touching at $G\cdot p \cap T(p,\sigma)$ since $\dim(G\cdot p) \leq n-2$, which contradicts the classical maximum principle. 
	Therefore, $N=1$, $k_i=1$, and this finished the proof. 
\end{proof}


\appendix
	\addcontentsline{toc}{section}{Appendices}
	\renewcommand{\thesection}{\Alph{section}}

\section{Equivariant Riemannian extension}\label{Sec: extension}

In this appendix, we will briefly introduce the equivariant extending of a Riemannian $G$-manifold with boundary to a closed Riemannian $G$-manifold. 
Let $(M,g_{_M})$ be a compact connected Riemannian $(n+1)$-manifold with $\partial M\neq\emptyset$. 
Suppose $G$ is a compact Lie group acting as isometries on $M$. 
Denote by $\mu$ a bi-invariant Haar measure on $G$, which has been normalized to $\mu(G)=1$. 

\begin{lemma}
	There exists a closed Riemannian $(n+1)$-manifold $(\tM,g_{_{\tM}})$ so that 
	\begin{itemize}
		\item[(1)] $G$ acts on $\tM$ as isometries with the same cohomogeneity;
		\item[(2)] $M\subset \tM$, and the inclusion map is a $G$-equivariant isometric embedding.
	\end{itemize}
\end{lemma}
\begin{proof}
	Let $M_1$ be a copy of $M$ and $id:\partial M\to \partial M_1$ be a $G$-equivariant diffeomorphism. 
	By \cite{kankaanrinta2007equivariant}, we can define a closed smooth  $G$-manifold $\tM = M\cup_{id} M_1$ as the smooth equivariant gluing, whose differentiable structure is obtained from $G$-equivariant collars of $\partial M$. 
	Hence, $G$ acts on $\tM$ as diffeomorphisms, and $\tM$ induces the original differentiable structure of $M$ as well as $M_1$. 
	By \cite{pigola2016smooth}, there exists a Riemannian metric $\widetilde{g}_{_{\tM}}$ on $\tM$ so that $\widetilde{g}_{_{\tM}}\llcorner M=g_{_M}$. 
	Define a new Riemannian metric $g_{_{\tM}}$ as 
	$$ g_{_{\tM}}(u,v) :=\int_G \widetilde{g}_{_{\tM}} (g_*u,g_*v)~d\mu(g),\quad \forall u,v\in T_p\tM, p\in\tM. $$
	Clearly $g_{_{\tM}}$ is bilinear and symmetric. 
	Since $G$ acts on $\tM$ as diffeomorphisms, the tangent map $g_*$ is non-degenerate, which implies $g_{_{\tM}}(\cdot,\cdot)$ is positive definite. 
	By \cite[Theorem 3.3]{bredon1972introduction}, $g_{_{\tM}}(\cdot,\cdot)$ is smooth and thus a Riemannian metric. 
	This metric is also $G$-invariant since
	\begin{eqnarray*}
		g_{_{\tM}}(h_*u,h_*v) = \int_G \widetilde{g}_{_{\tM}} (g_*h_*u,g_*h_*v)~d\mu(g)= \int_G \widetilde{g}_{_{\tM}} ((g\circ h)_*u,(g\circ h)_*v)~d\mu(g) = g_{_{\tM}}(u,v),
	\end{eqnarray*}
	for all $h\in G$. 
	Finally, noting $\widetilde{g}_{_{\tM}}\llcorner M=g_{_M}$ and $g_{_M}$ is $G$-invariant, we have $g_{_{\tM}}\llcorner M=g_{_M} $. 
\end{proof}

\section{Fermi tube}\label{Sec:Fermi half-tubes}

In this appendix, we generalize the Fermi half-balls in \cite[Appendix A]{li2021min} to Fermi half-tubes.

Let $p\in\partial M$ and $(x_1,\dots,x_{\dim(G\cdot p)})$ be the geodesic normal coordinates of $G\cdot p$ centered at $p$. 
Denote $\{E_i\}_{i=\dim(G\cdot p)+1}^n$ to be the orthonormal sections of the normal bundle ${\bf N}^{\partial M}(G\cdot p)$ of $G\cdot p$ in $\partial M$ around $p$. 
Let $t=\dist_M(\cdot,\partial M)$ which is well-defined and smooth in a small relatively open neighborhood of $\partial M$ in $M$. 
As in \cite[Page 17]{gray2003tubes}, we can define the Fermi coordinates $(x_1,\dots,x_n)$ of $G\cdot p$ in $\partial M$ centered at $p$ relative to $(x_1,\dots,x_{\dim(G\cdot p)})$ and $\{E_{\dim(G\cdot p)+1},\dots,E_n\}$.
By \cite[Lemma 2.4]{gray2003tubes}, we have $\frac{\partial}{\partial x_i}\big\vert_{G\cdot p} = E_i$ around $p$ for $i\in\{\dim(G\cdot p)+1,\dots,n\}$, and $\{\frac{\partial}{\partial x_i}\vert_{G\cdot p}\}_{i=\dim(G\cdot p)+1}^n$ are orthonormal coordinate vector fields.

\begin{definition}
	The {\em Fermi coordinates of $G\cdot p\subset M$ centered at $p$} are defined as $(x_1,\dots,x_n,t)$, where $x_i$ and $t$ are defined as above. 
	By compositing the Fermi coordinate map at $p$ with $g^{-1}\in G$, $(x_1,\dots,x_n,t)$ are also Fermi coordinates of $G\cdot p\subset M$ centered at $g\cdot p$. 
\end{definition}

\begin{definition}[Fermi distance function from $G\cdot p$ around $p$]
	Let $(x_1,\dots,x_n,t)$ be a Fermi coordinate system of $G\cdot p\subset M$ centered at $p$. 
	The {\em Fermi distance function from $G\cdot p$ around $p$} is defined on a relative open neighborhood of $p$ in $M$ by 
	\begin{equation}\label{Eq:Fermi distance}
		\tilde{r}^G_p(q):=|(x_{\dim(G\cdot p)+1},\dots,x_n,t)|=\sqrt{x_{\dim(G\cdot p)+1}^2+\dots+x_n^2+t^2}. 
	\end{equation}
\end{definition}

By \cite[Lemma 2.6]{gray2003tubes}, $\tilde{r}^G_p$ does not depend on the choice of Fermi coordinates at $p$. 
Moreover, by compositing the Fermi coordinate map with $g^{-1}\in G$, we can see that $\tilde{r}^G_p$ is well-defined in a $G$-neighborhood $B^G_\rho(p)\cap M$ of $G\cdot p$. 
Indeed, we have the following result:

\begin{lemma}
	For any $q\in M$, denote $n_{\partial M}(q)$ to be the unique nearest (geodesic) project point of $q$ to $\partial M$, 
	which is well-defined in a neighborhood of $\partial M$. 
	Then {\em the Fermi distance function around $G\cdot p$} defined by
	\begin{equation}\label{Eq:Fermi distance 3}
		\tilde{r}_{G\cdot p}(q) := \sqrt{(\dist_{\partial M}(n_{\partial M}(q),G\cdot p))^2+(\dist_M(q, \partial M))^2}
	\end{equation}
	is a non-negative $G$-invariant function, which is well defined in a (relative) $G$-neighborhood of $G\cdot p$ in $M$, and coincides with $\tilde{r}^G_p$ around $p$. 
\end{lemma}
\begin{proof}
	Since $G$ acts by isometries, both $\dist_{\partial M}(\cdot,G\cdot p)$ and $\dist_M(\cdot,\partial M)=t$ are $G$-invariant functions. 
	Hence, $\tilde{r}_{G\cdot p}$ is a $G$-invariant function, which is well defined in a $G$-neighborhood of $G\cdot p$ in $M$. 
	If $q=(x_1,\dots,x_n,t)$ in the Fermi coordinates of $G\cdot p$ around $p$. 
	Then we have $n_{\partial M}(q)=(x_1,\dots,x_n,0) $. 
	By \cite[Lemma 2.7]{gray2003tubes}, $\sum_{i=\dim(G\cdot p)+1}^n x_i^2=(\dist_{\partial M}(n_{\partial M}(q),G\cdot p))^2$. 
	Thus, the function $\tilde{r}_{G\cdot p}$ coincides with $\tilde{r}^G_p$ around $p$. 
\end{proof}

Since $\tilde{r}_{G\cdot p}$ is a $G$-invariant function, we can define the following $G$-subsets of $M$, which generalize the definitions of Fermi half-balls and Fermi half-spheres in \cite[Definition A.4]{li2021min}.

\begin{definition}\label{Def:Fermi tube}
	For any $p\in\partial M$, we define the {\em Fermi half-tube} and {\em Fermi half-cylinder} of radius $r$ centered at $G\cdot p$ respectively by 
	$$\tBcal^{G,+}_r(p) := \{q\in M~\vert~\tilde{r}_{G\cdot p}(q)<r \},\quad \tScal^{G,+}_r(p):=\{q\in M~\vert~\tilde{r}_{G\cdot p}(q)=r  \}. $$
\end{definition}

\begin{definition}\label{Def:Fermi exponential map}
	The {\em Fermi exponential map of $G\cdot p$} is a $G$-equivariant diffeomorphism on a $G$-invariant half-tube of ${\bf N}^{\partial M}(G\cdot p) \times \R^+\cong {\bf N}(G\cdot p)$ centered at the zero section $({\bf 0}, 0)$ given by: 
	\begin{eqnarray*}
		&& \widetilde{\exp}_{G\cdot p}^\perp(v,s) := \exp_{q'}(s\frac{\partial}{\partial t}),
	\end{eqnarray*}
	where $v\in {\bf N}^{\partial M}(G\cdot p)$, $s\in \R^+$, $q'=\exp_{G\cdot p}^{\partial M,\perp}(v)$.
\end{definition}

Using $\{x_i\}_{i=1}^{\dim(G\cdot p)}$ and $\{E_i\}_{i=\dim(G\cdot p)+1}^n$, the normal bundle ${\bf N}(G\cdot p)\cong {\bf N}^{\partial M}(G\cdot p) \times \R^+$ can be locally trivialized around $p$ with Fermi coordinates $(x_1,\dots,x_n,t)$. 
Therefore, in a neighborhood of $p$, the Fermi exponential map of $G\cdot p$ and the local trivialization of ${\bf N}(G\cdot p)$ gives the Fermi coordinate system of $G\cdot p\subset M$ centered at $p$.

Moreover, let 
\begin{equation}\label{Eq:Fermi slice}
	S_p:= \{q\in M : \mbox{$q$ is close to $p$ and } (\widetilde{\exp}_{G\cdot p}^\perp)^{-1}(q) \in 	{\bf N}^{\partial M}_p(G\cdot p) \times \R^+ \}
\end{equation}
be a (properly) embedded submanifold of $M$ around $p$, which is also a {\em slice} of $G\cdot p$ at $p$ by the arguments similar to \cite[Section 2.2]{berndt2016submanifolds}. 
By definitions, in a neighborhood of $(0,0)\in {\bf N}^{\partial M}_p(G\cdot p)\times \R^+ \cong T_p S_p$, we have 
\begin{equation}\label{Eq:Fermi exponential map}
	\widetilde{\exp}_{G\cdot p}^\perp = \widetilde{\exp}_p ,
\end{equation}
 where $\widetilde{\exp}_p$ is the Fermi exponential map at $p$ defined in \cite[Definition A.1]{li2021min}. 
 As a result, we can apply \cite[Lemma A.2, A.3]{li2021min} to $\widetilde{\exp}_{G\cdot p}^\perp $ on the slice $S_p$, and obtain the following Lemma, which is parallel to \cite[Lemma A.5]{li2021min}.

\begin{lemma}\label{Lem:Fermi convex}
	For any $p\in \partial M$, there exists a small constant $r_{\textrm{Fermi}}^G(p)>0$, such that for all $0<r<r_{\textrm{Fermi}}^G(p)$
\begin{itemize}
	\item $\tScal^{G,+}_r(p)$ is a smooth hypersurface meeting $\partial M$ orthogonally,
	\item $\tBcal^{G,+}_r(p)$ is relatively mean convex in the sense of \cite[Definition 2.4]{li2021min}, 
	\item $B_{r/2}^G(p) \cap M \subset \tBcal^{G,+}_r(p) \subset B_{2r}^G(p) \cap M$.
\end{itemize}
\end{lemma}

\begin{proof}
	Firstly, take any $q\in \tScal^{G,+}_r(p)$, and assume $(\widetilde{\exp}_{G\cdot p}^\perp)^{-1}(q) \in {\bf N}_p(G\cdot p)$. 
	Let $S$ be the slice of $G\cdot p$ at $p$ as in (\ref{Eq:Fermi slice}). 
	Under the Fermi coordinates centered at $p$, we have $q=(0,\dots, 0, x_{\dim(G\cdot p) +1},\dots, x_n, t)\in S$. 
	Let $i,j\in\{1,\dots,\dim(G\cdot p)\}$, $a,b\in\{\dim(G\cdot p) +1, \dots, n\}$. 
	By (\ref{Eq:Fermi exponential map}), the estimates on $g_{ab}$ and Christoffel symbols containing $a,b$ are the same as in \cite[Lemma A.2]{li2021min}. 
	Additionally, one can can derive as in \cite[Lemma A.2]{li2021min} the estimates $g_{ij}(q) = \delta_{ij} + O(\tilde{r}_{G\cdot p}(q))$, $g_{ia}(q) = O(\tilde{r}_{G\cdot p}(q))$, $g_{it}(q) = 0$, $g_{at}(q)= 0$, $g_{tt}(q)= 1$, $\Gamma_{it}^t(q)=\Gamma_{tt}^i(q)=\Gamma_{tt}^t(q)=0$, and the rest of the Christoffel symbols containing $i$ are $O(1)$. 

	If $q\in \tScal^{G,+}_r(p) \cap \partial M$, then $t=0$ and $\nabla \tilde{r}_{G\cdot p}(q)\in T_p\partial M$. 
	Since $\nabla \tilde{r}_{G\cdot p}$ is the normal vector filed on $\tScal^{G,+}_r(p) $, the $G$-hypersurface $\tScal^{G,+}_r(p)$ meets $\partial M$ orthogonally. 
	
	Combining the above estimates with the arguments in \cite[Lemma A3]{li2021min}, 
	we also have $\|\nabla \tilde{r}_{G\cdot p}(q) - \frac{ \partial }{\partial \tilde{r}_{G\cdot p}}(q)\| \leq C \tilde{r}_{G\cdot p}(q)$ and $\| {\rm Hess}(\tilde{r}_{G\cdot p})(q) - 0\otimes \frac{ 1}{r} \cdot g_{r}\|\leq C $, where $g_{r}$ is the metric on the Euclidean sphere $\partial \mathbb{B}_{r}^{n+1-\dim(G\cdot p)}(0)$ of radius $r=\tilde{r}_{G\cdot p}(q)$. 
	Then the second bullet follows from the fact that the second fundamental form of the level set $\tScal^{G,+}_r(p) = (\tilde{r}_{G\cdot p})^{-1}(r)$ at $q$ is ${\rm Hess}(\tilde{r}_{G\cdot p})(q)$.

	By (\ref{Eq:Fermi exponential map}), we have $\tBcal^+_r(p)\cap S=  \tBcal^{G,+}_r(p)\cap S$, where $\tBcal^+_r(p) $ is defined in \cite[Definition A.4]{li2021min}. 
	Hence, the estimates in \cite[Lemma A.2, A3]{li2021min} can be applied to $\tBcal^{G,+}_r(p)\cap S$ with coordinates $(x_{\dim(G\cdot p) +1},\dots, x_n, t)$. 
	Additionally, \cite[Lemma A.5]{li2021min} suggests that for $r>0$ sufficiently small, $\tBcal^{G,+}_r(p)\cap S$ is a relatively convex domain in $S$ satisfying 
	$$B_{r/2}^G(p) \cap S ~\subset~ \tBcal^{G,+}_r(p)\cap S ~\subset~ B_{2r}^G(p) \cap S, $$
	which gives the last bullet since $\tBcal^{G,+}_r(p) = \bigcup_{g\in G}(\tBcal^+_r(p)\cap g\cdot S)$. 
\end{proof}

\begin{remark}\label{Rem:lower bound of Fermi radius}
	We mention that the Fermi radius $r_{\textrm{Fermi}}^G(p)>0$ in the above lemma can be made uniformly in any compact subset of an orbit type stratum. 
	Indeed, suppose $U\subset \partial M$ is a $G$-subset so that every orbit in $\Clos(U)$ has the same orbit type, i.e. $\Clos(U)\subset N_i $ for some $i=1,\dots, m'$ (see the notations in (\ref{Eq: stratification of boundary and dimension})). 
	Note $\pi : N_i \to \pi(N_i)\subset M/G$ is a Riemannian submersion. 
	There exists a vector bundle ${\bf N}$ over $\Clos(U)$ with fiber ${\bf N}^{\partial M}_p(G\cdot p) \times \R^+  \cong {\bf N}_p(G\cdot p)$ at $p\in \Clos(U)$, and a smooth map $\widetilde{\exp}^\perp: {\bf N} \to \pi(\Clos(U))\times M$ given by $\widetilde{\exp}^\perp(p,v) := ([p], \widetilde{\exp}^\perp_{G\cdot p}(v))$. 
	One can verify that $d\widetilde{\exp}^\perp(p,0): T_{(p,0)}{\bf N} \to T_{([p],p)}(\pi(\Clos(U))\times M) \cong T_{[p]}\pi(\Clos(U)) \oplus T_p M$ is an isomorphism, and thus $\widetilde{\exp}^\perp$ is a diffeomorphism in a neighborhood $\mathcal{W}$ of $(p,0)$ in ${\bf N}$. 
	Consider the function $\tilde{r} : {\bf N} \to \mathbb{R}$ given by $\tilde{r}(q,w) = |w|$, which is smooth whenever $\tilde{r} \neq 0$. 
	Then $\tilde{r}_{G\cdot p}(q) = \tilde{r} \circ (\widetilde{\exp}^\perp)^{-1}([p], q)$ is smooth in $([p], q)\in \widetilde{\exp}^\perp( \mathcal{W} )$ whenever $\tilde{r}_{G\cdot p}(q) \neq 0 $. 
	Therefore, by the smoothness, the estimates in the proof of Lemma \ref{Lem:Fermi convex}, as well as the Fermi radius $r_{\textrm{Fermi}}^G$, can be made uniformly on any such compact set $\Clos(U)$. 
\end{remark}

\section{Proof of Theorem \ref{Thm: equivalence-a.m.v}}\label{Sec:proof-of-equivalence-a.m.v}

The proof of Theorem \ref{Thm: equivalence-a.m.v} is a combination of the work in \cite[Appendix B]{wang2022min} and \cite[Appendix B]{li2021min}. 
By definitions and $\F(\tau-\sigma)\leq\mF(\tau,\sigma)\leq 2\M(\tau-\sigma)$, it is obvious that (a) $\Rightarrow$ (b) $\Rightarrow$ (c) in Theorem \ref{Thm: equivalence-a.m.v}. 
As for (c) $\Rightarrow$ (d), we use the stratification of orbit types to generalize the arguments in \cite[Appendix B]{li2021min}.

First, we will need the following tube formula to transform the estimate of mass into the normal direction. 
Given any $p\in M$ and $r>0$, denote $k := \dim(G\cdot p)$, 
\begin{equation}\label{Eq: notations}
	{\rm Exp}_{G\cdot p}^\perp := \left\{ \begin{array}{cl}
	\exp_{G\cdot p}^\perp  & \text{ if } p \notin  \partial M, \\
	\widetilde{\exp}_{G\cdot p}^\perp  & \text{ if } p \in \partial M; 
	\end{array} \right. 
	~
	r_{G\cdot p}:= \left\{ \begin{array}{cl}
	\dist_M(G\cdot p, \cdot)  & \text{ if } p \notin  \partial M, \\
	\tilde{r}_{G\cdot p}  & \text{ if } p \in \partial M;
	\end{array} \right. 
\end{equation}
and $\B_r^G(p) := \{ q\in M: r_{G\cdot p}(q) < r\}$. 
By choosing $r>0$ sufficiently small, we may assume ${\rm Exp}_{G\cdot p}^\perp$ is a $G$-equivariant diffeomorphism on $\B_r^G(p)$. 
Then for any $q\in \B_r^G(p)$, there exists a unique $p_q\in G\cdot p$ so that $({\rm Exp}_{G\cdot p}^\perp)^{-1}(q)\in {\bf N}_{p_q}(G\cdot p)$. 
Let $B(p, \xi)\subset G\cdot p$ be a geodesic open neighborhood of $p$ in $G\cdot p$ with radius $\xi >0$. 
Define then 
$$T\big(B(p,\xi), r\big) := \{q\in \B_r^G(p): p_q\in B(p,\xi)\}, $$
which is a part of the tube $\B_r^G(p)$. 
Suppose $\widetilde{\Sigma} \subset \B_r^G(p)$ is a $G$-invariant hypersurface (or $n$-rectifiable set). 
Since $G$ acts by isometries, we have $\mathcal{H}^{k}(B(p,\xi)) = \mathcal{H}^{k}(B(p',\xi)) $ and $\mathcal{H}^n( \widetilde{\Sigma} \cap T\big(B(p,\xi), r\big) ) = \mathcal{H}^n( \widetilde{\Sigma} \cap T\big(B(p',\xi), r\big) ) $ for any $p'\in G\cdot p$. 
Therefore, 
\begin{eqnarray}\label{Eq: part tube formula}
	\mathcal{H}^n \left( \widetilde{\Sigma} \cap T\big(B(p,\xi), r\big) \right) &= & \frac{1}{\mathcal{H}^{k}(G\cdot p)}   \int_{G\cdot p} \int_{\widetilde{\Sigma}} 1_{T (B(p',\xi), r )}(q) ~d\mathcal{H}^n(q) d \mathcal{H}^k(p') \nonumber
	\\
	&=& \frac{1}{\mathcal{H}^{k}(G\cdot p)}   \int_{\widetilde{\Sigma}}\int_{G\cdot p}  1_{B(p_q,\xi)}(p')  ~d \mathcal{H}^k(p')d\mathcal{H}^n(q)  \nonumber
	\\
	&=& \frac{1}{\mathcal{H}^{k}(G\cdot p)}  \int_{\widetilde{\Sigma}} \mathcal{H}^k(B(p_q, \xi)) ~ d\mathcal{H}^n(q)  
	\\
	&=& \frac{\mathcal{H}^{k}(B(p,\xi))}{\mathcal{H}^{k}(G\cdot p)} \mathcal{H}^n( \widetilde{\Sigma} ).  \nonumber
\end{eqnarray}
Here, we used the fact that $1_{T (B(p',\xi), r )}(q) = 1_{B(p',\xi)}(p_q) = 1_{B(p_q, \xi)}(p') $. 
Now, we write $\Sigma := ({\rm Exp}_{G\cdot p}^\perp)^{-1}(\widetilde{\Sigma} )$, $U_\xi := ({\rm Exp}_{G\cdot p}^\perp)^{-1}(T(B(p,\xi), r))$, $\Sigma_p := \Sigma\cap {\bf N}_p(G\cdot p)$, $\widetilde{\Sigma}_p := {\rm Exp}_{G\cdot p}^\perp(\Sigma_p)$. 
Under the action of $G$, we can regard ${\bf N}(G\cdot p)$ as a fiber bundle over $G\cdot p\cong G/G_p$ (see \cite[Section 2.2]{berndt2016submanifolds}). 
By the $G$-invariance of $\widetilde{\Sigma}$, we have $\Sigma$ is a sub-bundle of ${\bf N}(G\cdot p)$ with fiber $\Sigma_p$. 
Then for $\xi >0$ sufficiently small, we have $U_\xi \cong B(p,\xi) \times \mathbb{B}_r(0)$ and $\Sigma\cap U_\xi  \cong  B(p,\xi) \times \Sigma_p$, where $\mathbb{B}_r(0)\subset {\bf N}_pG\cdot p$ is an open ball or half-ball. 
Therefore, we can use the area formula and (\ref{Eq: part tube formula}) to see 
\begin{eqnarray}\label{Eq:weyl-tube}
	\mathcal{H}^n( \widetilde{\Sigma} ) &=&  \frac{\mathcal{H}^{k}(G\cdot p)}{\mathcal{H}^{k}(B(p,\xi))} \int_{\Sigma_p}\int_{B(p,\xi)}   J({\rm Exp}_{G\cdot p}^\perp \llcorner \Sigma )   d\mathcal{H}^{k} d\mathcal{H}^{n-k} \nonumber
	\\
	&=& \mathcal{H}^{k}(G\cdot p) \int_{\widetilde{\Sigma}_p} \vartheta^{G\cdot p}({\widetilde{\Sigma}}, q) ~d\mathcal{H}^{n-k}(q),
\end{eqnarray}
where $\vartheta^{G\cdot p}({\widetilde{\Sigma}}, \cdot) = \big( \int_{B(p,\xi)} J({\rm Exp}_{G\cdot p}^\perp \llcorner \Sigma ) d\mathcal{H}^{k} \big)  /   \big(  \mathcal{H}^{k}(B(p,\xi)) \cdot J({\rm Exp}_{G\cdot p}^\perp \llcorner (\Sigma_p) ) \big)$.
Since $|d{\rm Exp}_{G\cdot p}^\perp|$ tends to $1$ as $r\to 0$, we have $\vartheta^{G\cdot p}({\widetilde{\Sigma}}, \cdot )= 1+o(1)  $ as $r\to 0$ (independent on $\widetilde{\Sigma}$). 
Hence, for any $\epsilon'\in (0, 1/3)$, there exists $r_{\epsilon' }(G\cdot p)>0$ so that 
\begin{equation}\label{Eq:bound-of-theta}
	\frac{1}{1+\epsilon'} \leq  \vartheta^{G\cdot p} \leq 1+\epsilon'
\end{equation}
holds in $\B_{r_{\epsilon' }(G\cdot p)}^G(p) $. 
By the arguments in Remark \ref{Rem:lower bound of Fermi radius}, such $r_{\epsilon'}$ can be taken uniformly in a compact $G$-subset of $M_{(G_p)}$ or $(\partial M)_{(G_p)}$. 
More generally, we have 
$$\int_{\widetilde{\Sigma}}\theta d\mH^n = \mathcal{H}^{k}(G\cdot p) \int_{\widetilde{\Sigma}_p} \vartheta^{G\cdot p}({\widetilde{\Sigma}}, q)\cdot \theta(q) ~d\mathcal{H}^{n-k}(q), $$
for any $G$-invariant $\mH^n$-measurable function $\theta\geq 0$ on $\widetilde{\Sigma}$. 

Next, let $U\subset M$ be a relative open $G$-set of $M$, and $W\subset\subset U$ be a relative open $G$-subset of $U$ containing no isolated orbit. 
Then for any $G\cdot q\subset W$, we define 
\begin{equation}\label{Eq: stratum}
	G\cdot q \subset 
		K := \left\{ \begin{array}{cl}
		M_i^j  & \text{ if } q \in M \setminus \partial M, \\
		N_i^j  & \text{ if } q \in \partial M, 
		\end{array} \right. 
\end{equation}
where $M_i^j, N_i^j$ are the $G$-connected components of some orbit type strata in $M$, $\partial M$ (see (\ref{Eq: stratification of M})(\ref{Eq: stratification of boundary})). 
Note $\dim(K)=n+1$ if $q\in M^{reg}$, $\dim(K)=n$ if $q\in (\partial M)^{reg}$, and $\dim(K) - k \geq 1$, where $k=\dim(G\cdot q)$. 
By Remark \ref{Rem:lower bound of Fermi radius} and a similar argument for $G\cdot q \subset M\setminus\partial M$, for any $\epsilon\in (0,1)$ and $\epsilon'\in (0,1/3)$, there exists $r_0 = r_0(\epsilon, \epsilon', q, M, G ) > 0$ sufficiently small such that for all $p\in K\cap \B_{r_0}^G(q) \subset \subset K$ the following properties hold: 
\begin{itemize}
	\item[(1)] $4 r_0 < \min \big\{ {\rm Inj}(G\cdot p), r_{\textrm{Fermi}}^G(p), r_{\epsilon'}(G\cdot p), {\rm Inj}_M, r_{\textrm{Fermi}} \big\}$, 
		where ${\rm Inj}(G\cdot p)$ is the injectivity radius of $G\cdot p$ in $\widetilde{M}$, ${\rm Inj}_M$ is the infimum of the injectivity radii for points in $\widetilde{M}$, 
		and $r_{\textrm{Fermi}}^G(p)$, $r_{\epsilon'}(G\cdot p)$, $r_{\textrm{Fermi}}$ are defined in Lemma \ref{Lem:Fermi convex}, (\ref{Eq:bound-of-theta}), \cite[Lemma A.5]{li2021min};
	\item[(2)] ${\rm Exp}_{G\cdot p}^\perp$ is an equivariant diffeomorphism on $\B_{4r_0}^G(p)\supset \B_{r_0}^G(q)$, thus we can define 
		$$E := {\rm Exp}_{G\cdot p}^\perp \big\vert_{({\rm Exp}_{G\cdot p}^\perp)^{-1}(\B_{4r_0}^G(p))}, \quad E_p := E\big\vert_{{\bf N}_p(G\cdot  p)\cap ({\rm Exp}_{G\cdot p}^\perp)^{-1}(\B_{4r_0}^G(p))},$$
		and define $S_p$ as the image of $E_p$, which is a slice of $G\cdot p$ at $p$ in $\B_{4r_0}^G(p)$;
	\item[(3)] $\max\{    ({\rm Lip}(E_p ))^{n-k}({\rm Lip}(E_p^{-1}))^{n-k} , ({\rm Lip}(E))^n ({\rm Lip}((E^{-1}))^n   \}   \leq \frac{9}{8} <2$;
	\item[(4)] ${\rm Lip}(r_{G\cdot p} ) \leq \frac{9}{8} < 2$;
	\item[(5)] if $x\in S_p\subset \B_{4r_0}^G(p)$, $\lambda\in [0,1]$, and $\omega\in \Lambda_{n-k} T_xS_p$, then $\lambda^{n-k}(1+\epsilon(1-\lambda)) \leq 1$, $E\circ {\bm \mu}_\lambda\circ E^{-1} (x) \in S_p\subset \B_{4r_0}^G(p)$, and 
		\begin{equation}\label{Eq: scale estimate}
			\| D(E_p\circ {\bm \mu}_\lambda\circ E^{-1}_p)_*\omega \| \leq \lambda^{n-k}(1+\epsilon(1-\lambda)) \|\omega \|. 
		\end{equation}
			(This follows from \cite[3.4(4)]{pitts2014existence}\cite[Lemma B.4]{li2021min} and the fact that $E_p = \exp_p$ or $\widetilde{\exp}_p$.)
\end{itemize} 
If $q\notin\partial M$, then we further assume $5 r_0 < \dist_M(q, \partial M)$.

\begin{lemma}\label{Lem:mass estimate under exponential map}
	Given any non-isolated orbit $G\cdot q\in M$, $\epsilon\in  (0, 1)$, and $\epsilon'\in  (0, 1/3)$, let $K$, $r_0$, $E$, $r_{G\cdot q}$, $\B_r^G$, ${\rm Exp}^{\perp}_{G\cdot q}$ be defined as above. 
	Let $Z:= \B_{r_0}^G(q)$, $k:=\dim(G\cdot q)$. 
	Then for any $p\in Z\cap K$, $\lambda\in [0,1]$, and $r>0$ with $\B^{G}_r(p)\subset\subset Z$, we have 
	\begin{equation}\label{Eq: cone estimates}
		{\bf M}((E\circ{\bm \mu}_\lambda\circ E^{-1})_\# T)\leq \lambda^{n-k}(1+\epsilon(1-\lambda))(1+\epsilon')^2{\bf M}(T) \leq (1+\epsilon')^2{\bf M}(T)
	\end{equation}
	holds for all $T\in Z^G_n \big(\Clos(\B^{G}_r(p)),\partial \Clos(\B^{G}_r(p));\mathbb{Z}_2 \big)$. 
	Moreover, denote 
	$$ 
	\partial_1T := \left\{ \begin{array}{cl}
		\partial T  & \text{ if } q \notin \partial M, \\
		\partial T\llcorner (\tScal^{G,+}_r(p))  & \text{ if } q \in \partial M, 
		\end{array} \right. 
	\mbox{ and }
	T_\lambda := E_\#\big(\delta_{\bf 0} \ttimes [E^{-1}_\#\partial_1 T-({\bm \mu}_\lambda\circ E^{-1})_\#\partial_1 T]\big),
	$$
	where ${\bf 0}$ is the zero-section in ${\bf N}(G\cdot p)$ and $\delta_{\bf 0} \ttimes S := h_\#([[0,1]]\times S)$, $h:\mathbb{R}\times {\bf N}(G\cdot p)\to {\bf N}(G\cdot p)$, $h(t,v)=tv$.
	Then 
	\begin{itemize}
		\item $ \partial T_\lambda-\big[\partial_1 T- (E\circ{\bm \mu}_\lambda\circ E^{-1})_\#\partial_1 T\big] \in Z_{n-1}^G(Z\cap\partial M,Z\cap\partial M;\mZ_2),$
		\item $\spt(T_\lambda)\subset \A_{\lambda r,r}^G(p)$, where $\A_{\lambda r,r}^G(p)$ is defined in Definition \ref{Def:annulus},
		\item $ {\bf M}(T_\lambda)\leq 2r(n-k)^{-1}(1-\lambda^{n-k}){\bf M}(\partial_1 T).$  
	\end{itemize}
\end{lemma}
\begin{proof}
	With the notations in \cite[27.1]{simon1983lectures}, we write $\underline{\tau}(\widetilde{\Sigma}, \theta, \xi )$, $\theta\in\{0,1\}$, as a representative modulo $2$ of $T$ (see \cite[4.2.26]{federer2014geometric}).  
	By (\ref{Eq:weyl-tube})(\ref{Eq:bound-of-theta})(\ref{Eq: scale estimate}), we have $\M(T)=\int_{\widetilde{\Sigma}}\theta d\mH^n$ and 
	\begin{eqnarray*}
		 {\bf M}((E\circ{\bm \mu}_\lambda\circ E^{-1})_\# T) 
		 &\leq & (1+\epsilon') \mathcal{H}^k(G\cdot p)    \int_{E\circ{\bm \mu}_\lambda\circ E^{-1}(\widetilde{\Sigma})\cap Z_p} \theta \circ (E\circ{\bm \mu}_\lambda\circ E^{-1})^{-1} d\mH^{n-k}
		 \\
		 &\leq & (1+\epsilon') \mathcal{H}^k(G\cdot p)  \lambda^{n-k}(1+\epsilon(1-\lambda)) \int_{\widetilde{\Sigma}\cap Z_p} \theta \mH^{n-k}
		 \\
		 &\leq & (1+\epsilon')^2 \lambda^{n-k}(1+\epsilon(1-\lambda)) \M(T),
	\end{eqnarray*}
	where $Z_p = Z\cap S_p$ is the slice of $G\cdot p$ at $p$ in $Z$, and the second inequality comes from \cite[3.4 (7)]{pitts2014existence} and \cite[Lemma B.5]{li2021min} since ${\rm Exp}^\perp_{G\cdot p}=\exp_p$ or $\widetilde{\exp}_p$ in $Z_p$. 
	Additionally, the first two bullets follow directly from the definition of $T_\lambda$. 
	Because $\epsilon' < 1/3$ and (3), we have 
	\begin{equation}\label{Eq: no epsilon}
		(1+\epsilon' )^2 ({\rm Lip}(E_p ))^{n-k}({\rm Lip}(E_p^{-1}))^{n-k} \leq 2. 
	\end{equation}
	Hence, the inequality in the last bullet follows from (\ref{Eq:weyl-tube})(\ref{Eq:bound-of-theta})\cite[3.4(7)]{pitts2014existence}\cite[Lemma B.5]{li2021min} and a computation similar to the above. 
\end{proof}

Using the above lemma, we can adapt the proof of \cite[Lemma 3.5]{pitts2014existence} (and \cite[Lemma B.8]{li2021min}) straightforwardly and get the following result:
\begin{lemma}\label{L:pre-interpolation-1}
	Given any non-isolated orbit $G\cdot q\subset M$, $\epsilon\in(0,1)$ and $\epsilon'\in (0, 1/3)$, let $Z$, $K$, $E$ be defined as above. 
	Suppose $\delta>0$, $p\in K\cap Z$, $r>0$ with $\B^{G}_r(p)\subset\subset Z$, and $T\in Z^G_n \big(\Clos(\B^{G}_r(p)),\partial \Clos(\B^{G}_r(p));\mathbb{Z}_2 \big)$. 
	Assume additionally: 
	\begin{itemize}
		\item $\epsilon \mathbf{M}(T)< (n-k) \delta \text { and } \frac{2 r}{n-k} \mathbf{M}\left(\partial_1 T\right)+\delta \leq \mathbf{M}(T)$, where $k=\dim(G\cdot q)$;
		\item $\|T\| (\partial_{rel}(\B^{G}_s(p)))=0, \text { for all } 0 \leq s \leq r$.
	\end{itemize}
	Then for any $\beta>0$, there exists a sequence:
	$$T=R_0,R_1,\dots,R_m\in Z^G_n\big( \Clos(\B^{G}_r(p)),\partial \Clos(\B^{G}_r(p));\mathbb{Z}_2 \big)$$
	such that for each $i$, we have: 
	\begin{itemize}
		\item $\partial R_i\llcorner \partial_{rel}(\B^{G}_r(p)) =\partial_1 T$;
		\item ${\bf M}(R_i)\leq (1+\epsilon')^2{\bf M}(T)+\beta$;
		\item ${\bf M}(R_i-R_{i-1})\leq \beta $;
		\item $R_m=E_\# \big( \delta_{\bf 0} ~\ttimes ~ E^{-1}_\#\partial_1 T \big)$;
		\item ${\bf M}(R_m)\leq \frac{2 r}{n-k} \mathbf{M}\left(\partial_1 T\right)\leq {\bf M}(T)-\delta$.
	\end{itemize} 
\end{lemma}
\begin{proof}
	The proof is essentially the same as \cite[Lemma 3.5]{pitts2014existence}, and we only point out some modifications. 
	First, given any $\beta >0$, let $\beta' = \frac{\beta}{2}$. 
	Then we replace the number $\beta$ in \cite[Lemma 3.5]{pitts2014existence} by $\beta'$. 
	Secondly, we use $r_{G\cdot p}$, $\B^G_r(p)$ and $\A_{s,r}^G(p)$ in place of $u_p$, $B_r(p)$ and $An(p,s,r)$ in \cite[Lemma 3.5]{pitts2014existence}. 
	Thirdly, in Part2-5 of \cite[Lemma 3.5]{pitts2014existence}, we use Lemma \ref{Lem:mass estimate under exponential map} in place of \cite[3.4(4)]{pitts2014existence}. 
	Since 
	$(E\circ{\bm \mu}_\lambda\circ E^{-1}) \circ (E\circ{\bm \mu}_\lambda\circ E^{-1}) = E\circ{\bm \mu}_{\lambda^2}\circ E^{-1},$ our mass estimates have just one more term of $(1+\epsilon' )^2$ than those in \cite[Lemma 3.5]{pitts2014existence}. 
	Nevertheless, we could use (\ref{Eq: no epsilon}) and the fact that $(1+ \epsilon'^2 )\beta' < \beta $ to eliminate $\epsilon'$ in some places. 
\end{proof}

We also have the following lemma as a modification of \cite[Lemma B.7]{li2021min}. 
\begin{lemma}\label{L:zero-boundary}
	Suppose $Z, K$ are given as in Lemma \ref{Lem:mass estimate under exponential map}, and $V\in \mathcal{V}^G_n(M)$ is rectifiable in $Z$, then for $\mathcal{H}^{\dim(K)}$-a.e. $p\in Z\cap K$, we have $\|V\|(\partial_{rel}\B^{G}_r(p))=0$ for all $r>0$ with $\B^{G}_r(p)\subset\subset Z$. 
\end{lemma}
\begin{proof}
	For simplicity, we write $\mathcal{S}_r^G(p) := \partial_{rel} \B_r^G(p)$ and $k := \dim(G\cdot p)$ for $p\in Z\cap K$. 
	Note $K=M_i^j$ or $N_i^j$, and $\dim(K) - k \geq 1$ since orbits in $K$ are non-isolated. 
	The proof is separated into two cases depending on whether $Z\cap K$ is contained in $\partial M$: 
	
	{\bf Case 1}: $Z\cap K\subset M\setminus\partial M$. 
	We denote 
	$$ K' := \left\{ \begin{array}{cl}
		\emptyset  & \text{ if } K \subset M^{reg}, \\
		K  & \text{ if } K \subset M\setminus M^{reg}. 
		\end{array} \right. $$
	Thus, $\dim(\mathcal{S}_r^G(p)\cap K') \leq \max\{ 0, n_{i,j} -1 \} \leq n-1$, and $\|V\|(\mathcal{S}_r^G(p)) = \|V\|(\mathcal{S}_r^G(p)\setminus K') $. 
	By the rectifiability of $V$ and \cite[15.3]{simon1983lectures}, we have 
	\begin{eqnarray*}
		\|V\|(\mathcal{S}_r^G(p)) &=& \|V\|(\mathcal{S}_r^G(p)\setminus K') 
		\\
		&=& \|V\|(\{x\in \mathcal{S}_r^G(p)\setminus K' : T_xV\subset T_x(\mathcal{S}_r^G(p)) \}), 
	\end{eqnarray*}
	where $T_xV$ is the approximate tangent space of $V$. 
	Hence, we only need to prove that for $\mathcal{H}^{\dim(K)}$-a.e. $p\in Z\cap K$, 
	$$\|V\|(	\{ x\in Z\setminus K' : T_xV\subset T_x(\mathcal{S}_{r_{G\cdot p}(x)}^G(p))   \}) = 0.$$
	By Fubini's theorem, 
	\begin{eqnarray*}
		&& \int_{Z\cap K} \|V\|(	\{ x\in Z\setminus K' : T_xV\subset T_x(\mathcal{S}_{r_{G\cdot p}(x)}^G(p))   \}) ~d\mathcal{H}^{\dim(K)}(p)
		\\
		&=& \int_{Z\setminus K'} \mathcal{H}^{\dim(K)}(\{ p\in Z\cap K : T_xV\subset T_x(\mathcal{S}_{r_{G\cdot p}(x)}^G(p))  \}) ~d\|V\|(x).
	\end{eqnarray*}
	Now, it is sufficient to show that 
	\begin{equation}\label{Eq: orthogonal}
		\mathcal{H}^{\dim(K)}(\{ p\in Z\cap K : T_xV\bot \nabla r_{G\cdot p}(x)  \}) = 0, \quad \mbox{for $\|V\|$-a.e. $x\in Z\setminus K'$}. 
	\end{equation}
	For fixed $x\in Z\setminus K'$, suppose $p\in Z\cap K$ satisfies $T_xV\bot \nabla r_{G\cdot p}(x)$ and $\dist_M(G\cdot p, x) = \dist_M(p, x)$. 
	Since $T_xV$ is an $n$-dimensional subspace of $T_xM$ and $r_{G\cdot p} = \dist_M(G\cdot p, \cdot )$, all such $p$ lie in a geodesic $\gamma(t)$ starting at $\gamma(0)= x$ with $\dot{\gamma}(0) \in (T_xV)^\perp$. 
	If $K\subset M^{reg}$, then 
	$$\dim(\{ p\in Z\cap K : T_xV\bot \nabla r_{G\cdot p}(x)  \} ) \leq 1 + k < n+1 = \dim(K), ~\mbox{(by Cohom$(G)\geq 3$),}$$
	which implies (\ref{Eq: orthogonal}). 
	If $K\subset M\setminus M^{reg}$. 
	Because $x\notin K$ and ${\rm Exp}_{G\cdot p}^\perp$ is an equivariant diffeomorphism on $Z$, then we have $G_x\subsetneq G_p$. 
	Thus, $p$ is a cut point of $\exp_{G\cdot x}^\perp$ in the direction $\dot{\gamma}(0)$. 
	Since $\dim((T_xV)^\perp)=1$, there are at most two such cut points $p$. 
	Therefore, 
	$$\dim(\{ p\in Z\cap K : T_xV\bot \nabla r_{G\cdot p}(x)  \} ) \leq k < \dim(K),~ \mbox{(by the non-isolated assumption)},$$
	which also implies (\ref{Eq: orthogonal}).
	
	{\bf Case 2}: $Z\cap K\subset \partial M$. 
	Let $Z' := \{x\in Z: n_{\partial M}(x)\in K\}$ be a submanifold in $Z$ with dimension $n_{i,j}' +1$, where $n_{\partial M}$ is the nearest projection to $\partial M$. 
	Denote 
	$$ K' := \left\{ \begin{array}{cl}
		\emptyset  & \text{ if } K \subset (\partial M)^{reg}, \\
		Z'  & \text{ if } K \subset \partial M\setminus (\partial M)^{reg}, 
		\end{array} \right. $$
	which satisfies $\dim(\mathcal{S}_r^G(p)\cap K') \leq \max\{0, n_{i,j}' \} \leq n-1$, and $\|V\|(\mathcal{S}_r^G(p)) = \|V\|(\mathcal{S}_r^G(p)\setminus K') $. 
	By the same arguments as in {\bf Case 1}, we only need to verify (\ref{Eq: orthogonal}) for such $K$ and $K'$. 
	
	Given $x\in Z\setminus K'$, suppose $p\in Z\cap K$ satisfies $T_xV\bot \nabla r_{G\cdot p}(x)$ and $\dist_{\partial M}(G\cdot p, n_{\partial M}(x)) = \dist_{\partial M}(p, n_{\partial M}(x))$. 
	Noting $T_xV$ is an $n$-dimensional subspace of $T_x\widetilde{M}$ and $r_{G\cdot p} = \tilde{r}_{G\cdot p}$, we can apply the arguments in \cite[Lemma B.7 Claim]{li2021min} to see $\nabla^{\partial M} r_{G\cdot p}(n_{\partial M}(x) )$ lies in an affine subspace of $T_{n_{\partial M}(x)}\partial M$ with dimension at most $n+1 -n =1$. 
	Since $\tilde{r}_{G\cdot p}\llcorner\partial M = \dist_{\partial M}(G\cdot p, \cdot )$, we also have $|\nabla^{\partial M} r_{G\cdot p}(n_{\partial M}(x) )| = 1 $. 
	Therefore, there are at most two possible choices $\{v_1, v_2\}$ for $\nabla^{\partial M} r_{G\cdot p}(n_{\partial M}(x) )$. 
	If $K\subset (\partial M)^{reg}$, then all such $p$ lie in at most two geodesic curves of $\partial M$. 
	Thus, (\ref{Eq: orthogonal}) follows from 
	$$\dim(\{ p\in Z\cap K : T_xV\bot \nabla r_{G\cdot p}(x)  \} ) \leq 1 + k < n = \dim(K),  ~\mbox{(by Cohom$(G)\geq 3$)}.$$
	If $K\subset \partial M \setminus (\partial M)^{reg}$, we have $G_{n_{\partial M}(x)}\subsetneq G_p$ since $n_{\partial M}(x)\notin K$ and ${\rm Exp}_{G\cdot p}^\perp$ is an equivariant diffeomorphism on $Z$. 
	Again, such $p$ is a cut point of $\exp_{G\cdot n_{\partial M}(x)}^{\partial M, \perp}$ in the direction $-v_1$ or $-v_2$. 
	Hence, there are at most two such points $p$, which implies 
	$$\dim(\{ p\in Z\cap K : T_xV\bot \nabla r_{G\cdot p}(x)  \} ) \leq  k .$$
	Since every orbit in $K$ is non-isolated, (\ref{Eq: orthogonal}) follows from $k<k+1\leq \dim(K)$. 
\end{proof}

Next, we have the following pre-interpolation lemma, which generalizes \cite[Lemma B.3]{li2021min}. 

\begin{lemma}\label{L:pre-interpolation}
	Suppose $L>0$, $\delta>0$, $\tau\in \mathcal{Z}^G_n(M,\partial M;\mathbb{Z}_2)$, $W\subset\subset U$ is a relative open $G$-set containing no isolated orbit. 
	Given $\sigma_0\in \mathcal{Z}^G_n(M,\partial M;\mathbb{Z}_2)$ with 
	$${\rm spt}(\sigma_0-\tau)\subset W\quad {\rm and}\quad {\bf M}(\sigma_0)\leq L, $$
	then there exists $\epsilon=\epsilon(L,\delta,W,\tau,\sigma_0)>0$ such that if $\sigma\in \mathcal{Z}^G_n(M, \partial M;\mathbb{Z}_2)$ satisfies 
	\begin{itemize}
		\item ${\rm spt}(\sigma-\tau)\subset W$,
		\item ${\bf M}(\sigma)\leq L$,
		\item $\mathcal{F}(\sigma-\sigma_0)\leq\epsilon$,
	\end{itemize}
	there exists a sequence $\sigma=\tau_0,\tau_1,\dots,\tau_m=\sigma_0\in \mathcal{Z}^G_n(M,\partial M;\mathbb{Z}_2)$ such that for each $j$, we have 
	\begin{itemize}
		\item[(i)] ${\rm spt}(\tau_j-\tau)\subset U$,
		\item[(ii)] ${\bf M}(\tau_j)\leq L+\delta$,
		\item[(iii)] ${\bf M}(\tau_j-\tau_{j-1})\leq\delta$.
	\end{itemize}
\end{lemma}
\begin{proof}
	Suppose the lemma is false, then there exists a sequence $\{\sigma_j\}_{j=1}^\infty\subset \mathcal{Z}^G_n(M,\partial M;\mathbb{Z}_2)$ so that 
	$${\rm spt}(\sigma_j-\tau)\subset W,\quad {\bf M}(\sigma_j)\leq L,\quad \mathcal{F}(\sigma_j-\sigma_0)\to 0, $$
	but none of the $\sigma_j$ can be connected to $\sigma_0$ by a finite sequence in $\mathcal{Z}^G_n(M,\partial M;\mathbb{Z}_2)$ satisfying (i)-(iii). 
	Denote by $S_j$ the canonical representative of $\sigma_j$. 
	Noting $\M(S_j)=\M(\sigma_j)\leq L$, we can suppose $|S_j|\to V\in\mathcal{V}^G_n(M)$ up to a subsequence, and hence 
	\begin{equation}\label{Eq:Appendix 1}
		\|S_0\|(A)\leq \|V\|(A),
	\end{equation}
	for any Borel $A\subset M$ by \cite[Lemma 3.7]{li2021min}. 
	
	Let $\alpha = \delta/5$. 
	The proof of this lemma is now divided into two cases as in \cite[Page 49]{li2021min} depending on whether there exists any orbit $G\cdot q$ in $W$ with $\|V\|(G\cdot q)$ lager than $\alpha$:
	\begin{itemize}
		\item {\bf Case 1}: $\|V\|(G\cdot q)\leq \alpha$ for all $q\in W$;
		\item {\bf Case 2}: $\{q\in W : \|V\|(G\cdot q)>\alpha\}\neq\emptyset$.
	\end{itemize}

	 For {\bf Case 1}, there are finite geodesic $G$-tubes $\{\tBcal^G_{r_i}(p_i)\}_{i=1}^m$ in $\widetilde{M}$ with $p_i\in\spt(\|V\|)\cap W$ so that 
	\begin{itemize}
		\item $\Clos(\tBcal_{r_i}^G(p_i)) \cap M \subset U $, and $\Clos(\tBcal_{r_i}^G(p_i)) \cap \Clos(\tBcal_{r_k}^G(p_k))\cap M=\emptyset$ for all $i\neq k\in\{1,\dots, m\}$; 
		\item $\|V\| (\tScal_{r_i}^G(p_i) \cap M)=0$;
		\item $\|V\| (\Clos(\tBcal_{r_i}^G(p_i)))  <2\al$;
		\item $\|V\| (W \setminus \cup_{i=1}^m \Clos(\tBcal_{r_i}^G(p_i)) ) <2\al$.
	\end{itemize}
	By (\ref{Eq:Appendix 1}), the weight measure of $S_0$ also satisfies all the above inequalities.
	
	Noting $\F(\si_j-\si_0)\to 0$, the $G$-invariant $\F$-isoperimetric Lemma \ref{Lem:F-isoperimetric} gives an integral $G$-current $Q_j\in\bI_{n+1}^G(M;\mZ_2)$ for every $j$ large enough satisfying:
	\begin{itemize}
		\item $\spt(S_j-S_0-\partial Q_j)\subset \partial M$;
		\item $\lim_{j\rightarrow\infty}\M(Q_j)=0$.
	\end{itemize}
	Note the distance function $d^G_p=\dist_{\tM}(G\cdot p,\cdot)$ to an orbit is a $G$-invariant function. 
	Thus, the slice of $Q_j$ by $d^G_p$ is also a $G$-invariant current. 
	Hence, we can find a sequence $r_i^j\searrow r_i$ as $j\to\infty$ for every $i\in\{1,\dots,m\}$ so that all the eight statements in the middle of \cite[Page 49]{li2021min} hold with $\tBcal_{r_i^j}^G(p_i),\tScal_{r_i^j}^G(p_i) $ in place of $\tBcal_{r_i^j}(p_i),\tScal_{r_i^j}(p_i)$ respectively. 
	Then we can construct a sequence of $G$-currents $R_i^j\in Z^G_n(M,\partial M;\mZ_2)$ by 
	\begin{eqnarray}\label{Eq: pre-interpolation}
		&&R_0^j  :=  S_j,\nonumber
		\\
		&&R_i^j   :=  S_j - \sum_{s=1}^i \left( \partial [Q_j\lc \Clos(\tBcal_{r_s^j}^G(p_s))]-(\partial Q_j) \lc [ \Clos(\tBcal_{r_s^j}^G(p_s)) \cap \partial M] \right),
		\\
		 &&R_{m+1}^j  :=  S_0. \nonumber
	\end{eqnarray}
	One can follow the process in \cite[Page 50]{li2021min} to show that the sequence $\tau_i^j=[R_i^j]\in\Z_n^G(M,\partial M;\mZ_2)$ satisfies all the requirements (i)-(iii), which is a contradiction.

	As for {\bf Case 2}, it is clear that $ \{q\in W~:~ \|V\|(G\cdot q)>\alpha\}$ can only contain a finite number of orbits since $\|V\|\leq L$. 
	First, let us assume that $\{q\in W~:~ \|V\|(G\cdot q)>\alpha\}=\{G\cdot q\}$.

	Denote $k=\dim(G\cdot q)$, and denote ${\rm Exp}_{G\cdot p}^\perp$, $r_{G\cdot p}$, $\B^G_r(p)$, $\A^G_{s,t}$ as in (\ref{Eq: notations}), Definition \ref{Def:annulus}. 
	Take $\epsilon\in (0,1)$ and $\epsilon'\in (0,1/3)$ such that 
	$$ 2\epsilon \|V\|(U)\leq (n-k)\frac{\alpha}{2} \quad \mbox{   and    }\quad 2(2\epsilon'+\epsilon'^2)\|V\|(U)\leq \frac{\alpha}{64}.$$ 
	Since $G\cdot q$ is non-isolated, let $K$ be given as in (\ref{Eq: stratum}), and $Z$ be the open $G$-neighborhood of $q$ as in Lemma \ref{L:pre-interpolation-1} corresponding to $\epsilon,\epsilon'$. 
	By \cite[2.1(14b)]{pitts2014existence}, we have 
	\begin{eqnarray*}
		\lim_{j\to \infty}\|S_j\|\llcorner {\rm Clos}(\A^G_{s,r}(p)) = \|V\|\llcorner {\rm Clos}(\A^G_{s,r}(p)), 
		\quad 
		\lim_{r\to 0}\|V\|({\rm Clos}(\B^{G}_r(q))\setminus G\cdot q) = 0,
	\end{eqnarray*}
	whenever $p\in Z\cap K$, $0<s<r$, and $\|V\|(\partial \A^G_{s,r}(p))=0$. 
	By a similar argument as in \cite[Page 119]{pitts2014existence}, one verifies that there exists a positive integer $J$ with the following properties:
	
	For each $j\geq J$, there exists $p_j\in Z\cap K $, $0<\frac{s_j}{2}<r_j<s_j$, such that 
	\begin{itemize}
		\item[(a)] $p_j\to q$, $s_j\to 0$, as $j\to\infty$;
		\item[(b)] $\B^{G}_{r_j/4}(q)\subset \B^{G}_{r_j/2}(p_j)\subset \B^{G}_{2s_j}(p_j)\subset Z$, where $\B^G_r(p)$ and $r_{G\cdot p}$ are defined in (\ref{Eq: notations});
		\item[(c)] $\|S_j\|(\partial_{rel} \B^{G}_s(p_j))=0,~\forall 0<s\leq r_j$;
		\item[(d)] $\langle S_j, r_{G\cdot p_j}, r_j\rangle $ is a rectifiable $G$-invariant current with $\mZ_2$-coefficients; 
		\item[(e)] $\lim_{j\to\infty} \|S_j\|({\rm Clos}(\A_{s_j/2,2s_j}^G(p_j)))=0$;
		\item[(f)] $16  \|S_j\|({\rm Clos}( \A_{s_j/2,2s_j}^G(p_j))) \geq 3r_j {\bf M}(\langle S_j, r_{G\cdot p_j},r_j\rangle)$;
		\item[(g)] $\epsilon \|S_j\|(U)\leq (n-k)\frac{\alpha}{2}$, and $(2\epsilon'+\epsilon'^2)\|S_j\|(U)\leq \frac{\alpha}{64}$;
		\item[(h)] $\frac{2r_j}{n-k}{\bf M}(\langle S_j, r_{G\cdot p_j},r_j\rangle) + \frac{\alpha}{2} \leq \|S_j\|({\rm Clos}(\B^{G}_{r_j}(p_j)))$;
		\item[(i)] $\lim_{j\to\infty} |S_j|\llcorner (M\setminus {\rm Clos}(\B^{G}_{r_j}(p_j))) = V\llcorner (M\setminus (G\cdot q))$.
	\end{itemize}
	We mention that (c) comes from Lemma \ref{L:zero-boundary}, (g) comes from the choice of $\epsilon$ as well as $\epsilon'$, and (h) comes from (e) and (f). 
	Define then 
	$$\widetilde{S}_j := S_j\llcorner(M\setminus {\rm Clos}(\B^{G}_{r_j}(p_j)) ) + ({\rm  Exp}^\perp_{G\cdot p_j})_\#(\delta_{\bf 0} \ttimes ({\rm  Exp}^\perp_{G\cdot p_j})^{-1}_\# \langle S_j, r_{G\cdot p_j},r_j\rangle).$$
	Noting $(2\epsilon'+\epsilon'^2)\|S_j\|(U)\leq \frac{\delta}{2}$ by (g), and 
	$$\partial_1( S_j\llcorner\Clos(\B^{G}_{r_j}(p_j)  ) ) = \langle S_j, r_{G\cdot p_j}, r_j\rangle,$$
	we can apply Lemma \ref{L:pre-interpolation-1} with $r=r_j$, $\delta=\frac{\alpha}{2}$, $p=p_j$, $\beta=\frac{\delta}{2}$, $T=S_j\llcorner \B^{G}_{r_j}(p_j)$ to obtain a finite sequence $\{R_i^j\}_{i=0}^{m_j}\subset Z_n^G(M,\partial M;\mZ_2)$ connecting $S_j$ and $\widetilde{S}_j$ so that the sequence $\{[R_i^j]\}_{i=0}^{m_j}$ satisfies (i)-(iii). 
	Moreover, we have 
	\begin{itemize}
		\item ${\bf M}(\widetilde{S}_j )\leq {\bf M}(S_j)-\frac{\alpha}{2}$ by the last inequality in Lemma \ref{L:pre-interpolation-1},
		\item $\M( ({\rm  Exp}^\perp_{G\cdot p_j})_\#(\delta_{\bf 0} \ttimes ({\rm  Exp}^\perp_{G\cdot p_j})^{-1}_\# \langle S_j, r_{G\cdot p_j}, r_j\rangle) )     \leq   C r_j {\bf M}(\langle S_j, r_{G\cdot p_j}, r_j\rangle) \to 0$, by Lemma \ref{Lem:mass estimate under exponential map}, (f), and (e),
		\item thus, $\lim_{j\to\infty}\widetilde{S}_j=\lim_{j\to\infty}S_j\in \sigma_0$ as currents,
		\item $\lim_{j\to\infty}|\widetilde{S}_j|=V\llcorner (M\setminus{G\cdot q})$.
	\end{itemize}
	Now, the procedure in {\bf Case 1} gives a finite sequence connecting $[\widetilde{S}_j]$ to $[S_0]=\sigma_0$ so that (i)-(iii) are valid for $j$ large enough, which is a contradiction. 
		
	As for the case that $\{q\in W : \|V\|(G\cdot q)>\alpha\}$ contains more than one orbit, we notice that ${\bf M}(R_{m_j}^j)\leq {\bf M}(S_j)-\frac{\alpha}{2}$ and $\lim_{j\to\infty} |R_{m_j}^j|=V\llcorner(M\setminus{G\cdot q})$. 
	Hence, one can carry out an induction argument to get a contradiction. 
\end{proof}

Combining Lemma \ref{L:pre-interpolation} with a covering argument as in \cite[Page 124]{pitts2014existence}, we have the following interpolation lemma:
\begin{lemma}\label{L:interpolation}
	Suppose $L>0$, $\delta>0$, $W\subset\subset U$ is an open $G$-set containing no isolated orbit, and $\tau\in \mathcal{Z}^G_n(M,\partial M;\mathbb{Z}_2)$. 
	Then there exists $\epsilon=\epsilon(L,\delta,W,\tau)>0$ such that if $\sigma_1,\sigma_2\in \mathcal{Z}^G_n(M,\partial M;\mathbb{Z}_2)$ satisfies 
	\begin{itemize}
		\item ${\rm spt}(\sigma_i-\tau)\subset W$, for $i=1,2$;
		\item ${\bf M}(\si_i)\leq L$, for $i=1,2$;
		\item $\mathcal{F}(\si_1-\si_2)\leq\epsilon$,
	\end{itemize}
	there exists a sequence $\si_1=\tau_0,\tau_1,\dots,\tau_m=\si_2\in \mathcal{Z}^G_n(M, \partial M;\mathbb{Z}_2)$ such that for each $j$: 
	\begin{itemize}
		\item[(i)] ${\rm spt}(\tau_j-\tau)\subset U$,
		\item[(ii)] ${\bf M}(\tau_j)\leq L+\delta$,
		\item[(iii)] ${\bf M}(\tau_j-\tau_{j-1})\leq\delta$.
	\end{itemize}
\end{lemma}

The interpolation lemma \ref{L:interpolation} implies that if two equivalent classes of $G$-invariant relative $n$-cycles are close enough in the flat metric and only different in a relative open $G$-set containing no isolated orbit, then there exists a finite interpolation sequence of equivalent classes of $G$-invariant relative $n$-cycles with small ${\bf M}$-fineness, small increasing of mass, and small change of supports. 

Finally, combining Lemma \ref{L:interpolation} with the arguments in \cite[Theorem 3.10]{pitts2014existence}(\cite[Appendix B]{li2021min}), one can easily verify that $(c)\Rightarrow (d)$ in Theorem \ref{Thm: equivalence-a.m.v}.

\bibliographystyle{abbrv}

\providecommand{\bysame}{\leavevmode\hbox to3em{\hrulefill}\thinspace}
\providecommand{\MR}{\relax\ifhmode\unskip\space\fi MR }
\providecommand{\MRhref}[2]{%
  \href{http://www.ams.org/mathscinet-getitem?mr=#1}{#2}}
\providecommand{\href}[2]{#2}

\bibliography{reference.bib}   

\end{document}